\documentclass[nosumlimits,twoside]{amsart}
\usepackage{amsfonts, amsmath, amssymb}
\usepackage[bookmarksnumbered,plainpages,hypertex]{hyperref}
\usepackage{graphicx}
\usepackage{float}
\usepackage{srcltx}
\usepackage[all]{xy}
\usepackage{comment}

\newtheorem{theorem}{\sc Theorem}[section]
\newtheorem{proposition}[theorem]{\sc Proposition}
\newtheorem{notation}[theorem]{\sc Notation}

\newtheorem{lemma}[theorem]{\sc Lemma}

\theoremstyle{definition}
\newtheorem{definition}[theorem]{\sc Definition}

\theoremstyle{remark}
\newtheorem{remark}[theorem]{\sc Remark}

\newtheorem{claim}[theorem]{}

\allowdisplaybreaks
\IfFileExists{tcilatex.tex}{\input{tcilatex}}{}
 \allowdisplaybreaks
 \excludecomment{proof?}

\begin{document}
\title{Adjunctions and Braided Objects}
\author{Alessandro Ardizzoni}
\address{University of Turin, Department of Mathematics "G. Peano", via
Carlo Alberto 10, I-10123 Torino, Italy}
\email{alessandro.ardizzoni@unito.it}
\urladdr{www.unito.it/persone/alessandro.ardizzoni}
\author{Claudia Menini}
\address{University of Ferrara, Department of Mathematics and Computer Science, Via Machiavelli
35, Ferrara, I-44121, Italy}
\email{men@unife.it}
\urladdr{http://www.unife.it/utenti/claudia.menini}
\thanks{This paper was written while both authors were members of GNSAGA.
The first author was partially supported by the research grant ``Progetti di
Eccellenza 2011/2012'' from the ``Fondazione Cassa di Risparmio di Padova e
Rovigo''.}

\begin{abstract}
In this paper we investigate the categories of braided objects, algebras and bialgebras in a given monoidal category, some pairs of adjoint functors between them and their relations. In particular we construct a braided primitive functor and its left adjoint, the braided tensor bialgebra functor, from the category of braided objects to the one of braided bialgebras. The latter is obtained by a specific elaborated construction introducing a braided tensor algebra functor as a left adjoint of the forgetful functor from the category of braided algebras to the one of braided objects.
The behaviour of these functors in the case when the base category is braided is also considered.
\end{abstract}

\maketitle
\tableofcontents

\section*{Introduction}

Let $B$ be a braided bialgebra over a field $\Bbbk$. This means that $B$ is both an algebra and a coalgebra and these structures are suitably compatible with a braiding $c:B\otimes B\rightarrow B\otimes B$ of $B$. It is well-known that $c$ induces a braiding $c_P$ on the space $P=P(B)$ of primitive elements of $B$, see e.g. \cite[page 4]{Kh} in the connected case. It is natural to wonder whether this result remains true for braided bialgebras in a monoidal category $\mathcal{M}$. Note that $\mathcal{M}$ needs not to be braided, a priori exactly as the above braiding $c$ needs not to be the evaluation of a braiding defined on the whole category of vector spaces. On the other hand it is also well-known that, under mild assumptions, the forgetful functor $\Omega :\mathrm{Alg}_{\mathcal{M}}\rightarrow \mathcal{M}$ from the category of algebras into $\mathcal{M}$ has a left adjoint $T:\mathcal{M}\rightarrow \mathrm{Alg}_{\mathcal{M}}$ given by the tensor algebra functor, see Remark \ref{cl: AlgMon}.

In this paper we prove that, under mild assumptions, the forgetful functor $\Omega_\mathrm{Br}:\mathrm{BrAlg}_{\mathcal{M}}\rightarrow \mathrm{Br}_{\mathcal{M}}$ from the category $\mathrm{BrAlg}_{\mathcal{M}}$ of braided algebras in $\mathcal{M}$ to the category $\mathrm{Br}_{\mathcal{M}}$ of braided objects in $\mathcal{M}$ has a left adjoint $T_\mathrm{Br}$, see Proposition \ref{pro:TbrStrict}, which is induced by $T$. This is achieved by a rather technical tool which makes use of suitable morphisms $c_{T}^{m,n}$ constructed in Proposition \ref{pro:CT} by means of Lemma \ref{lem:Kassel}, where the Braid Category plays a central role. We also introduce a braided primitive functor $P_\mathrm{Br}:\mathrm{BrBialg}_{\mathcal{M}}\rightarrow \mathrm{Br}_{\mathcal{M}}$ where $\mathrm{BrBialg}_{\mathcal{M}}$ denotes the category of braided bialgebras in $\mathcal{M}$, see Lemma \ref{lem:primitive}. We prove that this functor $P_\mathrm{Br}$ has also a left adjoint, namely the functor $\overline{T}_\mathrm{Br}$ which is induced by the functor $T_\mathrm{Br}$.

Another problem is to investigate the case when the monoidal category $\mathcal{M}$ is braided.
In this case one can also consider  the category $\mathrm{Bialg}_{\mathcal{M}}$ of bialgebras  in $\mathcal{M}$. Moreover, the categories $\mathcal{M}$, $\mathrm{Alg}_{\mathcal{M}}$ and $\mathrm{Bialg}_{\mathcal{M}}$ are related to their braided analogues by means of the functors \begin{equation*}
J:\mathcal{M}\rightarrow \mathrm{Br}_{\mathcal{M}},\quad J_{\mathrm{Alg}}:%
\mathrm{Alg}_{\mathcal{M}}\rightarrow \mathrm{BrAlg}_{\mathcal{M}}\quad
\text{and}\quad J_{\mathrm{Bialg}}:\mathrm{Bialg}_{\mathcal{M}}\rightarrow
\mathrm{BrBialg}_{\mathcal{M}},
\end{equation*}
see Proposition \ref{coro:BrBialg}. Using these functors we investigate the relation between $T$ and $T_\mathrm{Br}$ (Proposition \ref{pro:JAlg-J}). We also show that $P_\mathrm{Br}$ gives rise to a primitive functor $P:\mathrm{Bialg}_{\mathcal{M}}\rightarrow \mathcal{M}$, see Proposition \ref{pro:J}. We prove that even this functor $P$ has a left adjoint $\overline{T}$ (Theorem \ref{teo:Tbar}) which is related to the functor $\overline{T}_\mathrm{Br}$ as in \eqref{form:JPbar}.

In this paper we also investigate the behaviour of the primitive functors mentioned above when the base category changes through an arbitrary monoidal functors $F:\mathcal{M}\rightarrow \mathcal{M}^\prime$. This is done in Proposition \ref{pro:BrBialg}, Proposition \ref{pro:PrimFunct} and Proposition \ref{pro:BialgF}.\medskip\newline
The adjunctions considered above will be studied in connection with monadic decomposition of functors in a forthcoming paper where the particular cases when $\mathcal{M}$ or $\mathcal{M}^\prime$ are the category of vector spaces or the category of (co)modules over a not necessarily finite-dimensional (dual) quasi-bialgebra will be investigated.

\section{Preliminaries}

\label{preliminares}

In this section, we shall fix some basic notation and terminology.

\begin{notation}
Throughout this paper $\Bbbk $ will denote a field. All vector spaces will
be defined over $\Bbbk $. The unadorned tensor product $\otimes $ will
denote the tensor product over $\Bbbk $ if not stated otherwise.
\end{notation}

\begin{claim}
\textbf{Monoidal Categories.} Recall that (see \cite[Chap. XI]{Kassel}) a
\emph{monoidal category}\textbf{\ }is a category $\mathcal{M}$ endowed with
an object $\mathbf{1}\in \mathcal{M}$ (called \emph{unit}), a functor $%
\otimes :\mathcal{M}\times \mathcal{M}\rightarrow \mathcal{M}$ (called \emph{%
tensor product}), and functorial isomorphisms $a_{X,Y,Z}:(X\otimes Y)\otimes
Z\rightarrow $ $X\otimes (Y\otimes Z)$, $l_{X}:\mathbf{1}\otimes
X\rightarrow X,$ $r_{X}:X\otimes \mathbf{1}\rightarrow X,$ for every $X,Y,Z$
in $\mathcal{M}$. The functorial morphism $a$ is called the \emph{%
associativity constraint }and\emph{\ }satisfies the \emph{Pentagon Axiom, }%
that is the equality
\begin{equation*}
(U\otimes a_{V,W,X})\circ a_{U,V\otimes W,X}\circ (a_{U,V,W}\otimes
X)=a_{U,V,W\otimes X}\circ a_{U\otimes V,W,X}
\end{equation*}%
holds true, for every $U,V,W,X$ in $\mathcal{M}.$ The morphisms $l$ and $r$
are called the \emph{unit constraints} and they obey the \emph{Triangle
Axiom, }that is $(V\otimes l_{W})\circ a_{V,\mathbf{1},W}=r_{V}\otimes W$,
for every $V,W$ in $\mathcal{M}$.

A \emph{monoidal functor}\label{MonFun} (also called strong monoidal in the
literature)
\begin{equation*}
(F,\phi _{0},\phi _{2}):(\mathcal{M},\otimes ,\mathbf{1},a,l,r\mathbf{%
)\rightarrow (}\mathcal{M}^{\prime }\mathfrak{,}\otimes ^{\prime },\mathbf{1}%
^{\prime },a^{\prime },l^{\prime },r^{\prime }\mathbf{)}
\end{equation*}%
between two monoidal categories consists of a functor $F:\mathcal{M}%
\rightarrow \mathcal{M}^{\prime },$ an isomorphism $\phi
_{2}(U,V):F(U)\otimes ^{\prime }F(V)\rightarrow F(U\otimes V),$ natural in $%
U,V\in \mathcal{M}$, and an isomorphism $\phi _{0}:\mathbf{1}^{\prime
}\rightarrow F(\mathbf{1})$ such that the diagram
\begin{equation*}
\xymatrixcolsep{63pt}\xymatrixrowsep{35pt}\xymatrix{ (F(U)\otimes'
F(V))\otimes' F(W) \ar[d]|{a'_{F(U),F(V),F(W)}} \ar[r]^-{\phi_2(U,V)\otimes'
F(W)} & F(U\otimes V)\otimes' F(W) \ar[r]^{\phi_2(U\otimes V,W)} &
F((U\otimes V)\otimes W) \ar[d]|{F(a_{ U,V, W})} \\ F(U)\otimes'
(F(V)\otimes' F(W)) \ar[r]^-{F(U)\otimes' \phi_2(V,W)} & F(U)\otimes'
F(V\otimes W) \ar[r]^{\phi_2(U,V\otimes W)} & F(U\otimes (V\otimes W)) }
\end{equation*}%
is commutative, and the following conditions are satisfied:
\begin{equation*}
{F(l_{U})}\circ {\phi _{2}(\mathbf{1},U)}\circ ({\phi _{0}\otimes }^{\prime }%
{F(U)})={l}^{\prime }{_{F(U)}},\text{\quad }{F(r_{U})}\circ {\phi _{2}(U,%
\mathbf{1})}\circ ({F(U)\otimes }^{\prime }{\phi _{0}})={r}^{\prime }{%
_{F(U)}.}
\end{equation*}%
The monoidal functor is called \emph{strict }if the isomorphisms $\phi
_{0},\phi _{2}$ are identities of $\mathcal{M}^{\prime }$.
\end{claim}

The notions of algebra, module over an algebra, coalgebra and comodule over
a coalgebra can be introduced in the general setting of monoidal categories.

From now on we will omit the associativity and unity constraints unless
needed to clarify the context.\medskip

Let $V$ be an object in a monoidal category $\left( \mathcal{M},\otimes ,%
\mathbf{1}\right) $. Define iteratively $V^{\otimes n}$ for all $n\in
\mathbb{N}
$ by setting $V^{\otimes 0}:=\mathbf{1}$ for $n=0$ and $V^{\otimes
n}:=V^{\otimes \left( n-1\right) }\otimes V$ for $n>0.$

\begin{remark}
\label{cl: AlgMon} Let $\mathcal{M}$ be a monoidal category. Denote by $%
\mathrm{Alg}_{\mathcal{M}}$ the category of algebras in $\mathcal{M}$ and
their morphisms. Assume that $\mathcal{M}$ has denumerable coproducts and
that the tensor products (i.e. $M\otimes \left( -\right) :\mathcal{M}%
\rightarrow \mathcal{M}$ and $\left( -\right) \otimes M:\mathcal{M}%
\rightarrow \mathcal{M}$, for every object $M$ in $\mathcal{M}$) preserve
such coproducts. By \cite[Theorem 2, page 172]{MacLane}, the forgetful
functor
\begin{equation*}
\Omega :\mathrm{Alg}_{\mathcal{M}}\rightarrow \mathcal{M}
\end{equation*}%
has a left adjoint $T:\mathcal{M}\rightarrow \mathrm{Alg}_{\mathcal{M}}$. By
construction $\Omega TV=\oplus _{n\in
\mathbb{N}
}V^{\otimes n}$ for every $V\in \mathcal{M}$. For every $n\in \mathbb{N},$
we will denote by
\begin{equation*}
\alpha _{n}V:V^{\otimes n}\rightarrow \Omega TV
\end{equation*}
the canonical injection. Given a morphism $f:V\rightarrow W$ in $\mathcal{M}$%
, we have that $Tf$ is uniquely determined by the following equality%
\begin{equation}
\Omega Tf\circ \alpha _{n}V=\alpha _{n}W\circ f^{\otimes n},\text{ for every
}n\in \mathbb{N}\text{.}  \label{form:Tf}
\end{equation}%
The multiplication $m_{\Omega TV}$ and the unit $u_{\Omega TV}$ are uniquely
determined by
\begin{eqnarray}
m_{\Omega TV}\circ \left( \alpha _{m}V\otimes \alpha _{n}V\right) &=&\alpha
_{m+n}V,\text{ for every }m,n\in \mathbb{N}\text{,}  \label{form:TVm} \\
u_{\Omega TV} &=&\alpha _{0}V.  \label{form:TVu}
\end{eqnarray}
\end{remark}

\begin{remark}
The unit $\eta $ and the counit $\epsilon $ of the adjunction $\left(
T,\Omega \right) $ are uniquely determined, for all $V\in \mathcal{M}$ and $%
\left( A,m_{A},u_{A}\right) \in \mathrm{Alg}_{\mathcal{M}}$ by the following
equalities%
\begin{equation}
\eta V:=\alpha _{1}V\qquad \text{and}\qquad \Omega \epsilon \left(
A,m_{A},u_{A}\right) \circ \alpha _{n}A:=m_{A}^{n-1}\text{ for every }n\in
\mathbb{N}  \label{form:etaeps}
\end{equation}%
where $m_{A}^{n-1}:A^{\otimes n}\rightarrow A$ is the iterated
multiplication of $A$ defined by $m_{A}^{-1}:=u_{A},m_{A}^{0}:=\mathrm{Id}%
_{A}$ and, for $n\geq 2,$ $m_{A}^{n-1}=m_{A}(m_{A}^{n-2}\otimes A).$
\end{remark}

\section{Braided objects}

\begin{definition}
Let $\left( \mathcal{M},\otimes ,\mathbf{1}\right) $ be a monoidal category
(as usual we omit the brackets although we are not assuming the constraints
are trivial).

1) Let $V$ be an object in $\mathcal{M}$. A morphism $c=c_{V}:V\otimes
V\rightarrow V\otimes V$ is called a \emph{braiding }(see \cite[Definition
XIII.3.1]{Kassel} where it is called a Yang-Baxter operator) if it satisfies
the quantum Yang-Baxter equation%
\begin{equation}
\left( c\otimes V\right) \left( V\otimes c\right) \left( c\otimes V\right)
=\left( V\otimes c\right) \left( c\otimes V\right) \left( V\otimes c\right)
\label{ec: braided equation}
\end{equation}%
on $V\otimes V\otimes V.$ \textbf{We further assume that }$c$\textbf{\ is
invertible}. The pair $\left( V,c\right) $ will be called a \emph{braided
object in }$\mathcal{M}$. A morphism of braided objects $(V,c_{V})$ and $%
(W,c_{W})$ in $\mathcal{M}$ is a morphism $f:V\rightarrow W$ such that $%
c_{W}(f\otimes f)=(f\otimes f)c_{V}.$ This defines the category $\mathrm{Br}%
_{\mathcal{M}}$ of braided objects and their morphisms.

2) \cite{Ba} A quadruple $(A,m,u,c)$ is called a \emph{braided algebra} if

\begin{itemize}
\item $(A,m,u)$ is an algebra in $\mathcal{M}$;

\item $(A,c)$ is a braided object in $\mathcal{M}$;

\item $m$ and $u$ commute with $c$, that is the following conditions hold:
\begin{gather}
c(m\otimes A)=(A\otimes m)(c\otimes A)(A\otimes c),  \label{Br2} \\
c(A\otimes m)=(m\otimes A)\left( A\otimes c\right) (c\otimes A),  \label{Br3}
\\
c(u\otimes A)l_{A}^{-1}=\left( A\otimes u\right) r_{A}^{-1},\qquad
c(A\otimes u)r_{A}^{-1}=\left( u\otimes A\right) l_{A}^{-1}.  \label{Br4}
\end{gather}
\end{itemize}

A morphism of braided algebras is, by definition, a morphism of algebras
which, in addition, is a morphism of braided objects. This defines the
category $\mathrm{BrAlg}_{\mathcal{M}}$ of braided algebras and their
morphisms.

3) A quadruple $(C,\Delta ,\varepsilon ,c)$ is called a \emph{braided
coalgebra} if

\begin{itemize}
\item $(C,\Delta ,\varepsilon )$ is a coalgebra in $\mathcal{M}$;

\item $(C,c)$ is a braided object in $\mathcal{M}$;

\item $\Delta $ and $\varepsilon $ commute with $c$, that is the following
relations hold:
\begin{gather}
(\Delta \otimes C)c=(C\otimes c)(c\otimes C)(C\otimes \Delta ),  \label{Br5}
\\
(C\otimes \Delta )c=(c\otimes C)(C\otimes c)(\Delta \otimes C),  \label{Br6}
\\
l_{C}(\varepsilon \otimes C)c=r_{C}\left( C\otimes \varepsilon \right)
,\qquad r_{C}(C\otimes \varepsilon )c=l_{C}\left( \varepsilon \otimes
C\right) .  \label{Br7}
\end{gather}
\end{itemize}

A morphism of braided coalgebras is, by definition, a morphism of coalgebras
which, in addition, is a morphism of braided objects. This defines the
category $\mathrm{BrCoalg}_{\mathcal{M}}$ of braided coalgebras and their
morphisms.

4) \cite[Definition 5.1]{Ta} A sextuple $(B,m,u,\Delta ,\varepsilon ,c)$ is
a called a \emph{braided bialgebra} if

\begin{itemize}
\item $(B,m,u,c)$ is a braided algebra;

\item $(B,\Delta ,\varepsilon ,c)$ is a braided coalgebra;

\item the following relations hold:%
\begin{eqnarray}
\Delta m &=&(m\otimes m)(B\otimes c\otimes B)(\Delta \otimes \Delta ).
\label{Br1} \\
\Delta u &=&(u\otimes u)\Delta _{\mathbf{1}},  \label{Br8} \\
\varepsilon m &=&m_{\mathbf{1}}\left( \varepsilon \otimes \varepsilon
\right) ,  \label{Br9} \\
\varepsilon u &=&\mathrm{Id}_{\mathbf{1}}.  \label{Br10}
\end{eqnarray}
\end{itemize}

A morphism of braided bialgebras is both a morphism of braided algebras and
coalgebras. This defines the category $\mathrm{BrBialg}_{\mathcal{M}}$ of
braided bialgebras.
\end{definition}

\begin{proposition}
\label{pro:AotB}Let $\mathcal{M}$ be a monoidal category.

1) Consider a datum $(A_{1},A_{2},c_{2,1})$ where $A_{1}=(A_{1},m_{1},u_{1})$
and $A_{2}=(A_{2},m_{2},u_{2})$ are algebras in $\mathcal{M}$ and $%
c_{2,1}:A_{2}\otimes A_{1}\rightarrow A_{1}\otimes A_{2}$ is a morphism in $%
\mathcal{M}$ such that
\begin{gather}
c_{2,1}(m_{2}\otimes A_{1})=(A_{1}\otimes m_{2})(c_{2,1}\otimes
A_{2})(A_{2}\otimes c_{2,1}),  \label{c1} \\
c_{2,1}\left( A_{2}\otimes m_{1}\right) =(m_{1}\otimes A_{2})\left(
A_{1}\otimes c_{2,1}\right) (c_{2,1}\otimes A_{1}),  \label{c2} \\
c_{2,1}(u_{2}\otimes 1)l_{A_{1}}^{-1}=\left( A_{1}\otimes u_{2}\right)
r_{A_{1}}^{-1},\qquad c_{2,1}(A_{2}\otimes u_{1})r_{2}^{-1}=\left(
u_{1}\otimes A_{2}\right) l_{A_{2}}^{-1}.  \label{c3}
\end{gather}%
Then $\left( A_{1}\otimes A_{2},m_{A_{1}\otimes A_{2}},u_{A_{1}\otimes
A_{2}}\right) $ is an algebra in $\mathcal{M}$ where%
\begin{eqnarray}
m_{A_{1}\otimes A_{2}} &:&=\left( m_{1}\otimes m_{2}\right) \left(
A_{1}\otimes c_{2,1}\otimes A_{2}\right) ,  \label{m1} \\
u_{A_{1}\otimes A_{2}} &:&=\left( u_{1}\otimes u_{2}\right) \Delta _{\mathbf{%
1}}.  \label{m2}
\end{eqnarray}

2) Let $(A_{1},m_{1},u_{1}),A_{2}=(A_{2},m_{2},u_{2})\in \mathrm{Alg}_{%
\mathcal{M}}.$ Assume that, for all $i,j\in \left\{ 1,2\right\} ,$ there are
isomorphisms $c_{i,j}:A_{i}\otimes A_{j}\rightarrow A_{j}\otimes A_{i}$ such
that the following equalities are fulfilled%
\begin{gather}
c_{i,j}(m_{i}\otimes A_{j})=\left( A_{j}\otimes m_{i}\right) \left(
c_{i,j}\otimes A_{i}\right) \left( A_{i}\otimes c_{i,j}\right) ,  \label{c21}
\\
c_{i,j}\left( A_{i}\otimes m_{j}\right) =\left( m_{j}\otimes A_{i}\right)
\left( A_{j}\otimes c_{i,j}\right) \left( c_{i,j}\otimes A_{j}\right) ,
\label{c22} \\
c_{i,j}\left( u_{i}\otimes A_{j}\right) l_{A_{j}}^{-1}=\left( A_{j}\otimes
u_{i}\right) r_{A_{j}}^{-1},\qquad c_{i,j}(A_{i}\otimes
u_{j})r_{A_{i}}^{-1}=\left( u_{j}\otimes A_{i}\right) l_{A_{i}}^{-1},
\label{c31} \\
\left( A_{k}\otimes c_{i,j}\right) \left( c_{i,k}\otimes A_{j}\right) \left(
A_{i}\otimes c_{j,k}\right) =\left( c_{j,k}\otimes A_{i}\right) \left(
A_{j}\otimes c_{i,k}\right) \left( c_{i,j}\otimes A_{k}\right) ,  \label{cij}
\end{gather}%
for all $i,j,k\in \left\{ 1,2\right\} $.

Then $(A_{1},m_{1},u_{1},c_{1,1}),(A_{2},m_{2},u_{2},c_{2,2})\in \mathrm{%
BrAlg}_{\mathcal{M}}.$

Moreover, for all $i,j\in \left\{ 1,2\right\}$ , $\left( A_{i}\otimes
A_{j},m_{A_{i}\otimes A_{j}},u_{A_{i}\otimes A_{j}},c_{A_{i}\otimes
A_{j}}\right) \in \mathrm{BrAlg}_{\mathcal{M}}$ where $m_{A_{i}\otimes
A_{j}} $ and $u_{A_{i}\otimes A_{j}}$ are as in 1) and
\begin{equation*}
c_{A_{i}\otimes A_{j}}=\left( A_{i}\otimes c_{i,j}\otimes A_{j}\right)
\left( c_{i,i}\otimes c_{j,j}\right) \left( A_{i}\otimes c_{j,i}\otimes
A_{j}\right) .
\end{equation*}%
3) Let $A_{1},A_{2}$ (respectively $A_{1}^{\prime },A_{2}^{\prime }$) be
objects in $\mathrm{Alg}_{\mathcal{M}}$ that fulfil the requirements in 2).
Let $f_{1}:A_{1}\rightarrow A_{1}^{\prime }$ and $f_{2}:A_{2}\rightarrow
A_{2}^{\prime }$ be morphisms in $\mathrm{Alg}_{\mathcal{M}}$ such that
\begin{equation}
\left( f_{i}\otimes f_{j}\right) c_{j,i}=c_{j,i}^{\prime }\left(
f_{j}\otimes f_{i}\right) ,  \label{gotf}
\end{equation}%
for all $i,j\in \left\{ 1,2\right\} $. Then,for all $i,j\in \left\{
1,2\right\} ,$ $f_{i}$ and $f_{i}\otimes f_{j}$ are morphisms in $\mathrm{%
BrAlg}_{\mathcal{M}}.$
\end{proposition}

\begin{proof}
It is straightforward. %
\end{proof}

\begin{lemma}
\label{Lem:AotAotA}Let $(A,m_{A},u_{A},c_{A})\in \mathrm{BrAlg}_{\mathcal{M}%
}.$ Then $(A_{1},m_{1},u_{1}),(A_{2},m_{2},u_{2}),c_{i,j}$ fulfil the
requirements of Proposition \ref{pro:AotB}, where $(A_{1},m_{1},u_{1}):=%
\left( A,m_{A},u_{A}\right) ,c_{1,1}:=c_{A},$
\begin{eqnarray*}
A_{2} &=&A\otimes A,\quad m_{2}:=\left( m_{A}\otimes m_{A}\right) \left(
A\otimes c_{A}\otimes A\right) ,\quad u_{2}:=\left( u_{A}\otimes
u_{A}\right) \Delta _{\mathbf{1}}, \\
c_{2,2} &:=&\left( A\otimes c_{A}\otimes A\right) \left( c_{A}\otimes
c_{A}\right) \left( A\otimes c_{A}\otimes A\right) , \\
c_{2,1} &:=&\left( c_{A}\otimes A\right) \left( A\otimes c_{A}\right)
:A_{2}\otimes A_{1}\rightarrow A_{1}\otimes A_{2}, \\
c_{1,2} &:=&\left( A\otimes c_{A}\right) \left( c_{A}\otimes A\right)
:A_{1}\otimes A_{2}\rightarrow A_{2}\otimes A_{1}.
\end{eqnarray*}%
In particular $(E,m_{E},u_{E},c_{E})\in \mathrm{BrAlg}_{\mathcal{M}}$, where%
\begin{eqnarray*}
E &:=&A_{1}\otimes A_{2},\quad m_{E}:=\left( m_{1}\otimes m_{2}\right)
\left( A_{1}\otimes c_{2,1}\otimes A_{2}\right) ,\quad u_{E}:=\left(
u_{1}\otimes u_{2}\right) \Delta _{\mathbf{1}}, \\
c_{E} &:=&\left( A_{1}\otimes c_{1,2}\otimes A_{2}\right) \left(
c_{1,1}\otimes c_{2,2}\right) \left( A_{1}\otimes c_{2,1}\otimes
A_{2}\right) .
\end{eqnarray*}
\end{lemma}

\begin{proof}
It is straightforward.%
\end{proof}

\begin{definition}
\label{def:conservative}A functor is called \emph{conservative} if it
reflects isomorphisms.
\end{definition}

\begin{proposition}
\label{pro:BrBialg}Let $\mathcal{M}$ and $\mathcal{M}^{\prime }$ be monoidal
categories. Let $\left( F,\phi _{0},\phi _{2}\right) :\mathcal{M}\rightarrow
\mathcal{M}^{\prime }$ be a monoidal functor. Then $F$ induces functors
\begin{gather*}
\mathrm{Br}F:\mathrm{Br}_{\mathcal{M}}\rightarrow \mathrm{Br}_{\mathcal{M}%
^{\prime }},\qquad \mathrm{Alg}F:\mathrm{Alg}_{\mathcal{M}}\rightarrow
\mathrm{Alg}_{\mathcal{M}^{\prime }}, \\
\mathrm{BrAlg}F:\mathrm{BrAlg}_{\mathcal{M}}\rightarrow \mathrm{BrAlg}_{%
\mathcal{M}^{\prime }},\qquad \mathrm{BrBialg}F:\mathrm{BrBialg}_{\mathcal{M}%
}\rightarrow \mathrm{BrBialg}_{\mathcal{M}^{\prime }}
\end{gather*}%
which act as $F$ on morphisms and defined on objects by%
\begin{eqnarray*}
\left( \mathrm{Br}F\right) \left( V,c_{V}\right) &:&=\left( FV,c_{FV}\right)
, \\
\left( \mathrm{Alg}F\right) \left( A,m_{A},u_{A}\right) &:&=\left(
FA,m_{FA},u_{FA}\right) , \\
\left( \mathrm{BrAlg}F\right) \left( A,m_{A},u_{A},c_{A}\right) &:&=\left(
FA,m_{FA},u_{FA},c_{FA}\right) , \\
\left( \mathrm{BrBialg}F\right) \left( B,m_{B},u_{B},\Delta _{B},\varepsilon
_{B},c_{B}\right) &:&=\left( FB,m_{FB},u_{FB},\Delta _{FB},\varepsilon
_{FB},c_{FB}\right) ,
\end{eqnarray*}%
where%
\begin{eqnarray*}
c_{FV} &:&=\phi _{2}^{-1}\left( V,V\right) \circ Fc_{V}\circ \phi _{2}\left(
V,V\right) :FV\otimes FV\rightarrow FV\otimes FV \\
m_{FA} &:&=Fm_{A}\circ \phi _{2}\left( A,A\right) :FA\otimes FA\rightarrow
FA, \\
u_{FA} &:&=Fu_{A}\circ \phi _{0}:\mathbf{1}\rightarrow FA, \\
\Delta _{FB} &:&=\phi _{2}^{-1}\left( B,B\right) \circ F\Delta
_{B}:FB\rightarrow FB\otimes FB, \\
\varepsilon _{FB} &:&=\phi _{0}^{-1}\circ F\varepsilon _{B}:FB\rightarrow
\mathbf{1},
\end{eqnarray*}%
and the following diagrams commute, where the vertical arrows denote the
obvious forgetful functors.
\begin{gather*}
\xymatrixrowsep{15pt}\xymatrix{\mathrm{Br}_{\mathcal{M}}
\ar[r]^{\mathrm{Br}F}\ar[d]_H&\mathrm{Br}_{\mathcal{M}^{\prime }}\ar[d]^{H^\prime}\\
\mathcal{M} \ar[r]^F&\mathcal{M}^{\prime }}
\qquad
\xymatrixrowsep{15pt}\xymatrix{\mathrm{Alg}_{\mathcal{M}}
\ar[r]^{\mathrm{Alg}F}\ar[d]_\Omega&\mathrm{Alg}_{\mathcal{M}^{\prime }}\ar[d]^{\Omega^\prime}\\
\mathcal{M} \ar[r]^F&\mathcal{M}^{\prime }}
\qquad
\xymatrixrowsep{15pt} \xymatrixcolsep{25pt}\xymatrix{\mathrm{BrAlg}_{\mathcal{M}}
\ar[r]^{\mathrm{BrAlg}F}\ar[d]_{H_{\mathrm{Alg}}}&\mathrm{BrAlg}_{\mathcal{M}^{\prime }}\ar[d]^{{H^\prime_{\mathrm{Alg}}}}\\
\mathrm{Alg}_{\mathcal{M}} \ar[r]^{\mathrm{Alg}F}&\mathrm{Alg}_{\mathcal{M}^{\prime }}}\\
\xymatrixrowsep{15pt} \xymatrixcolsep{25pt}\xymatrix{\mathrm{BrAlg}_{\mathcal{M}}
\ar[r]^{\mathrm{BrAlg}F}\ar[d]_{\Omega_{\mathrm{Br}}}&\mathrm{BrAlg}_{\mathcal{M}^{\prime }}\ar[d]^{{\Omega^\prime_{\mathrm{Br}}}}\\
\mathrm{Br}_{\mathcal{M}} \ar[r]^{\mathrm{Br}F}&\mathrm{Br}_{\mathcal{M}^{\prime }}}
\qquad
\xymatrixrowsep{15pt} \xymatrixcolsep{35pt}\xymatrix{\mathrm{BrBialg}_{\mathcal{M}}
\ar[r]^{\mathrm{BrBialg}F}\ar[d]_{\mho_{\mathrm{Br}}}&\mathrm{BrBialg}_{\mathcal{M}^{\prime }}\ar[d]^{{\mho^\prime_{\mathrm{Br}}}}\\
\mathrm{BrAlg}_{\mathcal{M}} \ar[r]^{\mathrm{BrAlg}F}&\mathrm{BrAlg}_{\mathcal{M}^{\prime }}}.
\end{gather*}%
Moreover

1) The functors $H,\Omega ,H_{\mathrm{Alg}},\Omega _{\mathrm{Br}},\mho _{%
\mathrm{Br}}$ are conservative.

2) $\mathrm{Br}F,\mathrm{Alg}F,\mathrm{BrAlg}F$ and $\mathrm{BrBialg}F$ are
equivalences (resp. isomorphisms or conservative) whenever $F$ is.
\end{proposition}

\begin{proof}
Let $\left( V,c_{V}\right) $ be a braided object in $\mathcal{M}$. Let us
check that $\left( FV,c_{FV}\right) $ is a braided object in $\mathcal{M}%
^{\prime }$. We have%
\begin{equation*}
\phi _{2}\left( V\otimes V,V\right) \circ \left( \phi _{2}\left( V,V\right)
\otimes FV\right) =\phi _{2}\left( V,V\otimes V\right) \circ \left(
FV\otimes \phi _{2}\left( V,V\right) \right) .
\end{equation*}%
Call $\omega :FV\otimes FV\otimes FV\rightarrow F\left( V\otimes V\otimes
V\right) $ this composition. Using the definition of $c_{FV}$ and the
naturality of $\phi _{2}$ one easily gets%
\begin{equation}
\omega \circ \left( c_{FV}\otimes FV\right) =F\left( c_{V}\otimes V\right)
\circ \omega ,\quad \omega \circ \left( FV\otimes c_{FV}\right) =F\left(
V\otimes c_{V}\right) \circ \omega .  \label{form:omega}
\end{equation}%
Thus we obtain%
\begin{eqnarray*}
&&\omega \circ \left( c_{FV}\otimes FV\right) \circ \left( FV\otimes
c_{FV}\right) \circ \left( c_{FV}\otimes FV\right) \\
&=&F\left[ \left( c_{V}\otimes V\right) \circ \left( V\otimes c_{V}\right)
\circ \left( c_{V}\otimes V\right) \right] \circ \omega \\
&=&F\left[ \left( V\otimes c_{V}\right) \circ \left( c_{V}\otimes V\right)
\circ \left( V\otimes c_{V}\right) \right] \circ \omega \\
&=&\omega \circ \left( FV\otimes c_{FV}\right) \circ \left( c_{FV}\otimes
FV\right) \circ \left( FV\otimes c_{FV}\right) .
\end{eqnarray*}%
Since $\omega $ is an isomorphism, we conclude that $c_{FV}$ is a braiding.
Thus $\left( FV,c_{FV}\right) $ is a braided object. Let $f:\left(
V,c_{V}\right) \rightarrow \left( V^{\prime },c_{V^{\prime }}\right) $ be a
morphism of braided objects in $\mathcal{M}$. Using the definition of $%
c_{FV^{\prime }},$ the naturality of $\phi _{2},$ that $f$ is compatible
with the braiding, one easily gets $c_{FV^{\prime }}\circ \left( Ff\otimes
Ff\right) =\left( Ff\otimes Ff\right) \circ c_{FV}$. Thus the functor $%
\mathrm{Br}F:\mathrm{Br}_{\mathcal{M}}\rightarrow \mathrm{Br}_{\mathcal{M}%
^{\prime }}$ of the statement is well-defined. By construction one easily
checks that $H^{\prime }\circ \mathrm{Br}F=F\circ H.$

Let $\left( A,m_{A},u_{A}\right) \in \mathrm{Alg}_{\mathcal{M}}.$ By \cite[%
Proposition 1.5]{AMS-splitting}, we have that $\left(
FA,m_{FA},u_{FA}\right) $ is in $\mathrm{Alg}_{\mathcal{M}^{\prime }}.$ Let $%
f:\left( A,m_{A},u_{A}\right) \rightarrow \left( A^{\prime },m_{A^{\prime
}},u_{A^{\prime }}\right) $ be a morphism of algebras in $\mathcal{M}$.
Using the definition of $m_{FA^{\prime }},$ the naturality of $\phi _{2}$
and the multiplicativity of $f$ one gets $m_{FA^{\prime }}\circ \left(
Ff\otimes Ff\right) =Ff\circ m_{FA}.$ Moreover, using the definition of $%
u_{FA}$ and the unitarity of $f$ one has $Ff\circ u_{FA}=u_{FA^{\prime }}$.
Thus the functor $\mathrm{Alg}F:\mathrm{Alg}_{\mathcal{M}}\rightarrow
\mathrm{Alg}_{\mathcal{M}^{\prime }}$ is well-defined. It is clear that $%
F\circ \Omega =\Omega ^{\prime }\circ \mathrm{Alg}F$.

Let $\left( A,m_{A},u_{A},c_{A}\right) $ be an object in $\mathrm{BrAlg}_{%
\mathcal{M}}.$ Then $\left( A,c_{A}\right) \in \mathrm{Br}_{\mathcal{M}}$
and $\left( A,m_{A},u_{A}\right) \in \mathrm{Alg}_{\mathcal{M}}$ so that, by
the foregoing, we get that $\left( FA,c_{FA}\right) \in \mathrm{Br}_{%
\mathcal{M}^{\prime }}$ and $\left( FA,m_{FA},u_{FA}\right) \in \mathrm{Alg}%
_{\mathcal{M}^{\prime }}$. We have%
\begin{gather*}
\phi _{2}\left( A,A\right) \circ (FA\otimes m_{FA})\circ (c_{FA}\otimes
FA)\circ (FA\otimes c_{FA}) \\
\overset{(\ast )}{=}F(A\otimes m_{A})\circ \omega \circ \left( c_{FA}\otimes
FA\right) \circ (FA\otimes c_{FA}) \\
\overset{(\ref{form:omega})}{=}F(A\otimes m_{A})\circ F\left( c_{A}\otimes
A\right) \circ F(A\otimes c_{A})\circ \omega \\
\overset{(\ref{Br2})}{=}F\left[ c_{A}\circ (m_{A}\otimes A)\right] \circ
\omega \overset{(\ast )}{=}Fc_{A}\circ \phi _{2}\left( A,A\right) \circ
(m_{FA}\otimes FA)=\phi _{2}\left( A,A\right) \circ c_{FA}\circ
(m_{FA}\otimes FA)
\end{gather*}

where in (*) we used the definition of $m_{FA},$ the naturality of $\phi
_{2}\left( A,A\right) $ and the definition of $\omega .$ Thus%
\begin{equation*}
(FA\otimes m_{FA})\circ (c_{FA}\otimes FA)\circ (FA\otimes
c_{FA})=c_{FA}\circ (m_{FA}\otimes FA).
\end{equation*}%
Similarly one proves that $(m_{FA}\otimes FA)\circ \left( FA\otimes
c_{FA}\right) \circ (c_{FA}\otimes FA)=c_{FA}\circ (FA\otimes m_{FA}).$
Moreover
\begin{eqnarray*}
&&\phi _{2}\left( A,A\right) \circ c_{FA}\circ (u_{FA}\otimes FA)\circ
l_{FA}^{-1}\overset{(\ast \ast )}{=}F\left[ c_{A}\circ (u_{A}\otimes A)\circ
l_{A}^{-1}\right] \\
&&\overset{(\ref{Br4})}{=}F\left[ \left( A\otimes u_{A}\right) \circ
r_{A}^{-1}\right] \overset{(\ast \ast )}{=}\phi _{2}\left( A,A\right) \circ
\left( FA\otimes u_{FA}\right) \circ r_{FA}^{-1}
\end{eqnarray*}

where in (**) we used the definitions of $c_{FA}$ and $u_{FA},$ the
naturality of $\phi _{2}$ and the definition of monoidal functor. Thus $%
c_{FA}\circ (u_{FA}\otimes FA)\circ l_{FA}^{-1}=\left( FA\otimes
u_{FA}\right) \circ r_{FA}^{-1}.$ Similarly one proves that $c_{FA}\circ
(FA\otimes u_{FA})\circ r_{FA}^{-1}=\left( u_{FA}\otimes FA\right) \circ
l_{FA}^{-1}.$ We have so proved that $\left( FA,m_{FA},u_{FA},c_{FA}\right) $
is a braided algebra in $\mathcal{M}^{\prime }$. Since, by definition, a
morphism of braided algebras is just a morphism of braided objects and of
algebras, it is clear, by the foregoing, that $F$ preserves morphisms of
braided algebras so that the functor $\mathrm{BrAlg}F:\mathrm{BrAlg}_{%
\mathcal{M}}\rightarrow \mathrm{BrAlg}_{\mathcal{M}^{\prime }}$ is
well-defined. It is clear that $\mathrm{Alg}F\circ H_{\mathrm{Alg}}=H_{%
\mathrm{Alg}}^{\prime }\circ \mathrm{BrAlg}F$ and $\mathrm{Br}F\circ \Omega
_{\mathrm{Br}}=\Omega _{\mathrm{Br}}^{\prime }\circ \mathrm{BrAlg}F.$

Let us define the functor $\mathrm{BrBialg}F.$ Let $\left(
B,m_{B},u_{B},\Delta _{B},\varepsilon _{B},c_{B}\right) $ be a braided
bialgebra in $\mathcal{M}$. Then $\left( B,m_{B},u_{B},c_{B}\right) $ is a
braided algebra in $\mathcal{M}$, so that, by the foregoing, $\left(
FB,m_{FB},u_{FB},c_{FB}\right) $ is a braided algebra in $\mathcal{M}%
^{\prime }$. A dual argument proves that $\left( FB,\Delta _{FB},\varepsilon
_{FB},c_{FB}\right) $ is a braided coalgebra in $\mathcal{M}^{\prime }$.

We compute
\begin{align*}
& \phi _{2}\left( B,B\right) \circ (m_{FB}\otimes m_{FB})\circ (FB\otimes
c_{FB}\otimes FB)\circ (\Delta _{FB}\otimes \Delta _{FB}) \\
& \overset{(\ast \ast \ast )}{=}F(m_{B}\otimes m_{B})\circ \phi _{2}\left(
B\otimes B\otimes B,B\right) \circ \left( \omega \otimes FB\right) \circ
(FB\otimes c_{FB}\otimes FB)\circ (\Delta _{FB}\otimes \Delta _{FB}) \\
& \overset{(\ref{form:omega})}{=}F(m_{B}\otimes m_{B})\circ \phi _{2}\left(
B\otimes B\otimes B,B\right) \circ (F\left( B\otimes c_{B}\right) \otimes
FB)\circ \left( \omega \otimes FB\right) \circ (\Delta _{FB}\otimes \Delta
_{FB}) \\
& \overset{(\ast \ast \ast )^{\prime }}{=}F\left[ (m_{B}\otimes m_{B})\circ
(B\otimes c_{B}\otimes B)\circ (\Delta _{B}\otimes \Delta _{B})\right] \circ
\phi _{2}\left( B,B\right) \\
\overset{(\ref{Br1})}{=}& F\left( \Delta _{B}\circ m_{B}\right) \circ \phi
_{2}\left( B,B\right) =F\Delta _{B}\circ m_{FB}=\phi _{2}\left( B,B\right)
\circ \Delta _{FB}\circ m_{FB}
\end{align*}%
where in (***) we used the definition of $m_{FB},$ the naturality of $\phi
_{2},$ the fact that $F$ is a monoidal functor and the definition of $\omega
$ while in (***)' we used the naturality of $\phi _{2},$ the definition of $%
\omega ,$ the definition of $\Delta _{FB}$ (the one on the left of the
tensor)$,$ the fact that $F$ is monoidal, again the definition of $\Delta
_{FB}$ (the one on the right of the tensor) and the naturality of $\phi
_{2}. $ Thus%
\begin{equation*}
(m_{FB}\otimes m_{FB})\circ (FB\otimes c_{FB}\otimes FB)\circ (\Delta
_{FB}\otimes \Delta _{FB})=\Delta _{FB}\circ m_{FB}.
\end{equation*}%
We calculate

definitions of $\Delta _{FB}$ and $u_{FB},$ the unitarity of $\Delta _{B},$
the equality $\Delta _{\mathbf{1}}=l_{\mathbf{1}}^{-1}$, the monoidality of $%
F,$ the naturality of the left unit constraint, the naturality of $\phi _{2}$
and again the definition of $u_{FB},$ one gets $\phi _{2}\left( B,B\right)
\circ \Delta _{FB}\circ u_{FB}=\phi _{2}\left( B,B\right) \circ
(u_{FB}\otimes u_{FB})\circ \Delta _{\mathbf{1}}$ so that $\Delta _{FB}\circ
u_{FB}=(u_{FB}\otimes u_{FB})\circ \Delta _{\mathbf{1}}$. Dually one gets $%
\varepsilon _{FB}\circ m_{FB}=m_{\mathbf{1}}\circ \left( \varepsilon
_{FB}\otimes \varepsilon _{FB}\right) .$ Finally we have%
\begin{equation*}
\varepsilon _{FB}\circ u_{FB}=\phi _{0}^{-1}\circ F\varepsilon _{B}\circ
Fu_{B}\circ \phi _{0}=\phi _{0}^{-1}\circ F\left( \varepsilon _{B}\circ
u_{B}\right) \circ \phi _{0}=\phi _{0}^{-1}\circ \phi _{0}=\mathrm{Id}_{%
\mathbf{1}}.
\end{equation*}%
We have so proved that $\left( FB,m_{FB},u_{FB},\Delta _{FB},\varepsilon
_{FB},c_{FB}\right) $ is a braided bialgebra. Let $f$ be a morphism of
braided bialgebras from $\left( B,m_{B},u_{B},\Delta _{B},\varepsilon
_{B},c_{B}\right) $ to $\left( B^{\prime },m_{B^{\prime }},u_{B^{\prime
}},\Delta _{B^{\prime }},\varepsilon _{B^{\prime }},c_{B^{\prime }}\right) $%
. Then $f:\left( B,m_{B},u_{B}\right) \rightarrow \left( B^{\prime
},m_{B^{\prime }},u_{B^{\prime }}\right) $ is a morphism of algebras and $%
f:\left( B,\Delta _{B},\varepsilon _{B}\right) \rightarrow \left( B^{\prime
},\Delta _{B^{\prime }},\varepsilon _{B^{\prime }}\right) $ is a morphism of
coalgebras. Thus $Ff:\left( FB,m_{FB},u_{FB}\right) \rightarrow \left(
FB^{\prime },m_{FB^{\prime }},u_{FB^{\prime }}\right) $ is a morphism of
algebras and $Ff:\left( FB,\Delta _{FB},\varepsilon _{FB}\right) \rightarrow
\left( FB^{\prime },\Delta _{FB^{\prime }},\varepsilon _{FB^{\prime
}}\right) $ is a morphism of coalgebras. Moreover we know that $Ff:\left(
FB,c_{FB}\right) \rightarrow \left( FB^{\prime },c_{FB^{\prime }}\right) $
is a morphism of braided objects. We have so proved that $Ff$ is a morphism
of braided bialgebras. Thus the functor $\mathrm{BrBialg}F:\mathrm{BrBialg}_{%
\mathcal{M}}\rightarrow \mathrm{BrBialg}_{\mathcal{M}^{\prime }}$ of the
statement is well-defined.

Let $\left( F,\phi _{0}^{F},\phi _{2}^{F}\right) :\mathcal{M}\rightarrow
\mathcal{M}^{\prime }$ and $\left( F^{\prime },\phi _{0}^{F^{\prime }},\phi
_{2}^{F^{\prime }}\right) :\mathcal{M}^{\prime }\rightarrow \mathcal{M}%
^{\prime \prime }$ be monoidal functors. Then $\left( F^{\prime }F,\phi
_{2}^{F^{\prime }F},\phi _{0}^{F^{\prime }F}\right) $is a monoidal functor
where $\phi _{2}^{F^{\prime }F}\left( U,V\right) :=F^{\prime }\left( \phi
_{2}^{F}\right) \left( U,V\right) \circ \phi _{2}^{F^{\prime }}\left(
FU,FV\right) $ and $\phi _{0}^{F^{\prime }F}:=F^{\prime }\left( \phi
_{0}^{F}\right) \circ \phi _{0}^{F^{\prime }}.$We compute%
\begin{equation*}
\left( \mathrm{Br}F^{\prime }\circ \mathrm{Br}F\right) \left( V,c_{V}\right)
=\mathrm{Br}F^{\prime }\left( FV,c_{FV}\right) =\left( F^{\prime
}FV,c_{F^{\prime }\left( FV\right) }\right) \overset{(\bullet )}{=}\left(
F^{\prime }FV,c_{\left( F^{\prime }F\right) V}\right) =\mathrm{Br}\left(
F^{\prime }F\right) \left( V,c_{V}\right) .
\end{equation*}%
where $(\bullet )$ follows from the following computation
\begin{eqnarray*}
c_{F^{\prime }\left( FV\right) } &=&\left( \phi _{2}^{F^{\prime }}\right)
^{-1}\left( FV,FV\right) \circ F^{\prime }c_{FV}\circ \phi _{2}^{F^{\prime
}}\left( FV,FV\right) \\
&=&\left( \phi _{2}^{F^{\prime }}\right) ^{-1}\left( FV,FV\right) \circ
F^{\prime }\left( \left( \phi _{2}^{F}\right) ^{-1}\left( V,V\right) \circ
Fc_{V}\circ \left( \phi _{2}^{F}\right) \left( V,V\right) \right) \circ \phi
_{2}^{F^{\prime }}\left( FV,FV\right) \\
&=&\left( \phi _{2}^{F^{\prime }}\right) ^{-1}\left( FV,FV\right) \circ
F^{\prime }\left( \phi _{2}^{F}\right) ^{-1}\left( V,V\right) \circ
F^{\prime }Fc_{V}\circ F^{\prime }\left( \phi _{2}^{F}\right) \left(
V,V\right) \circ \phi _{2}^{F^{\prime }}\left( FV,FV\right) \\
&=&\phi _{2}^{F^{\prime }F}\left( V,V\right) ^{-1}\circ F^{\prime
}Fc_{V}\circ \phi _{2}^{F^{\prime }F}\left( V,V\right) =c_{\left( F^{\prime
}F\right) V}.
\end{eqnarray*}%
Thus we have $\mathrm{Br}F^{\prime }\circ \mathrm{Br}F=\mathrm{Br}\left(
F^{\prime }F\right) .$ We compute%
\begin{eqnarray*}
\left( \mathrm{Alg}F^{\prime }\circ \mathrm{Alg}F\right) \left(
A,m_{A},u_{A}\right) &=&\mathrm{Alg}F^{\prime }\left(
FA,m_{FA},u_{FA}\right) =\left( F^{\prime }FA,m_{F^{\prime }\left( FA\right)
},u_{F^{\prime }\left( FA\right) }\right) \\
\overset{(\bullet \bullet )}{=}\left( F^{\prime }FA,m_{\left( F^{\prime
}F\right) A},u_{\left( F^{\prime }F\right) A}\right) &=&\mathrm{Alg}\left(
F^{\prime }F\right) \left( A,m_{A},u_{A}\right)
\end{eqnarray*}%
where $(\bullet \bullet )$ follows from the following computations%
\begin{eqnarray*}
m_{F^{\prime }\left( FA\right) } &=&F^{\prime }m_{FA}\circ \phi
_{2}^{F^{\prime }}\left( FA,FA\right) =F^{\prime }Fm_{A}\circ F^{\prime
}\phi _{2}^{F}\left( A,A\right) \circ \phi _{2}^{F^{\prime }}\left(
FA,FA\right) \\
&=&F^{\prime }Fm_{A}\circ \phi _{2}^{F^{\prime }F}\left( A,A\right)
=m_{\left( F^{\prime }F\right) A}, \\
u_{F^{\prime }\left( FA\right) } &=&F^{\prime }u_{FA}\circ \phi
_{0}^{F^{\prime }}=Fu_{A}\circ F^{\prime }\phi _{0}^{F}\circ \phi
_{0}^{F^{\prime }}=Fu_{A}\circ \phi _{0}^{F^{\prime }F}=u_{\left( F^{\prime
}F\right) A}.
\end{eqnarray*}%
Thus $\mathrm{Alg}F^{\prime }\circ \mathrm{Alg}F=\mathrm{Alg}\left(
F^{\prime }F\right) .$ By the foregoing it is clear that $\mathrm{BrAlg}%
F^{\prime }\circ \mathrm{BrAlg}F=\mathrm{BrAlg}\left( F^{\prime }F\right) .$

By the foregoing and a dual argument on the comultiplication and counit, one
also gets that $\mathrm{BrBialg}F^{\prime }\circ \mathrm{BrBialg}F=\mathrm{%
BrBialg}\left( F^{\prime }F\right) .$

Consider the strict monoidal functor $\mathrm{Id}_{\mathcal{M}}.$ A direct
computation shows
\begin{gather*}
\mathrm{Br}\left( \mathrm{Id}_{\mathcal{M}}\right) =\mathrm{Id}_{\mathrm{Br%
}_{\mathcal{M}}},\qquad \mathrm{Alg}\left( \mathrm{Id}_{\mathcal{M}}\right) =%
\mathrm{Id}_{\mathrm{Alg}_{\mathcal{M}}}, \\
\mathrm{BrAlg}\left( \mathrm{Id}_{\mathcal{M}}\right) =\mathrm{Id}_{%
\mathrm{BrAlg}_{\mathcal{M}}},\qquad \mathrm{BrBialg}\left( \mathrm{Id}_{%
\mathcal{M}}\right) =\mathrm{Id}_{\mathrm{BrBialg}_{\mathcal{M}}}.
\end{gather*}%
Let $\left( F,\phi _{0}^{F},\phi _{2}^{F}\right) :\mathcal{M}\rightarrow
\mathcal{M}^{\prime }$ and $\left( F^{\prime },\phi _{0}^{F^{\prime }},\phi
_{2}^{F^{\prime }}\right) :\mathcal{M}\rightarrow \mathcal{M}^{\prime }$ be
monoidal functors. Let $\xi :\left( F,\phi _{0}^{F},\phi _{2}^{F}\right)
\rightarrow \left( F^{\prime },\phi _{0}^{F^{\prime }},\phi _{2}^{F^{\prime
}}\right) $ be a morphism of monoidal functors i.e. a natural transformation
$\xi :F\rightarrow F^{\prime }$ such that $\xi \mathbf{1}\circ \phi
_{0}^{F}=\phi _{0}^{F^{\prime }}$ and $\xi \left( U\otimes V\right) \circ
\phi _{2}^{F}\left( U,V\right) =\phi _{2}^{F^{\prime }}\left( U,V\right)
\circ \left( \xi U\otimes \xi V\right) .$ Let us define a natural
transformation $\mathrm{Br}\xi :\mathrm{Br}F\rightarrow \mathrm{Br}F^{\prime
}.$ First we have to define a morphism $\mathrm{Br}\xi \left( V,c\right) :%
\mathrm{Br}F\left( V,c\right) \rightarrow \mathrm{Br}F^{\prime }\left(
V,c\right) .$ in $\mathrm{Br}_{\mathcal{M}^{\prime }}$. Now $\mathrm{Br}%
F\left( V,c\right) =\left( FV,c_{FV}\right) $ and $\mathrm{Br}F^{\prime
}\left( V,c\right) =\left( F^{\prime }V,c_{F^{\prime }V}\right) $ so that a
natural candidate is $\xi V.$ We have to check it is a morphism of braided
objects i.e. $\left( \xi V\otimes \xi V\right) \circ c_{FV}=c_{F^{\prime
}V}\circ \left( \xi U\otimes \xi V\right) $ but this is achieved by means of
the definition of $c_{FV},$ the fact that $\xi $ is a morphism of monoidal
functors and the naturality of $\xi $. Thus $\xi V$ really induces a unique
morphism $\mathrm{Br}\xi \left( V,c\right) :\mathrm{Br}F\left( V,c\right)
\rightarrow \mathrm{Br}F^{\prime }\left( V,c\right) $ such that $H^{\prime }%
\mathrm{Br}\xi \left( V,c\right) =\xi V.$ Let us check that $\mathrm{Br}\xi
\left( V,c\right) $ is natural in $\left( V,c\right) $. Let $f:\left(
V,c\right) \rightarrow \left( V^{\prime },c^{\prime }\right) $ be a morphism
of braided object in $\mathcal{M}$. Then%
\begin{gather*}
H^{\prime }\left( \mathrm{Br}\xi \left( V^{\prime },c^{\prime }\right) \circ
\mathrm{Br}F\left( f\right) \right) =H^{\prime }\mathrm{Br}\xi \left(
V^{\prime },c^{\prime }\right) \circ H^{\prime }\mathrm{Br}F\left( f\right)
\\
=\xi V^{\prime }\circ FH\left( f\right) =F^{\prime }H\left( f\right) \circ
\xi V=H^{\prime }\left( \mathrm{Br}F^{\prime }\left( f\right) \circ \mathrm{%
Br}\xi \left( V,c\right) \right)
\end{gather*}%
so that $\mathrm{Br}\xi \left( V^{\prime },c^{\prime }\right) \circ \mathrm{%
Br}F\left( f\right) =\mathrm{Br}F^{\prime }\left( f\right) \circ \mathrm{Br}%
\xi \left( V,c\right) $ and hence we get a natural transformation $\mathrm{Br%
}\xi :\mathrm{Br}F\rightarrow \mathrm{Br}F^{\prime }.$

We have to define a morphism $\mathrm{Alg}\xi \left( A,m_{A},u_{A}\right):\mathrm{Alg}F\left(
A,m_{A},u_{A}\right)\rightarrow  \mathrm{Alg}F^{\prime }\left(
A,m_{A},u_{A}\right) $
in $\mathrm{Alg}_{\mathcal{M}^{\prime }}$. Now $\mathrm{Alg}F\left( A,m_{A},u_{A}\right)
=\left( FA,m_{FA},u_{FA}\right) $ and $\mathrm{Alg}F^{\prime }\left(
A,m_{A},u_{A}\right) =\left( F^{\prime }A,m_{F^{\prime }A},u_{F^{\prime
}A}\right) $ so that a natural candidate is again $\xi A:FA\rightarrow
F^{\prime }A.$ We have to check it is a morphism of algebras in $\mathcal{M}%
^{\prime }$ i.e. that $\xi A\circ m_{FA}=m_{F^{\prime }A}\circ \left( \xi
A\otimes \xi A\right) $ and $\xi A\circ u_{FA}=u_{F^{\prime }A}$ but these
equalities follow by definition of $m_{FA}$ (resp. $u_{FA}$) the naturality
of $\xi $ and the fact that $\xi $ is a morphism of monoidal functors. Hence
there is a unique morphism $\mathrm{Alg}\xi \left( A,m_{A},u_{A}\right) $
such that $\Omega ^{\prime }\mathrm{Alg}\xi \left( A,m_{A},u_{A}\right) =\xi
A.$ We check it is natural in $\left( A,m_{A},u_{A}\right) .$  For an algebra morphism $f:\left( A,m_{A},u_{A}\right) \rightarrow \left( A^{\prime
},m_{A^{\prime }},u_{A^{\prime }}\right) ,$ we get%
\begin{gather*}
\Omega ^{\prime }\left[ \mathrm{Alg}\xi \left( A^{\prime },m_{A^{\prime
}},u_{A^{\prime }}\right) \circ \mathrm{Alg}F\left( f\right) \right] =\xi
A^{\prime }\circ F\Omega \left( f\right) \\
=F^{\prime }\Omega \left( f\right) \circ \xi A=\Omega ^{\prime }\left[
\mathrm{Alg}F^{\prime }\left( f\right) \circ \mathrm{Alg}\xi \left(
A,m_{A},u_{A}\right) \right]
\end{gather*}%
so that we get a natural transformation $\mathrm{Alg}\xi :\mathrm{Alg}%
F\rightarrow \mathrm{Alg}F^{\prime }.$ By the foregoing and the definition
of $\mathrm{BrAlg}F$ we can define a natural transformation $\mathrm{BrAlg}%
\xi :\mathrm{BrAlg}F\rightarrow \mathrm{BrAlg}F^{\prime }$ using again $\xi
A.$ Similarly one gets a natural transformation $\mathrm{BrBialg}\xi :%
\mathrm{BrBialg}F\rightarrow \mathrm{BrBialg}F^{\prime }$.

Let $\left( F,\phi _{0}^{F},\phi _{2}^{F}\right) :\mathcal{M}\rightarrow
\mathcal{M}^{\prime }$, $\left( F^{\prime },\phi _{0}^{F^{\prime }},\phi
_{2}^{F^{\prime }}\right) :\mathcal{M}\rightarrow \mathcal{M}^{\prime }$ and
$\left( F^{\prime \prime },\phi _{0}^{F^{\prime \prime }},\phi
_{2}^{F^{\prime \prime }}\right) :\mathcal{M}\rightarrow \mathcal{M}^{\prime
}$ be monoidal functors. Let $\xi :\left( F,\phi _{0}^{F},\phi
_{2}^{F}\right) \rightarrow \left( F^{\prime },\phi _{0}^{F^{\prime }},\phi
_{2}^{F^{\prime }}\right) $ and $\xi ^{\prime }:\left( F^{\prime },\phi
_{0}^{F^{\prime }},\phi _{2}^{F^{\prime }}\right) \rightarrow \left(
F^{\prime \prime },\phi _{0}^{F^{\prime \prime }},\phi _{2}^{F^{\prime
\prime }}\right) $ be morphisms of monoidal functors. Thus%
\begin{eqnarray*}
H^{\prime }\left( \mathrm{Br}\xi ^{\prime }\circ \mathrm{Br}\xi \right)
&=&H^{\prime }\mathrm{Br}\xi ^{\prime }\circ H^{\prime }\mathrm{Br}\xi =\xi
^{\prime }H^{\prime }\circ \xi H^{\prime }=\left( \xi ^{\prime }\xi \right)
H^{\prime }=H^{\prime }\left( \mathrm{Br}\left( \xi ^{\prime }\xi \right)
\right) , \\
\Omega ^{\prime }\left( \mathrm{Alg}\xi ^{\prime }\circ \mathrm{Alg}\xi
\right) &=&\Omega ^{\prime }\mathrm{Alg}\xi ^{\prime }\circ \Omega ^{\prime }%
\mathrm{Alg}\xi =\xi ^{\prime }\Omega ^{\prime }\circ \xi \Omega ^{\prime
}=\left( \xi ^{\prime }\xi \right) \Omega ^{\prime }=\Omega ^{\prime }\left(
\mathrm{Alg}\left( \xi ^{\prime }\xi \right) \right)
\end{eqnarray*}%
so that $\mathrm{Br}\xi ^{\prime }\circ \mathrm{Br}\xi =\mathrm{Br}\left(
\xi ^{\prime }\xi \right) $, $\mathrm{Alg}\xi ^{\prime }\circ \mathrm{Alg}%
\xi =\mathrm{Alg}\left( \xi ^{\prime }\xi \right) $ and hence $\mathrm{BrAlg}%
\xi ^{\prime }\circ \mathrm{BrAlg}\xi =\mathrm{BrAlg}\left( \xi ^{\prime
}\xi \right) $ and $\mathrm{BrBialg}\xi ^{\prime }\circ \mathrm{BrBialg}\xi =%
\mathrm{BrBialg}\left( \xi ^{\prime }\xi \right) .$ Moreover%
\begin{eqnarray*}
H^{\prime }\left( \mathrm{Br}\left( \mathrm{Id}_{F}\right) \right) &=&\left(
\mathrm{Id}_{F}\right) H=\mathrm{Id}_{FH}=H^{\prime }\mathrm{Id}_{\mathrm{Br}%
F}, \\
\Omega \left( \mathrm{Alg}\left( \mathrm{Id}_{F}\right) \right) &=&\left(
\mathrm{Id}_{F}\right) \Omega =\mathrm{Id}_{F\Omega }=\Omega \left( \mathrm{%
Id}_{\mathrm{Alg}F}\right) ,
\end{eqnarray*}%
so that $\mathrm{Br}\left( \mathrm{Id}_{F}\right) =\mathrm{Id}_{\mathrm{Br}%
F},\mathrm{Alg}\left( \mathrm{Id}_{F}\right) =\mathrm{Id}_{\mathrm{Alg}F}$
and hence $\mathrm{BrAlg}\left( \mathrm{Id}_{F}\right) =\mathrm{Id}_{\mathrm{%
BrAlg}F}$ and $\mathrm{BrBialg}\left( \mathrm{Id}_{F}\right) =\mathrm{Id}_{%
\mathrm{BrBialg}F}$. Now, if $\left( F,\phi _{0}^{F},\phi _{2}^{F}\right) :%
\mathcal{M}\rightarrow \mathcal{M}^{\prime }$ is a monoidal equivalence,
i.e. $F:\mathcal{M}\rightarrow \mathcal{M}^{\prime }$ is a monoidal functor
and there is a monoidal functor $\left( G,\phi _{0}^{G},\phi _{2}^{G}\right)
:\mathcal{M}^{\prime }\rightarrow \mathcal{M}$ and monoidal isomorphisms of
functors%
\begin{equation*}
\alpha :\mathrm{Id}_{\mathcal{M}^{\prime }}\rightarrow FG\qquad \beta
:GF\rightarrow \mathrm{Id}_{\mathcal{M}}.
\end{equation*}%
Then
\begin{eqnarray*}
\mathrm{Br}\alpha \circ \mathrm{Br}\left( \alpha ^{-1}\right) &=&\mathrm{Br}%
\left( \alpha \circ \alpha ^{-1}\right) =\mathrm{Br}\left( \mathrm{Id}%
_{FG}\right) =\mathrm{Id}_{\mathrm{Br}\left( FG\right) }=\mathrm{Id}_{%
\mathrm{Br}\left( F\right) \circ \mathrm{Br}\left( G\right) } \\
\mathrm{Br}\left( \alpha ^{-1}\right) \circ \mathrm{Br}\alpha &=&\mathrm{Br}%
\left( \alpha ^{-1}\circ \alpha \right) =\mathrm{Br}\left( \mathrm{Id}_{%
\mathrm{Id}_{\mathcal{M}^{\prime }}}\right) =\mathrm{Id}_{\mathrm{Br}\left(
\mathrm{Id}_{\mathcal{M}^{\prime }}\right) }=\mathrm{Id}_{\mathrm{Id}_{%
\mathrm{Br}\mathcal{M}^{\prime }}}
\end{eqnarray*}%
so that $\mathrm{Br}\alpha :\mathrm{Id}_{\mathrm{Br}\mathcal{M}^{\prime
}}\rightarrow \mathrm{Br}\left( F\right) \circ \mathrm{Br}\left( G\right) $
( and similarly $\mathrm{Br}\beta $) is a functorial isomorphism. This means
that $\mathrm{Br}F:\mathrm{Br}_{\mathcal{M}}\rightarrow \mathrm{Br}_{%
\mathcal{M}^{\prime }}$ is an equivalence. Analogously $\mathrm{Alg}F:%
\mathrm{Alg}_{\mathcal{M}}\rightarrow \mathrm{Alg}_{\mathcal{M}^{\prime }}$
is an equivalence and hence also $\mathrm{BrAlg}F:\mathrm{BrAlg}_{\mathcal{M}%
}\rightarrow \mathrm{BrAlg}_{\mathcal{M}^{\prime }}$ and $\mathrm{BrBialg}F:%
\mathrm{BrBialg}_{\mathcal{M}}\rightarrow \mathrm{BrBialg}_{\mathcal{M}%
^{\prime }}$ are equivalences.

If $F$ is a category isomorphism there is a monoidal functor $\left( G,\phi
_{0}^{G},\phi _{2}^{G}\right) :\mathcal{M}^{\prime }\rightarrow \mathcal{M}$
such that $FG=\mathrm{Id}_{\mathcal{M}^{\prime }}$ and $GF=\mathrm{Id}_{%
\mathcal{M}}.$ Hence $\mathrm{Br}\left( F\right) \circ \mathrm{Br}\left(
G\right) =\mathrm{Br}\left( FG\right) =\mathrm{Br}\left( \mathrm{Id}_{%
\mathcal{M}^{\prime }}\right) =\mathrm{Id}_{\mathrm{Br}_{\mathcal{M}^{\prime
}}}$ and similarly $\mathrm{Br}\left( G\right) \circ \mathrm{Br}\left(
F\right) =\mathrm{Id}_{\mathrm{Br}_{\mathcal{M}}}$ so that$\mathrm{\ Br}%
\left( F\right) $ is a category isomorphism. Analogously $\mathrm{Alg}F:%
\mathrm{Alg}_{\mathcal{M}}\rightarrow \mathrm{Alg}_{\mathcal{M}^{\prime }}$
is an isomorphism and hence also $\mathrm{BrAlg}F:\mathrm{BrAlg}_{\mathcal{M}%
}\rightarrow \mathrm{BrAlg}_{\mathcal{M}^{\prime }}$ and $\mathrm{BrBialg}F:%
\mathrm{BrBialg}_{\mathcal{M}}\rightarrow \mathrm{BrBialg}_{\mathcal{M}%
^{\prime }}$ are isomorphisms.

The proof of 1) is straightforward. If $F$ is conservative, using 1), one
easily check that so are $\mathrm{Br}F,\mathrm{Alg}F,\mathrm{BrAlg}F$ and $%
\mathrm{BrBialg}F$. For instance, $F$ and $H$ conservative implies $%
FH=H^{\prime }\left( \mathrm{Br}F\right) $ conservative and hence, since $%
H^{\prime }$, as any functor, preserves isomorphisms, we obtain that $%
\mathrm{Br}F$ is conservative.
\end{proof}

The following result is essentially \cite[Lemma XII.3.5, page 327]{Kassel}
in case the monoidal category is strict. We prove that it holds for any
monoidal category $\left( \mathcal{M},\otimes ,\mathbf{1}\right) $ using the
monoidal equivalence $\mathcal{M}^{\mathrm{str}}\rightarrow \mathcal{M}$
described in \cite[Theorem XI.5.3, page 291]{Kassel}, where $\mathcal{M}^{%
\mathrm{str}}$ is a strict monoidal category.

\begin{lemma}
\label{lem:Kassel} Let $\left( \mathcal{M},\otimes ,\mathbf{1}\right) $ be a
monoidal category and let $\left( V,c\right) \in \mathrm{Br}_{\mathcal{M}}$.
There there exists a unique monoidal functor $\left( F,\varphi _{2},\varphi
_{0}\right) :\mathcal{B}\rightarrow \mathcal{M}$ such that, for all $a,b\in
\mathbb{N},$
\begin{equation*}
F\left( 0\right) =\mathbf{1},\quad F\left( 1\right) =V,\quad F\left( a\otimes 1\right)
=F\left( a\right) \otimes V,\quad F\left( c_{1,1}\right) =c
\end{equation*}%
and%
\begin{eqnarray}
\varphi _{2}\left( 0,b\right) &=&l_{F\left( b\right) },  \label{form:ph1} \\
\varphi _{2}\left( a,0\right) &=&r_{F\left( a\right) },  \label{form:ph2} \\
\varphi _{2}\left( a,1\right) &=&\mathrm{Id}_{F\left( a\right) \otimes
V},\quad a\geq 1,  \label{form:ph3} \\
\varphi _{2}\left( a,b\otimes 1\right) &=&\left( \varphi _{2}\left(
a,b\right) \otimes V\right) \circ a_{F\left( a\right) ,F\left( b\right)
,V}^{-1},\quad a,b\geq 1,  \label{form:ph4} \\
\varphi _{0} &=&\mathrm{Id}_{\mathbf{1}}\text{.}  \label{form:ph5}
\end{eqnarray}%
Here $\mathcal{B}$ denotes the Braid Category, see \cite[page 321]{Kassel},
which is a strict monoidal category. Its braiding is denoted by $%
c_{m,n}:m\otimes n\rightarrow n\otimes m$. Moreover $F\left( n\right)
:=V^{\otimes n}$ for every $n\in
\mathbb{N}
$.
\end{lemma}

\begin{proof}
By \cite[Theorem XI.5.3, page 291]{Kassel}, there is a monoidal equivalence $%
F^{\prime }:\mathcal{M}^{\mathrm{str}}\rightarrow \mathcal{M}$, where $%
\mathcal{M}^{\mathrm{str}}$ is a suitable strict monoidal category. Recall
that objects in $\mathcal{M}^{\mathrm{str}}$ are all finite sequences $%
S=\left( V_{1},\ldots ,V_{k}\right) $ of objects of $\mathcal{M}$ including
the empty sequence $\emptyset $. The integer $k$ is by definition the length
of the sequence and is denoted by $l\left( S\right) $. This category is
strict monoidal with unit $\emptyset $ and tensor product given by%
\begin{eqnarray*}
\emptyset \otimes S &:&=S=:S\otimes \emptyset , \\
\left( V_{1},\ldots ,V_{k}\right) \otimes \left( V_{1}^{\prime },\ldots
,V_{s}^{\prime }\right) &:&=\left( V_{1},\ldots ,V_{k},V_{1}^{\prime
},\ldots ,V_{s}^{\prime }\right) .
\end{eqnarray*}%
To any sequence $S$ one assigns an object $F^{\prime }\left( S\right) $
defined recursively as follows: $F^{\prime }\left( \emptyset \right) :=%
\mathbf{1,}$ $F^{\prime }\left( \left( V_{1}\right) \right)
:=V_{1},F^{\prime }\left( \left( V_{1},\ldots ,V_{k},V_{k+1}\right) \right)
=F^{\prime }\left( \left( V_{1},\ldots ,V_{k}\right) \right) \otimes
V_{k+1}. $ The morphisms in $\mathcal{M}^{\mathrm{str}}$ are given by%
\begin{equation*}
\mathcal{M}^{\mathrm{str}}\left( S,S^{\prime }\right) :=\mathcal{M}\left(
F^{\prime }\left( S\right) ,F^{\prime }\left( S^{\prime }\right) \right) .
\end{equation*}%
Defining $F^{\prime }$ as the identity on morphisms, we get the functor $%
F^{\prime }:\mathcal{M}^{\mathrm{str}}\rightarrow \mathcal{M}$ . For
arbitrary objects $S,S^{\prime }\in \mathcal{M}^{\mathrm{str}}$ there is an
isomorphism $\varphi _{2}^{\prime }\left( S,S^{\prime }\right) :F^{\prime
}\left( S\right) \otimes F^{\prime }\left( S^{\prime }\right) \rightarrow
F^{\prime }\left( S\otimes S^{\prime }\right) $ defined iteratively as
follows%
\begin{eqnarray*}
\varphi _{2}^{\prime }\left( \emptyset ,S^{\prime }\right) &:&=l_{F^{\prime
}\left( S^{\prime }\right) },\quad \varphi _{2}^{\prime }\left( S,\emptyset
\right) :=r_{F^{\prime }\left( S\right) }, \\
\varphi _{2}^{\prime }\left( S,\left( Z\right) \right) &:&=\mathrm{Id}%
_{F^{\prime }\left( S\right) \otimes Z},l\left( S\right) \geq 1, \\
\varphi _{2}^{\prime }\left( S,S^{\prime }\otimes \left( Z\right) \right)
&:&=\left( \varphi _{2}^{\prime }\left( S,S^{\prime }\right) \otimes
Z\right) \circ a_{F^{\prime }\left( S\right) ,F^{\prime }\left( S^{\prime
}\right) ,Z}^{-1},l\left( S\right) \geq 1,l\left( S^{\prime }\right) \geq 1.
\end{eqnarray*}%
Define $\varphi _{0}^{\prime }=\mathrm{Id}_{\mathbf{1}}:\mathbf{1}%
\rightarrow F^{\prime }\left( \emptyset \right) .$ Then $\left( F^{\prime
},\varphi _{2}^{\prime },\varphi _{0}^{\prime }\right) $ is the claimed
monoidal functor. This comes out to be an equivalence. Its right adjoint $%
G^{\prime }:\mathcal{M}\rightarrow \mathcal{M}^{\mathrm{str}}$ is given by $%
G^{\prime }\left( Z\right) :=\left( Z\right) $ and is the identity on
morphisms too. Note that $F^{\prime }G^{\prime }=\mathrm{Id}_{\mathcal{M}}.$
The counit of this adjunction is the identity $\epsilon Z=\mathrm{Id}%
_{Z}:F^{\prime }G^{\prime }Z\rightarrow Z.$ The unit is $\eta S=\mathrm{Id}%
_{F^{\prime }S}:S\rightarrow G^{\prime }F^{\prime }S$. By \cite[Proposition
1.4.]{AMS-splitting}, we have that $\left( G^{\prime },\gamma _{2}^{\prime
},\gamma _{0}^{\prime }\right) $ is a monoidal functor where%
\begin{eqnarray*}
\gamma _{0}^{\prime } &:&=G^{\prime }\left( \left( \varphi _{0}^{\prime
}\right) ^{-1}\right) \circ \eta \emptyset :\emptyset \rightarrow G^{\prime }%
\mathbf{1.} \\
\gamma _{2}^{\prime }\left( X,Y\right) &:&=G^{\prime }\left( \epsilon
X\otimes \epsilon Y\right) \circ G^{\prime }\left( \left( \varphi
_{2}^{\prime }\right) ^{-1}\left( G^{\prime }X,G^{\prime }Y\right) \right)
\circ \eta \left( G^{\prime }X\otimes G^{\prime }Y\right) :G^{\prime
}X\otimes G^{\prime }Y\rightarrow G^{\prime }\left( X\otimes Y\right) .
\end{eqnarray*}%
Hence $\gamma _{0}^{\prime }=\mathrm{Id}_{\mathbf{1}}:\emptyset \rightarrow
\left( \mathbf{1}\right) $ and $\gamma _{2}^{\prime }\left( X,Y\right) =%
\mathrm{Id}_{X\otimes Y}:\left( X,Y\right) \rightarrow \left( X\otimes
Y\right) $ as%
\begin{equation*}
\left( \varphi _{2}^{\prime }\right) ^{-1}\left( G^{\prime }X,G^{\prime
}Y\right) =\left( \varphi _{2}^{\prime }\right) ^{-1}\left( \left( X\right)
,\left( Y\right) \right) =\mathrm{Id}_{F^{\prime }\left( \left( X\right)
\right) \otimes Y}=\mathrm{Id}_{X\otimes Y}.
\end{equation*}

Therefore we have a functor $\mathrm{Br}G^{\prime }:\mathrm{Br}_{\mathcal{M}%
}\rightarrow \mathrm{Br}_{\mathcal{M}^{\mathrm{str}}}.$ By construction%
\begin{equation*}
\mathrm{Br}G^{\prime }\left( V,c\right) =\left( G^{\prime }\left( V\right)
,\left( \gamma _{2}^{\prime }\right) ^{-1}\left( V,V\right) \circ G^{\prime
}c\circ \gamma _{2}^{\prime }\left( V,V\right) \right) =\left( \left(
V\right) ,c\right) .
\end{equation*}%
Thus $\left( \left( V\right) ,c\right) =\mathrm{Br}G^{\prime }\left(
V,c\right) \in \mathrm{Br}_{\mathcal{M}^{\mathrm{str}}}$. By \cite[Lemma
XII.3.5, page 327]{Kassel}, there is a unique strict monoidal functor $%
F^{\prime \prime }:\mathcal{B}\rightarrow \mathcal{M}^{\mathrm{str}}$ such
that $F^{\prime \prime }\left( 1\right) =\left( V\right) $ and $F^{\prime
\prime }\left( c_{1,1}\right) =c.$ Define $F:=F^{\prime }\circ F^{\prime
\prime }:\mathcal{B}\rightarrow \mathcal{M}$. Hence $F\left( 1\right)
=F^{\prime }F^{\prime \prime }\left( 1\right) =F^{\prime }\left( \left(
V\right) \right) =V$ and $F\left( c_{1,1}\right) =F^{\prime }F^{\prime
\prime }\left( c_{1,1}\right) =F^{\prime }\left( c\right) =c.$ Let us
compute the monoidal structure of $F$. By \cite[1.3]{AMS-splitting}, we have
that $\left( F,\varphi _{2},\varphi _{0}\right) $ is monoidal where%
\begin{eqnarray*}
\varphi _{2}\left( a,b\right) &:&=\varphi _{2}^{\prime }\left( F^{\prime
\prime }\left( a\right) ,F^{\prime \prime }\left( b\right) \right) :F\left(
a\right) \otimes F\left( b\right) \rightarrow F^{\prime }\left( F^{\prime
\prime }\left( a\right) \otimes F^{\prime \prime }\left( b\right) \right)
=F\left( a\otimes b\right) , \\
\varphi _{0} &=&\varphi _{0}^{\prime }:\mathbf{1}\rightarrow F^{\prime
}\left( \emptyset \right) =F\left( 0\right) .
\end{eqnarray*}

We get%
\begin{eqnarray*}
\varphi _{2}\left( 0,b\right) &=&\varphi _{2}^{\prime }\left( F^{\prime
\prime }\left( 0\right) ,F^{\prime \prime }\left( b\right) \right) =\varphi
_{2}^{\prime }\left( \emptyset ,F^{\prime \prime }\left( b\right) \right)
=l_{F\left( b\right) }, \\
\varphi _{2}\left( a,0\right) &=&\varphi _{2}^{\prime }\left( F^{\prime
\prime }\left( a\right) ,F^{\prime \prime }\left( 0\right) \right) =\varphi
_{2}^{\prime }\left( F^{\prime \prime }\left( a\right) ,\emptyset \right)
=r_{F\left( a\right) }, \\
\varphi _{2}\left( a,1\right) &=&\varphi _{2}^{\prime }\left( F^{\prime
\prime }\left( a\right) ,F^{\prime \prime }\left( 1\right) \right) =\varphi
_{2}^{\prime }\left( F^{\prime \prime }\left( a\right) ,\left( V\right)
\right) =\mathrm{Id}_{F\left( a\right) \otimes V}, \\
\varphi _{2}\left( a,b\otimes 1\right) &=&\varphi _{2}^{\prime }\left(
F^{\prime \prime }\left( a\right) ,F^{\prime \prime }\left( b\otimes
1\right) \right) =\varphi _{2}^{\prime }\left( F^{\prime \prime }\left(
a\right) ,F^{\prime \prime }\left( b\right) \otimes F^{\prime \prime }\left(
1\right) \right) \\
&=&\varphi _{2}^{\prime }\left( F^{\prime \prime }\left( a\right) ,F^{\prime
\prime }\left( b\right) \otimes \left( V\right) \right) =\left( \varphi
_{2}^{\prime }\left( F^{\prime \prime }\left( a\right) ,F^{\prime \prime
}\left( b\right) \right) \otimes V\right) \circ a_{F^{\prime }F^{\prime
\prime }\left( a\right) ,F^{\prime }F^{\prime \prime }\left( b\right)
,V}^{-1} \\
&=&\left( \varphi _{2}\left( a,b\right) \otimes V\right) \circ a_{F\left(
a\right) ,F\left( b\right) ,V}^{-1}.
\end{eqnarray*}

Thus (\ref{form:ph1}), (\ref{form:ph2}), (\ref{form:ph3}), (\ref{form:ph4})
and (\ref{form:ph5}) hold true for $\left( F,\varphi _{2},\varphi
_{0}\right) $.

Let $\left( \widetilde{F},\widetilde{\varphi }_{2},\widetilde{\varphi }%
_{0}\right) :\mathcal{B}\rightarrow \mathcal{M}$ be another monoidal functor
such that $\widetilde{F}\left( 0\right) =\mathbf{1},\widetilde{F}\left(
1\right) =V,\widetilde{F}\left( a\otimes 1\right) =\widetilde{F}\left(
a\right) \otimes V,$ $\widetilde{F}\left( c_{1,1}\right) =c$ and the
analogue equations (\ref{form:ph1}), (\ref{form:ph2}), (\ref{form:ph3}), (%
\ref{form:ph4}) and (\ref{form:ph5}) hold true.

In order to prove that $\left( F,\varphi _{2},\varphi _{0}\right) =\left(
\widetilde{F},\widetilde{\varphi }_{2},\widetilde{\varphi }_{0}\right) $ it
suffices to check that $F\left( a\right) =\widetilde{F}\left( a\right) $ for
every $a\in
\mathbb{N}
$ (in fact the constraints are then uniquely determined by the equalities
they fulfil).

Let us check that $\widetilde{F}\left( n\right) =V^{\otimes n}$ for $n\geq 0$%
. We proceed by induction on $n$. For $n=0,$ by assumption we have $%
\widetilde{F}\left( 0\right) =\mathbf{1}$. Let $n>0$ and assume that $%
\widetilde{F}\left( n-1\right) :=V^{\otimes \left( n-1\right) }.$ Then%
\begin{equation*}
\widetilde{F}\left( n\right) =\widetilde{F}\left( \left( n-1\right) \otimes
1\right) =\widetilde{F}\left( n-1\right) \otimes V=V^{\otimes \left(
n-1\right) }\otimes V=V^{\otimes n}.
\end{equation*}%
Similarly $F\left( n\right) =V^{\otimes n}$ for $n\geq 0$.
\end{proof}

\begin{proposition}
\label{pro:CT}Let $\left( \mathcal{M},\otimes ,\mathbf{1}\right) $ be a
monoidal category and let $\left( V,c\right) \in \mathrm{Br}_{\mathcal{M}}$.
Then there is a unique morphism
\begin{equation*}
c_{T}^{a,b}:V^{\otimes a}\otimes V^{\otimes b}\rightarrow V^{\otimes
b}\otimes V^{\otimes a}
\end{equation*}%
such that for all $l,m,n\in \mathbb{N}$ {\small
\begin{gather}
\left( V^{\otimes n}\otimes c_{T}^{l,m}\right) \circ \left(
c_{T}^{l,n}\otimes V^{\otimes m}\right) \circ \left( V^{\otimes l}\otimes
c_{T}^{m,n}\right) =\left( c_{T}^{m,n}\otimes V^{\otimes l}\right) \circ
\left( V^{\otimes m}\otimes c_{T}^{l,n}\right) \circ \left(
c_{T}^{l,m}\otimes V^{\otimes n}\right) ,  \label{form:ct1} \\
\left( c_{T}^{l,n}\otimes V^{\otimes m}\right) \circ \left( V^{\otimes
l}\otimes c_{T}^{m,n}\right) =c_{T}^{l+m,n},\quad l\neq 0,m\neq 0,
\label{form:ct2} \\
\left( V^{\otimes m}\otimes c_{T}^{l,n}\right) \circ \left(
c_{T}^{l,m}\otimes V^{\otimes n}\right) =c_{T}^{l,m+n},\quad m\neq 0,n\neq 0,
\label{form:ct3} \\
c_{T}^{0,n}\circ l_{V^{\otimes n}}^{-1}=r_{V^{\otimes n}}^{-1},
\label{form:ct4} \\
c_{T}^{n,0}\circ r_{V^{\otimes n}}^{-1}=l_{V^{\otimes n}}^{-1},
\label{form:ct5} \\
c_{T}^{1,1}=c.  \label{form:ct6}
\end{gather}%
}
\end{proposition}

\begin{proof}
Consider the monoidal functor $\left( F,\varphi _{2},\varphi _{0}\right) :%
\mathcal{B}\rightarrow \mathcal{M}$ of Lemma \ref{lem:Kassel}. Consider for $%
a,b\in
\mathbb{N}
,$ the isomorphism $\varphi _{2}\left( a,b\right) :F\left( a\right) \otimes
F\left( b\right) \rightarrow F\left( a+b\right) $ where $F\left( n\right)
=V^{\otimes n}$ for every $n\in
\mathbb{N}
$. Set%
\begin{equation*}
c_{T}^{a,b}:=\varphi _{2}\left( b,a\right) ^{-1}\circ F\left( c_{a,b}\right)
\circ \varphi _{2}\left( a,b\right) :V^{\otimes a}\otimes V^{\otimes
b}\rightarrow V^{\otimes b}\otimes V^{\otimes a}.
\end{equation*}%
Thus%
\begin{equation}
\varphi _{2}\left( b,a\right) \circ c_{T}^{a,b}=F\left( c_{a,b}\right) \circ
\varphi _{2}\left( a,b\right) .  \label{form:cTdef}
\end{equation}%
Note that, since $\varphi _{2}\left( 1,1\right) =\mathrm{Id}_{V\otimes V},$
we get $c_{T}^{1,1}=F\left( c_{1,1}\right) =c$ so that (\ref{form:ct6})
holds. Since $c_{m,n}:m\otimes n\rightarrow n\otimes m$ is the braiding in $%
\mathcal{B},$ we have that
\begin{equation}
\left( n\otimes c_{l,m}\right) \circ \left( c_{l,n}\otimes m\right) \circ
\left( l\otimes c_{m,n}\right) =\left( c_{m,n}\otimes l\right) \circ \left(
m\otimes c_{l,n}\right) \circ \left( c_{l,m}\otimes n\right) .
\label{form:braid1}
\end{equation}%
Thus, {\small
\begin{eqnarray*}
&&\left[
\begin{array}{c}
\varphi _{2}\left( n,m\otimes l\right) \circ \left( F\left( n\right) \otimes
\varphi _{2}\left( m,l\right) \right) \circ \left( F\left( n\right) \otimes
c_{T}^{l,m}\right) \circ a_{F\left( n\right) ,F\left( l\right) ,F\left(
m\right) } \\
\circ \left( c_{T}^{l,n}\otimes F\left( m\right) \right) \circ a_{F\left(
l\right) ,F\left( n\right) ,F\left( m\right) }^{-1}\circ \left( F\left(
l\right) \otimes c_{T}^{m,n}\right)
\end{array}%
\right]  \\
&&\overset{(\ast )}{=}F\left( n\otimes c_{l,m}\right) \circ F\left(
c_{l,n}\otimes m\right) \circ F\left( l\otimes c_{m,n}\right) \circ \varphi
_{2}\left( l,m\otimes n\right) \circ \left( F\left( l\right) \otimes \varphi
_{2}\left( m,n\right) \right)  \\
&&\overset{(\ref{form:braid1})}{=}F\left( c_{m,n}\otimes l\right) \circ
F\left( m\otimes c_{l,n}\right) \circ F\left( c_{l,m}\otimes n\right) \circ
\varphi _{2}\left( l,m\otimes n\right) \circ \left( F\left( l\right) \otimes
\varphi _{2}\left( m,n\right) \right)  \\
&&\overset{(\ast \ast )}{=}\left[
\begin{array}{c}
\varphi _{2}\left( n,m\otimes l\right) \circ \left( F\left( n\right) \otimes
\varphi _{2}\left( m,l\right) \right) \circ a_{F\left( n\right) ,F\left(
m\right) ,F\left( l\right) }\circ \left( c_{T}^{m,n}\otimes F\left( l\right)
\right) \circ a_{F\left( m\right) ,F\left( n\right) ,F\left( l\right) }^{-1}
\\
\circ \left( F\left( m\right) \otimes c_{T}^{l,n}\right) \circ a_{F\left(
m\right) ,F\left( l\right) ,F\left( n\right) }\circ \left(
c_{T}^{l,m}\otimes F\left( n\right) \right) \circ a_{F\left( l\right)
,F\left( m\right) ,F\left( n\right) }^{-1}%
\end{array}%
\right]
\end{eqnarray*}%
}

where in (*) we used in the order (\ref{form:cTdef}), the naturality of $%
\varphi _{2},$ the monoidality of $F,$ (\ref{form:cTdef}), the naturality of
$\varphi _{2},$ the monoidality of $F,$ (\ref{form:cTdef}) and the
naturality of $\varphi _{2},$ while in (**) we used in the order the
monoidality of $F$, the naturality of $\varphi _{2}$, (\ref{form:cTdef}),
the monoidality of $F,$ repeated three times.

Since $\varphi _{2}\left( n,m\otimes l\right) \circ \left( F\left( n\right)
\otimes \varphi _{2}\left( m,l\right) \right) $ is an isomorphism, from the
computation above we deduce {\small
\begin{eqnarray*}
&&\left( F\left( n\right) \otimes c_{T}^{l,m}\right) \circ a_{F\left(
n\right) ,F\left( l\right) ,F\left( m\right) }\circ \left(
c_{T}^{l,n}\otimes F\left( m\right) \right) \circ a_{F\left( l\right)
,F\left( n\right) ,F\left( m\right) }^{-1}\circ \left( F\left( l\right)
\otimes c_{T}^{m,n}\right)  \\
&=&\left[
\begin{array}{c}
a_{F\left( n\right) ,F\left( m\right) ,F\left( l\right) }\circ \left(
c_{T}^{m,n}\otimes F\left( l\right) \right) \circ a_{F\left( m\right)
,F\left( n\right) ,F\left( l\right) }^{-1}\circ \left( F\left( m\right)
\otimes c_{T}^{l,n}\right)  \\
\circ a_{F\left( m\right) ,F\left( l\right) ,F\left( n\right) }\circ \left(
c_{T}^{l,m}\otimes F\left( n\right) \right) \circ a_{F\left( l\right)
,F\left( m\right) ,F\left( n\right) }^{-1}%
\end{array}%
\right] .
\end{eqnarray*}%
}

This is (\ref{form:ct1}) with all the constraints. Since $c_{-,-}$ is a
braiding, we have%
\begin{equation}
\left( c_{l,n}\otimes m\right) \circ \left( l\otimes c_{m,n}\right)
=c_{l+m,n}.  \label{form:braid2}
\end{equation}%
We compute%
\begin{gather*}
\varphi _{2}\left( n,l\otimes m\right) \circ \left( F\left( n\right) \otimes
\varphi _{2}\left( l,m\right) \right) \circ a_{F\left( n\right) ,F\left(
l\right) ,F\left( m\right) }\circ \left( c_{T}^{l,n}\otimes F\left( m\right)
\right) \circ a_{F\left( l\right) ,F\left( n\right) ,F\left( m\right)
}^{-1}\circ \left( F\left( l\right) \otimes c_{T}^{m,n}\right) \\
\overset{(\bullet )}{=}F\left( c_{l,n}\otimes m\right) \circ F\left(
l\otimes c_{m,n}\right) \circ \varphi _{2}\left( l,m\otimes n\right) \circ
\left( F\left( l\right) \otimes \varphi _{2}\left( m,n\right) \right) \\
\overset{(\ref{form:braid2})}{=}F\left( c_{l+m,n}\right) \circ \varphi
_{2}\left( l,m\otimes n\right) \circ \left( F\left( l\right) \otimes \varphi
_{2}\left( m,n\right) \right) \\
\overset{(\bullet \bullet )}{=}\varphi _{2}\left( n,l\otimes m\right) \circ
c_{T}^{l+m,n}\circ \left( \varphi _{2}\left( l,m\right) \otimes F\left(
n\right) \right) \circ a_{F\left( l\right) ,F\left( m\right) ,F\left(
n\right) }^{-1}.
\end{gather*}

where in $(\bullet )$ we used monoidality of $F$, (\ref{form:cTdef}),
naturality of $\varphi _{2}$, repeated twice, while in $(\bullet \bullet )$
we used monoidality of $F$ and (\ref{form:cTdef}). Since $\varphi _{2}\left(
n,l\otimes m\right) $ is an isomorphism, we obtain%
\begin{eqnarray*}
&&\left( F\left( n\right) \otimes \varphi _{2}\left( l,m\right) \right)
\circ a_{F\left( n\right) ,F\left( l\right) ,F\left( m\right) }\circ \left(
c_{T}^{l,n}\otimes F\left( m\right) \right) \circ a_{F\left( l\right)
,F\left( n\right) ,F\left( m\right) }^{-1}\circ \left( F\left( l\right)
\otimes c_{T}^{m,n}\right) \\
&=&c_{T}^{l+m,n}\circ \left( \varphi _{2}\left( l,m\right) \otimes F\left(
n\right) \right) \circ a_{F\left( l\right) ,F\left( m\right) ,F\left(
n\right) }^{-1}.
\end{eqnarray*}%
If $l\neq 0$ and $m\neq 0$ this formula is (\ref{form:ct2}) with all the
constraints.

Equation (\ref{form:ct3}) follows analogously. Since $c_{-,-}$ is a
braiding, we have%
\begin{equation}
c_{0,n}\circ l_{n}^{-1}=r_{n}^{-1}.  \label{form:braid3}
\end{equation}%
We get%
\begin{eqnarray*}
&&\varphi _{2}\left( n,0\right) \circ c_{T}^{0,n}\circ \left( \varphi
_{0}\otimes F\left( n\right) \right) \circ l_{F\left( n\right) }^{-1}\overset%
{(\ref{form:cTdef})}{=}F\left( c_{0,n}\right) \circ \varphi _{2}\left(
0,n\right) \circ \left( \varphi _{0}\otimes F\left( n\right) \right) \circ
l_{F\left( n\right) }^{-1} \\
&=&F\left( c_{0,n}\right) \circ F\left( l_{n}^{-1}\right) \overset{(\ref%
{form:braid3})}{=}F\left( r_{n}^{-1}\right) =\varphi _{2}\left( n,0\right)
\circ \left( F\left( n\right) \otimes \varphi _{0}\right) \circ r_{F\left(
n\right) }^{-1}
\end{eqnarray*}%
so that%
\begin{equation*}
c_{T}^{0,n}\circ \left( \varphi _{0}\otimes F\left( n\right) \right) \circ
l_{F\left( n\right) }^{-1}=\left( F\left( n\right) \otimes \varphi
_{0}\right) \circ r_{F\left( n\right) }^{-1}.
\end{equation*}%
This is (\ref{form:ct4}) with all the constraints. Equation (\ref{form:ct5})
follows analogously.

We now deal with uniqueness. Suppose there exists another $c_{T^{\prime
}}^{a,b}:V^{\otimes a}\otimes V^{\otimes b}\rightarrow V^{\otimes b}\otimes
V^{\otimes a}$ that fulfills the analogue of the equalities that $%
c_{T}^{a,b}\ $does for all $a,b\in\mathbb{N}$. Since $c_{T^{\prime }}^{a,b}$ fulfills the analogue of (\ref{form:ct2})
we have%
\begin{equation*}
\left( c_{T^{\prime }}^{l,n}\otimes V^{\otimes m}\right) \circ \left(
V^{\otimes l}\otimes c_{T^{\prime }}^{m,n}\right) =c_{T^{\prime }}^{l+m,n},%
\text{ for all }l,m,n\in\mathbb{N},l\neq 0,m\neq 0.
\end{equation*}%
For $l=1$ we get%
\begin{equation*}
\left( c_{T^{\prime }}^{1,n}\otimes V^{\otimes m}\right) \circ \left(
V\otimes c_{T^{\prime }}^{m,n}\right) =c_{T^{\prime }}^{1+m,n}\text{, for all
}l,m,n\in\mathbb{N},m\neq 0\text{.}
\end{equation*}%
Hence, an induction process tells that $c_{T^{\prime }}^{a,b}$ is uniquely
determined by $c_{T}^{1,n}$ and $c_{T}^{0,n}$ for $n\in\mathbb{N}$. Since $c_{T}^{a,b}$ fulfills the same equalities, in order to prove $%
c_{T^{\prime }}^{a,b}=c_{T}^{a,b}$ it suffices to check that this is true
for $a=0,1$. Analogously, using the analogue of (\ref{form:ct3}) we can
further reduce to the case $a,b=0,1.$ The equality in these cases follows by
(\ref{form:ct4}), (\ref{form:ct5}), (\ref{form:ct6}) and their analogous.
\end{proof}

\section{Braided Adjunctions}

\begin{proposition}
\label{pro:TbrStrict}Let $\left( \mathcal{M},\otimes ,\mathbf{1}\right) $ be
a monoidal category. Consider the category $\mathrm{BrAlg}_{\mathcal{M}}$ of
braided algebras in $\mathcal{M}$ and their morphisms. Assume that $\mathcal{%
M}$ has denumerable coproducts and that the tensor functors preserve such
coproducts for every object $M$ in $\mathcal{M}$. Denote by $\Omega _{%
\mathrm{Br}}:\mathrm{BrAlg}_{\mathcal{M}}\rightarrow \mathrm{Br}_{\mathcal{M}%
},H:\mathrm{Br}_{\mathcal{M}}\rightarrow \mathcal{M}$ and $H_{\mathrm{Alg}}:%
\mathrm{BrAlg}_{\mathcal{M}}\rightarrow \mathrm{Alg}_{\mathcal{M}}$ the
obvious forgetful functors. Then the functor $\Omega _{\mathrm{Br}}:\mathrm{%
BrAlg}_{\mathcal{M}}\rightarrow \mathrm{Br}_{\mathcal{M}}$ has a left
adjoint
\begin{equation*}
T_{\mathrm{Br}}:\mathrm{Br}_{\mathcal{M}}\rightarrow \mathrm{BrAlg}_{%
\mathcal{M}}.
\end{equation*}%
Given $\left( V,c\right) \in \mathrm{Br}_{\mathcal{M}}$ one has that $T_{%
\mathrm{Br}}\left( V,c\right) =\left( TV,c_{T}\right) $ where $%
c_{T}:=c_{\Omega TV}$ is uniquely determined by
\begin{equation}
c_{T}\circ \left( \alpha _{m}V\otimes \alpha _{n}V\right) =\left( \alpha
_{n}V\otimes \alpha _{m}V\right) \circ c_{T}^{m,n}  \label{form:defcT}
\end{equation}%
and $c_{T}^{m,n}$ are the morphisms of Proposition \ref{pro:CT}. For $f\ $a
morphism in $\mathrm{Br}_{\mathcal{M}}$, one has $T_{\mathrm{Br}}\left(
f\right) :=T\left( f\right) .$

The unit $\eta _{\mathrm{Br}}$ and the counit $\epsilon _{\mathrm{Br}}$ are
uniquely determined by the following equations
\begin{equation}
H\eta _{\mathrm{Br}}=\eta H,\qquad H_{\mathrm{Alg}}\epsilon _{\mathrm{Br}%
}=\epsilon H_{\mathrm{Alg}},  \label{form:TbrStrict}
\end{equation}%
where $\eta $ and $\epsilon $ denote the unit and counit of the adjunction $%
\left( T,\Omega \right) $ of Remark \ref{cl: AlgMon}. Moreover the following
diagrams commute.%
\begin{equation}
\xymatrixrowsep{15pt} \xymatrixcolsep{25pt}\xymatrix{\mathrm{BrAlg}_{\mathcal{M}}
\ar[r]^{H_\mathrm{Alg}}\ar[d]_{\Omega_{\mathrm{Br}}}&\mathrm{Alg}_{\mathcal{M}}\ar[d]^{\Omega}\\
\mathrm{Br}_{\mathcal{M}} \ar[r]^H&\mathcal{M}}
\qquad
\xymatrixrowsep{15pt} \xymatrixcolsep{25pt}\xymatrix{\mathrm{BrAlg}_{\mathcal{M}}
\ar[r]^{H_\mathrm{Alg}}&\mathrm{Alg}_{\mathcal{M}}\\
\mathrm{Br}_{\mathcal{M}} \ar[u]^{T_{\mathrm{Br}}}\ar[r]^H&\mathcal{M}\ar[u]_T}.  \label{form:HtildeTbrOmegaBr}
\end{equation}
\end{proposition}

\begin{proof}
Let $\left( V,c\right) \in \mathrm{Br}_{\mathcal{M}}.$ By Proposition \ref%
{pro:CT} we can consider, for $m,n\in
\mathbb{N}
$, the morphisms $c_{T}^{m,n}:V^{\otimes m}\otimes V^{\otimes n}\rightarrow
V^{\otimes n}\otimes V^{\otimes m}.$ By Remark \ref{cl: AlgMon}, we can
consider the tensor algebra $TV\in \mathrm{Alg}_{\mathcal{M}}.$ Let us
define a braiding $c_{T}$ on $T=\Omega TV$ using $c_{T}^{m,n}$. Let $\alpha
_{n}V:V^{\otimes n}\rightarrow T$ be the canonical morphism. Since the
tensor functors preserves denumerable coproducts, there is a unique morphism
$c_{T}:T\otimes T\rightarrow T\otimes T$ such that (\ref{form:defcT}).

Let us check that $\left( T,m_{T},u_{T},c_{T}\right) $ is a braided algebra.
We know that $\left( T,m_{T},u_{T}\right) =TV$ is an algebra. We compute%
\begin{eqnarray*}
&&\left( T\otimes c_{T}\right) \circ \left( c_{T}\otimes T\right) \circ
\left( T\otimes c_{T}\right) \circ \left( \alpha _{l}V\otimes \alpha
_{m}V\otimes \alpha _{n}V\right) \\
&&\overset{(\ref{form:defcT})}{=}\left( \alpha _{n}V\otimes \alpha
_{m}V\otimes \alpha _{l}V\right) \circ \left( V^{\otimes n}\otimes
c_{T}^{l,m}\right) \circ \left( c_{T}^{l,n}\otimes V^{\otimes m}\right)
\circ \left( V^{\otimes l}\otimes c_{T}^{m,n}\right) \\
&&\overset{(\ref{form:ct1})}{=}\left( \alpha _{n}V\otimes \alpha
_{m}V\otimes \alpha _{l}V\right) \circ \left( c_{T}^{m,n}\otimes V^{\otimes
l}\right) \circ \left( V^{\otimes m}\otimes c_{T}^{l,n}\right) \circ \left(
c_{T}^{l,m}\otimes V^{\otimes n}\right) \\
&&\overset{(\ref{form:defcT})}{=}\left( c_{T}\otimes T\right) \circ \left(
T\otimes c_{T}\right) \circ \left( c_{T}\otimes T\right) \circ \left( \alpha
_{l}V\otimes \alpha _{m}V\otimes \alpha _{n}V\right) .
\end{eqnarray*}%
By arbitrariness of $l,m$ and $n$ we obtain that $c_{T}$ is a braiding i.e.
that $\left( T,c_{T}\right) $ is a braided object. We have%
\begin{eqnarray*}
&&(T\otimes m_{T})\circ (c_{T}\otimes T)\circ (T\otimes c_{T})\circ \left(
\alpha _{l}V\otimes \alpha _{m}V\otimes \alpha V_{n}\right) \\
&&\overset{(\ref{form:defcT})}{=}(T\otimes m_{T})\circ \left( \alpha
_{n}V\otimes \alpha _{l}V\otimes \alpha _{m}V\right) \circ \left(
c_{T}^{l,n}\otimes V^{\otimes m}\right) \circ \left( V^{\otimes l}\otimes
c_{T}^{m,n}\right) \\
&&\overset{(\ref{form:TVm})}{=}\left( \alpha _{n}V\otimes \alpha
_{l+m}V\right) \circ \left( c_{T}^{l,n}\otimes V^{\otimes m}\right) \circ
\left( V^{\otimes l}\otimes c_{T}^{m,n}\right) \\
&&\overset{(\ref{form:ct2})}{=}\left( \alpha _{n}V\otimes \alpha
_{l+m}V\right) \circ c_{T}^{l+m,n}\overset{(\ref{form:defcT})}{=}c_{T}\circ
\left( \alpha _{l+m}V\otimes \alpha _{n}V\right) \\
&&\overset{(\ref{form:TVm})}{=}c_{T}\circ (m_{T}\otimes T)\circ \left(
\alpha _{l}V\otimes \alpha _{m}V\otimes \alpha _{n}V\right) .
\end{eqnarray*}%
By arbitrariness of $l,m$ and $n$ we obtain that $(T\otimes m_{T})\circ
(c_{T}\otimes T)\circ (T\otimes c_{T})=c_{T}\circ (m_{T}\otimes T).$
Similarly, using (\ref{form:ct3}), one gets $(m_{T}\otimes T)\circ (T\otimes
c_{T})\circ (c_{T}\otimes T)=c_{T}\circ (T\otimes m_{T}).$

We have%
\begin{eqnarray*}
c_{T}\circ (u_{T}\otimes T)\circ l_{T}^{-1}\circ \alpha V_{n} &=&c_{T}\circ
(u_{T}\otimes \alpha V_{n})\circ l_{V^{\otimes n}}^{-1}\overset{(\ref%
{form:TVu})}{=}c_{T}\circ (\alpha V_{0}\otimes \alpha V_{n})\circ
l_{V^{\otimes n}}^{-1} \\
&\overset{(\ref{form:defcT})}{=}&(\alpha V_{n}\otimes \alpha V_{0})\circ
c_{T}^{0,n}\circ l_{V^{\otimes n}}^{-1}\overset{(\ref{form:ct4})}{=}\left(
\alpha V_{n}\otimes u_{T}\right) \circ r_{V^{\otimes n}}^{-1} \\
&=&\left( T\otimes u_{T}\right) \circ \left( \alpha V_{n}\otimes \mathbf{1}%
\right) \circ r_{V^{\otimes n}}^{-1}=\left( T\otimes u_{T}\right) \circ
r_{T}^{-1}\circ \alpha V_{n}.
\end{eqnarray*}%
By arbitrariness of $n$ we obtain that $c_{T}\circ (u_{T}\otimes T)\circ
l_{T}^{-1}=\left( T\otimes u_{T}\right) \circ r_{T}^{-1}.$ Similarly, using (%
\ref{form:ct5}), one gets $c_{T}\circ (T\otimes u_{T})\circ
r_{T}^{-1}=\left( u_{T}\otimes T\right) \circ l_{T}^{-1}.$ We have so proved
that $\left( T,m_{T},u_{T},c_{T}\right) $ is a braided algebra in $\mathcal{M%
}$. Define $T_{\mathrm{Br}}\left( V,c\right) $ to be this braided algebra in
$\mathcal{M}$.

Let $f:\left( V,c\right) \rightarrow \left( V^{\prime },c^{\prime }\right) $
be a morphism of braided objects. In particular $f:V\rightarrow V^{\prime }$
is a morphism in $\mathcal{M}$ so that we can consider the algebra
homomorphism $T\left( f\right) :\left( T,m_{T},u_{T}\right) \rightarrow
\left( T^{\prime },m_{T^{\prime }},u_{T^{\prime }}\right) $. Let us check by
induction on $m\in
\mathbb{N}
$ that
\begin{equation}
\left( f^{\otimes n}\otimes f^{\otimes m}\right) \circ
c_{T}^{m,n}=c_{T^{\prime }}^{m,n}\circ \left( f^{\otimes m}\otimes
f^{\otimes n}\right) .  \label{form:fcT}
\end{equation}%
For $m=0$ and $n\in
\mathbb{N}
$ we have%
\begin{eqnarray*}
\left( f^{\otimes n}\otimes f^{\otimes m}\right) \circ c_{T}^{m,n} &=&\left(
f^{\otimes n}\otimes f^{\otimes 0}\right) \circ c_{T}^{0,n}\overset{(\ref%
{form:ct4})}{=}\left( f^{\otimes n}\otimes \mathbf{1}\right) \circ
r_{V^{\otimes n}}^{-1}\circ l_{V^{\otimes n}}=r_{\left( V^{\prime }\right)
^{\otimes n}}^{-1}\circ f^{\otimes n}\circ l_{V^{\otimes n}} \\
&=&r_{\left( V^{\prime }\right) ^{\otimes n}}^{-1}\circ l_{\left( V^{\prime
}\right) ^{\otimes n}}\circ \left( \mathbf{1}\otimes f^{\otimes n}\right)
\overset{(\ref{form:ct4})}{=}c_{T^{\prime }}^{0,n}\circ \left( f^{\otimes
0}\otimes f^{\otimes n}\right) =c_{T^{\prime }}^{m,n}\circ \left( f^{\otimes
m}\otimes f^{\otimes n}\right) .
\end{eqnarray*}%
For $m=1.$ For $n=0,$ it follows in a similar way. For $n=1,$we have%
\begin{eqnarray*}
\left( f^{\otimes n}\otimes f^{\otimes m}\right) \circ c_{T}^{m,n} &=&\left(
f\otimes f\right) \circ c_{T}^{1,1}\overset{(\ref{form:ct6})}{=}\left(
f\otimes f\right) \circ c \\
&=&c^{\prime }\circ \left( f\otimes f\right) \overset{(\ref{form:ct6})}{=}%
c_{T^{\prime }}^{1,1}\circ \left( f\otimes f\right) =c_{T^{\prime
}}^{m,n}\circ \left( f^{\otimes m}\otimes f^{\otimes n}\right) .
\end{eqnarray*}

Assume that the formula holds for $n\geq 1$and let us check it for $n+1.$
Using (\ref{form:ct3}), (\ref{form:ct6}), the fact that $f$ is braided, (\ref%
{form:ct6}) and (\ref{form:ct3}) we get $\left( f^{\otimes n+1}\otimes
f\right) \circ c_{T}^{1,n+1}=c_{T^{\prime }}^{1,n+1}\circ \left( f\otimes
f^{\otimes n+1}\right) $. We have so proved that the statement holds for $%
m=1 $ and $n\in
\mathbb{N}
\,.$

For $m\in
\mathbb{N}
$ and $n=0$ the formula holds in analogy to the first case we considered
above.

Assume that the equation holds for $m\geq 1$ and $n\geq 1.$ Then the formula
holds for $\left( m+1,n\right) $ by means of (\ref{form:ct2}), induction
hypothesis and (\ref{form:ct2}). Thus the formula is proved for all $m,n\in
\mathbb{N}
$.

Now, using in this order (\ref{form:defcT}), (\ref{form:Tf}), (\ref{form:fcT}%
), (\ref{form:defcT}) and (\ref{form:Tf}) we get

$\left( \Omega T\left( f\right) \otimes \Omega T\left( f\right) \right)
\circ c_{T}\circ \left( \alpha _{m}V\otimes \alpha _{n}V\right) =c_{T}\circ
\left( \Omega T\left( f\right) \otimes \Omega T\left( f\right) \right) \circ
\left( \alpha _{m}V\otimes \alpha _{n}V\right) $. By arbitrariness of $m,n$
we obtain that $\left( \Omega T\left( f\right) \otimes \Omega T\left(
f\right) \right) \circ c_{T}=c_{T}\circ \left( \Omega T\left( f\right)
\otimes \Omega T\left( f\right) \right) $ so that $T\left( f\right) $ yields
a morphism of braided algebras in $\mathcal{M}$ that we denote by $T_{%
\mathrm{Br}}\left( f\right) $. Thus we have defined a functor $T_{\mathrm{Br}%
}:\mathrm{Br}_{\mathcal{M}}\rightarrow \mathrm{BrAlg}_{\mathcal{M}}$. It is
clear, from the construction that diagrams \eqref{form:HtildeTbrOmegaBr} commute.
Let us check that $\left( T_{\mathrm{Br}},\Omega _{\mathrm{Br}}\right) $ is
an adjunction. For $\left( V,c\right) \in \mathrm{Br}_{\mathcal{M}},$%
\begin{equation*}
c_{T}\circ \left( \eta V\otimes \eta V\right) \overset{(\ref{form:etaeps})}{=%
}c_{T}\circ \left( \alpha _{1}V\otimes \alpha _{1}V\right) \overset{(\ref%
{form:defcT})}{=}\left( \alpha _{1}V\otimes \alpha _{1}V\right) \circ
c_{T}^{1,1}\overset{(\ref{form:ct6}),(\ref{form:etaeps})}{=}\left( \eta
V\otimes \eta V\right) \circ c.
\end{equation*}%
Thus $\eta V:V\rightarrow \Omega TV$ defines a morphism $\eta _{\mathrm{Br}%
}\left( V,c\right) :\left( V,c\right) \rightarrow \left( \Omega _{\mathrm{Br}%
}\circ T_{\mathrm{Br}}\right) \left( V,c\right) $ such that $H\eta _{\mathrm{%
Br}}\left( V,c\right) =\eta V.$ Since $\eta V$ is natural in $V$ we get that
$\eta _{\mathrm{Br}}\left( V,c\right) $ is natural in $\left( V,c\right) $
so that we get a natural transformation $\eta _{\mathrm{Br}}:\mathrm{Id}_{%
\mathrm{Br}_{\mathcal{M}}}\rightarrow \Omega _{\mathrm{Br}}\circ T_{\mathrm{%
Br}}.$ Let $\left( A,m_{A},u_{A},c_{A}\right) \in \mathrm{BrAlg}_{\mathcal{M}%
}$. Then $\epsilon \left( A,m_{A},u_{A}\right) :T\Omega \left(
A,m_{A},u_{A}\right) \rightarrow \left( A,m_{A},u_{A}\right) $ is an algebra
map. Let us check it commutes with braidings:%
\begin{eqnarray*}
&&\left( \Omega \epsilon \left( A,m_{A},u_{A}\right) \otimes \Omega \epsilon
\left( A,m_{A},u_{A}\right) \right) \circ c_{\Omega T\Omega \left(
A,m_{A},u_{A}\right) }\circ \left( \alpha _{m}A\otimes \alpha _{n}A\right) \\
&=&\left( \Omega \epsilon \left( A,m_{A},u_{A}\right) \otimes \Omega
\epsilon \left( A,m_{A},u_{A}\right) \right) \circ c_{\Omega TA}\circ \left(
\alpha _{m}A\otimes \alpha _{n}A\right) \\
&\overset{(\ref{form:defcT})}{=}&\left( \Omega \epsilon \left(
A,m_{A},u_{A}\right) \otimes \Omega \epsilon \left( A,m_{A},u_{A}\right)
\right) \circ \left( \alpha _{n}A\otimes \alpha _{m}A\right) \circ c_{\Omega
TA}^{m,n} \\
&\overset{(\ref{form:etaeps})}{=}&\left( m_{A}^{n-1}\otimes
m_{A}^{m-1}\right) \circ c_{\Omega TA}^{m,n}\overset{(\ast )}{=}c_{A}\circ
\left( m_{A}^{m-1}\otimes m_{A}^{n-1}\right) \\
&\overset{(\ref{form:etaeps})}{=}&c_{A}\circ \left( \Omega \epsilon \left(
A,m_{A},u_{A}\right) \otimes \Omega \epsilon \left( A,m_{A},u_{A}\right)
\right) \circ \left( \alpha A_{m}\otimes \alpha A_{n}\right) .
\end{eqnarray*}%
Let us prove (*) by induction. For $m=0$ and $n\in
\mathbb{N}
$ we have%
\begin{eqnarray*}
\left( m_{A}^{n-1}\otimes m_{A}^{m-1}\right) \circ c_{\Omega TA}^{m,n}
&=&\left( m_{A}^{n-1}\otimes m_{A}^{-1}\right) \circ c_{\Omega TA}^{0,n}%
\overset{(\ref{form:ct4})}{=}\left( m_{A}^{n-1}\otimes u_{A}\right) \circ
r_{A^{\otimes n}}^{-1}\circ l_{A^{\otimes n}} \\
&=&\left( A\otimes u_{A}\right) \circ \left( m_{A}^{n-1}\otimes \mathbf{1}%
\right) \circ r_{A^{\otimes n}}^{-1}\circ l_{A^{\otimes n}}=\left( A\otimes
u_{A}\right) \circ r_{A}^{-1}\circ m_{A}^{n-1}\circ l_{A^{\otimes n}} \\
&=&\left( A\otimes u_{A}\right) \circ r_{A}^{-1}\circ l_{A}\circ \left(
\mathbf{1}\otimes m_{A}^{n-1}\right) \overset{(\ref{Br4})}{=}c_{A}\circ
\left( u_{A}\otimes A\right) \circ \left( \mathbf{1}\otimes
m_{A}^{n-1}\right) \\
&=&c_{A}\circ \left( u_{A}\otimes m_{A}^{n-1}\right) =c_{A}\circ \left(
m_{A}^{-1}\otimes m_{A}^{n-1}\right) =c_{A}\circ \left( m_{A}^{m-1}\otimes
m_{A}^{n-1}\right) .
\end{eqnarray*}%
For $n=0$ and $m\in
\mathbb{N}
$, the proof is similar. In particular we get the case $m=1\ $and $n=0$. For
$m=n=1,$we have%
\begin{equation*}
\left( m_{A}^{n-1}\otimes m_{A}^{m-1}\right) \circ c_{\Omega
TA}^{m,n}=\left( m_{A}^{0}\otimes m_{A}^{0}\right) \circ c_{\Omega
TA}^{1,1}=c_{\Omega TA}^{1,1}\overset{(\ref{form:ct6})}{=}c_{A}=c_{A}\circ
\left( m_{A}^{m-1}\otimes m_{A}^{n-1}\right) .
\end{equation*}%
For $m=1,$ assume that the formula holds for $n\geq 1$ and let us check it
for $n+1.$ We have%
\begin{gather*}
\left( m_{A}^{n}\otimes m_{A}^{m-1}\right) \circ c_{\Omega
TA}^{m,n+1}=\left( m_{A}^{n}\otimes A\right) \circ c_{\Omega TA}^{1,n+1}%
\overset{(\ref{form:ct3})}{=}\left( m_{A}^{n}\otimes A\right) \circ \left(
A^{\otimes n}\otimes c_{\Omega TA}^{1,1}\right) \circ \left( c_{\Omega
TA}^{1,n}\otimes A^{\otimes 1}\right) \\
\overset{(\ref{form:ct6})}{=}\left( m_{A}\otimes A\right) \circ \left(
m_{A}^{n-1}\otimes A\otimes A\right) \circ \left( A^{\otimes n}\otimes
c_{A}\right) \circ \left( c_{\Omega TA}^{1,n}\otimes A\right) \\
=\left( m_{A}\otimes A\right) \circ \left( A\otimes c_{A}\right) \circ
\left( m_{A}^{n-1}\otimes m_{A}^{0}\otimes A\right) \circ \left( c_{\Omega
TA}^{1,n}\otimes A\right) \\
=\left( m_{A}\otimes A\right) \circ \left( A\otimes c_{A}\right) \circ
\left( c_{A}\otimes A\right) \circ \left( m_{A}^{0}\otimes
m_{A}^{n-1}\otimes A\right) \\
\overset{(\ref{Br3})}{=}c_{A}\circ \left( A\otimes m_{A}\right) \circ \left(
m_{A}^{0}\otimes m_{A}^{n-1}\otimes A\right) =c_{A}\circ \left(
m_{A}^{m-1}\otimes m_{A}^{n}\right) .
\end{gather*}

We have so proved that the statement holds for $m=1$ and $n\in
\mathbb{N}
\,.$

Assume that the equation holds for $m\geq 1$ and $n\geq 1$ and let us prove
it for $\left( m+1,n\right) $. We have%
\begin{gather*}
\left( m_{A}^{n-1}\otimes m_{A}^{m}\right) \circ c_{\Omega TA}^{m+1,n}%
\overset{(\ref{form:ct2})}{=}\left( m_{A}^{n-1}\otimes m_{A}^{m}\right)
\circ \left( c_{\Omega TA}^{m,n}\otimes A^{\otimes 1}\right) \circ \left(
A^{\otimes m}\otimes c_{\Omega TA}^{1,n}\right) \\
=\left( A\otimes m_{A}\right) \circ \left( m_{A}^{n-1}\otimes
m_{A}^{m-1}\otimes A\right) \circ \left( c_{\Omega TA}^{m,n}\otimes A\right)
\circ \left( A^{\otimes m}\otimes c_{\Omega TA}^{1,n}\right) \\
=\left( A\otimes m_{A}\right) \circ \left( c_{A}\otimes A\right) \circ
\left( m_{A}^{m-1}\otimes m_{A}^{n-1}\otimes A\right) \circ \left(
A^{\otimes m}\otimes c_{\Omega TA}^{1,n}\right) \\
=\left( A\otimes m_{A}\right) \circ \left( c_{A}\otimes A\right) \circ
\left( A\otimes c_{A}\right) \circ \left( m_{A}^{m-1}\otimes A\otimes
m_{A}^{n-1}\right) \\
\overset{(\ref{Br2})}{=}c_{A}\circ \left( m_{A}\otimes A\right) \circ \left(
m_{A}^{m-1}\otimes A\otimes m_{A}^{n-1}\right) =c_{A}\circ \left(
m_{A}^{m}\otimes m_{A}^{n-1}\right) .
\end{gather*}%
We have so proved that (*) holds. Hence $\epsilon \left(
A,m_{A},u_{A}\right) :T\Omega \left( A,m_{A},u_{A}\right) \rightarrow \left(
A,m_{A},u_{A}\right) $ induces a morphism $\epsilon _{\mathrm{Br}}\left(
A,m_{A},u_{A},c_{A}\right) :T_{\mathrm{Br}}\Omega _{\mathrm{Br}}\left(
A,m_{A},u_{A},c_{A}\right) \rightarrow \left( A,m_{A},u_{A},c_{A}\right) $
such that%
\begin{equation*}
H_{\mathrm{Alg}}\left( \epsilon _{\mathrm{Br}}\left(
A,m_{A},u_{A},c_{A}\right) \right) =\epsilon \left( A,m_{A},u_{A}\right) .
\end{equation*}%
The morphism $\epsilon _{\mathrm{Br}}\left( A,m_{A},u_{A},c_{A}\right) $ is
natural as $\epsilon \left( A,m_{A},u_{A}\right) $ is natural. We have%
\begin{eqnarray*}
&&H_{\mathrm{Alg}}\left( \epsilon _{\mathrm{Br}}T_{\mathrm{Br}}\circ T_{%
\mathrm{Br}}\eta _{\mathrm{Br}}\right) =\epsilon H_{\mathrm{Alg}}T_{\mathrm{%
Br}}\circ H_{\mathrm{Alg}}T_{\mathrm{Br}}\eta _{\mathrm{Br}} \\
&=&\epsilon TH\circ TH\eta _{\mathrm{Br}}=\epsilon TH\circ T\eta H=\mathrm{Id%
}_{TH}=H_{\mathrm{Alg}}\left( \mathrm{Id}_{T}\right) , \\
&&H\left( \Omega _{\mathrm{Br}}\epsilon _{\mathrm{Br}}\circ \eta _{\mathrm{Br%
}}\Omega _{\mathrm{Br}}\right) =H\Omega _{\mathrm{Br}}\epsilon _{\mathrm{Br}%
}\circ H\eta _{\mathrm{Br}}\Omega _{\mathrm{Br}} \\
&=&\Omega H_{\mathrm{Alg}}\epsilon _{\mathrm{Br}}\circ \eta H\Omega _{%
\mathrm{Br}}=\Omega \epsilon H_{\mathrm{Alg}}\circ \eta \Omega H_{\mathrm{Alg%
}}=\mathrm{Id}_{H_{\mathrm{Alg}}}=H_{\mathrm{Alg}}\left( \mathrm{Id}_{%
\mathcal{M}}\right) .
\end{eqnarray*}%
Since both $H_{\mathrm{Alg}}$ and $H$ are faithful, we get that $\left( T_{%
\mathrm{Br}},\Omega _{\mathrm{Br}}\right) $ is an adjunction with unit $\eta
_{\mathrm{Br}}$ and counit $\epsilon _{\mathrm{Br}}$.
\end{proof}

\begin{definition}
\label{def:prim}Let $\mathcal{M}$ be a preadditive monoidal category with
equalizers. Assume that the tensor functors are additive. Let $\mathbb{C}%
:=\left( C,\Delta _{C},\varepsilon _{C},u_{C}\right) $ be a coalgebra $%
\left( C,\Delta _{C},\varepsilon _{C}\right) $ endowed with a coalgebra
morphism $u_{C}:\mathbf{1}\rightarrow C$. In this setting we always
implicitly assume that we can choose a specific equalizer%
\begin{equation*}
 \xymatrixcolsep{1cm}
\xymatrix{
   P\left( \mathbb{C}\right) \ar[r]^{\xi \mathbb{C}} & C \ar@<.5ex>[rr]^{\Delta _{C}} \ar@<-.5ex>[rr]_{\left( C\otimes u_{C}\right) r_{C}^{-1}+\left( u_{C}\otimes C\right)
l_{C}^{-1}}&&C\otimes C  }
\end{equation*}%
We will use the same symbol when $\mathbb{C}$ comes out to be enriched with
an extra structure such us when $\mathbb{C}$ will denote a bialgebra or a
braided bialgebra.
\end{definition}

Next result should be compared with \cite[Lemma 6.2]{GV}. Note that, in our
case, the braiding of the primitive elements has not order two, in general.
Also our proof of the existence of such a braiding follows different lines.

\begin{lemma}
\label{lem:primitive}Let $\mathcal{M}$ a preadditive monoidal category with
equalizers. Assume that the tensor functors are additive and preserve
equalizers. For any $\mathbb{A}:=(A,m_{A},u_{A},\Delta _{A},\varepsilon
_{A},c_{A})\in \mathrm{BrBialg}_{\mathcal{M}},$ there is a unique morphism $%
c_{P\left( \mathbb{A}\right) }:P\left( \mathbb{A}\right) \otimes P\left(
\mathbb{A}\right) \rightarrow P\left( \mathbb{A}\right) \otimes P\left(
\mathbb{A}\right) $ such that
\begin{equation}
\left( \xi \mathbb{A}\otimes \xi \mathbb{A}\right) \circ c_{P\left( \mathbb{A%
}\right) }=c_{A}\circ \left( \xi \mathbb{A}\otimes \xi \mathbb{A}\right) .
\label{form:xibraided}
\end{equation}%
We have that $\left( P\left( \mathbb{A}\right) ,c_{P\left( \mathbb{A}\right)
}\right) \in \mathrm{Br}_{\mathcal{M}}$ and that $\xi \mathbb{A}:P\left(
\mathbb{A}\right) \rightarrow A$ is a morphism of braided objects that will
be denoted by $\xi \mathbb{A}:\left( P\left( \mathbb{A}\right) ,c_{P\left(
\mathbb{A}\right) }\right) \rightarrow \left( A,c_{A}\right) $. For any
morphism $f:\mathbb{A\rightarrow A}^{\prime }$ in $\mathrm{BrBialg}_{%
\mathcal{M}},$ there is a unique morphism $P\left( f\right) :P\left( \mathbb{%
A}\right) \rightarrow P\left( \mathbb{A}^{\prime }\right) $ such that
\begin{equation}
\xi \mathbb{A}^{\prime }\circ P\left( f\right) =f\circ \xi \mathbb{A}.
\label{form:xiP}
\end{equation}%
The morphism $P\left( f\right) $ is indeed a morphism of braided objects.
This way we get a functor%
\begin{equation*}
P_{\mathrm{Br}}:\mathrm{BrBialg}_{\mathcal{M}}\rightarrow \mathrm{Br}_{%
\mathcal{M}}:\mathbb{A}\mapsto \left( P\left( \mathbb{A}\right) ,c_{P\left(
\mathbb{A}\right) }\right) ,f\mapsto P\left( f\right) .
\end{equation*}%
Moreover%
\begin{equation}
\xi \mathbb{A}^{\prime }\circ P_{\mathrm{Br}}\left( f\right) =\Omega _{%
\mathrm{Br}}\mho _{\mathrm{Br}}\left( f\right) \circ \xi \mathbb{A}.
\label{form:NatXi}
\end{equation}%
for every morphism $f:\mathbb{A\rightarrow A}^{\prime }$ in $\mathrm{BrBialg}%
_{\mathcal{M}}$ i.e. $\xi \mathbb{A}:P_{\mathrm{Br}}\mathbb{A}\rightarrow
\Omega _{\mathrm{Br}}\mho _{\mathrm{Br}}\mathbb{A}$ is natural in $\mathbb{A}
$.
\end{lemma}

\begin{proof}
Note that, using Definition \ref{def:prim}, we have%
\begin{equation*}
\left( P\left( \mathbb{A}\right) ,\xi \mathbb{A}\right) :=\left( P\left(
A,\Delta _{A},\varepsilon _{A},u_{A}\right) ,\xi \left( A,\Delta
_{A},\varepsilon _{A},u_{A}\right) \right) .
\end{equation*}

For sake of shortness we just write $P$ instead of $P\left( \mathbb{A}%
\right) .$ Let us check that the braiding of $A$ induces a braiding on $P$.
To this aim, first consider the following diagram.
\begin{equation*}
  \xymatrixcolsep{2cm}
\xymatrix{
   P\otimes A \ar[r]^{\xi \mathbb{A}\otimes A} & A\otimes A\ar[d]^{c_A} \ar@<.5ex>[rr]^{\Delta _{A}\otimes A} \ar@<-.5ex>[rr]_{[\left( A\otimes u_{A}\right) r_{A}^{-1}+\left( u_{A}\otimes A\right)
l_{A}^{-1}]\otimes A}&&A\otimes A\otimes A\ar[d]^{\left( c_{A}\otimes A\right) \left(A\otimes c_{A}\right)} \\
 A\otimes P \ar[r]^{A\otimes \xi\mathbb{A}} & A\otimes A \ar@<.5ex>[rr]^{A\otimes \Delta _{A}} \ar@<-.5ex>[rr]_{A\otimes[\left( A\otimes u_{A}\right) r_{A}^{-1}+\left( u_{A}\otimes A\right)
l_{A}^{-1}]}&&A\otimes A\otimes A}
\end{equation*}
Note that $(A,\Delta _{A},\varepsilon _{A},c_{A})$ is a braided coalgebra
whence we have that (\ref{Br6}) holds. Moreover, using the equalities $%
r_{A}^{-1}\otimes A=A\otimes l_{A}^{-1}$ and $l_{A}^{-1}\otimes
A=l_{A\otimes A}^{-1},$ $(\ref{Br4}),$ the naturality of $l^{-1},$ the
equalities $A\otimes r_{A}^{-1}=r_{A\otimes A}^{-1}$ and $l_{A\otimes
A}^{-1}=l_{A}^{-1}\otimes A,$ the naturality of $r^{-1},$ $(\ref{Br4}),$ the
equalities $r_{A\otimes A}^{-1}=A\otimes r_{A}^{-1}$ and $r_{A}^{-1}\otimes
A=A\otimes l_{A}^{-1},$ we get
\begin{eqnarray*}
&&\left( c_{A}\otimes A\right) \left( A\otimes c_{A}\right) \left\{ \left[
\left( A\otimes u_{A}\right) r_{A}^{-1}+\left( u_{A}\otimes A\right)
l_{A}^{-1}\right] \otimes A\right\} \\
&=&\left\{ A\otimes \left[ \left( A\otimes u_{A}\right) r_{A}^{-1}+\left(
u_{A}\otimes A\right) l_{A}^{-1}\right] \right\} c_{A}.
\end{eqnarray*}%
Hence the diagram above serially commutes. Since the tensor product
preserves equalizers, the bottom fork of the diagram is an equalizer so that
there is a unique morphism $c_{P,A}:P\otimes A\rightarrow A\otimes P$ such
that
\begin{equation}
\left( A\otimes \xi \mathbb{A}\right) \circ c_{P,A}=c_{A}\circ \left( \xi
\mathbb{A}\otimes A\right) .  \label{form:cPA}
\end{equation}%
Similarly there is a unique morphism $c_{A,P}:A\otimes P\rightarrow P\otimes
A$ such that
\begin{equation}
\left( \xi \mathbb{A}\otimes A\right) \circ c_{A,P}=c_{A}\circ \left(
A\otimes \xi \mathbb{A}\right) .  \label{form:cAP}
\end{equation}%
Consider the following diagram
\begin{equation*}
  \xymatrixcolsep{2cm}
\xymatrix{
   P\otimes P \ar[r]^{P\otimes \xi \mathbb{A}} & P\otimes A\ar[d]^{c_{P,A}}\\
 P\otimes P \ar[r]^{\xi\mathbb{A}\otimes P} & A\otimes P \ar@<.5ex>[rr]^{\Delta _{A}\otimes P} \ar@<-.5ex>[rr]_{[\left( A\otimes u_{A}\right) r_{A}^{-1}+\left( u_{A}\otimes A\right)
l_{A}^{-1}]\otimes P}&&A\otimes A\otimes P}
\end{equation*}
Using (\ref{form:cPA}), (\ref{form:cAP}), the equalizer defining $\xi
\mathbb{A},$ (\ref{form:cAP}), (\ref{form:cPA}) we get
\begin{eqnarray*}
&&\left( A\otimes A\otimes \xi \mathbb{A}\right) \circ \left( \Delta
_{A}\otimes P\right) \circ c_{P,A}\circ \left( P\otimes \xi \mathbb{A}\right)
\\
&=&\left( A\otimes A\otimes \xi \mathbb{A}\right) \circ \left\{ \left[
\left( A\otimes u_{A}\right) r_{A}^{-1}+\left( u_{A}\otimes A\right)
l_{A}^{-1}\right] \otimes P\right\} \circ c_{P,A}\circ \left( P\otimes \xi
\mathbb{A}\right)
\end{eqnarray*}%
so that
\begin{equation*}
\left( \Delta _{A}\otimes P\right) \circ c_{P,A}\circ \left( P\otimes \xi
\mathbb{A}\right) =\left\{ \left[ \left( A\otimes u_{A}\right)
r_{A}^{-1}+\left( u_{A}\otimes A\right) l_{A}^{-1}\right] \otimes P\right\}
\circ c_{P,A}\circ \left( P\otimes \xi \mathbb{A}\right) .
\end{equation*}%
Hence there is a unique morphism $c_{P}:P\otimes P\rightarrow P\otimes P$
such that
\begin{equation}
\left( \xi \mathbb{A}\otimes P\right) \circ c_{P}=c_{P,A}\circ \left(
P\otimes \xi \mathbb{A}\right) .  \label{form:cP}
\end{equation}%
Using (\ref{form:cP}) and (\ref{form:cPA}) one gets $\left( \xi \mathbb{A}%
\otimes \xi \mathbb{A}\right) c_{P}=c_{A}\left( \xi \mathbb{A}\otimes \xi
\mathbb{A}\right) $ so that (\ref{form:xibraided}) holds. Note that, since $%
\xi \mathbb{A}\otimes \xi \mathbb{A}$ is a monomorphism, the morphism $c_{P}$
is uniquely determined by (\ref{form:xibraided}).

Since $(A,m_{A},u_{A},\Delta _{A},\varepsilon _{A},c_{A})\in \mathrm{BrBialg}%
_{\mathcal{M}}$ we have that (\ref{Br6}), (\ref{Br5}) and (\ref{Br4})
holds.If we write these equalities with respect to $c_{A}^{-1}$, we get%
\begin{gather*}
\left( A\otimes c_{A}^{-1}\right) \left( c_{A}^{-1}\otimes A\right) \left(
A\otimes \Delta _{A}\right) =\left( \Delta _{A}\otimes A\right) c_{A}^{-1},
\\
\left( c_{A}^{-1}\otimes A\right) \left( A\otimes c_{A}^{-1}\right) \left(
\Delta _{A}\otimes A\right) =\left( A\otimes \Delta _{A}\right) c_{A}^{-1},
\\
(u_{A}\otimes A)l_{A}^{-1}=c_{A}^{-1}\left( A\otimes u_{A}\right)
r_{A}^{-1},\qquad (A\otimes u_{A})r_{A}^{-1}=c_{A}^{-1}\left( u_{A}\otimes
A\right) l_{A}^{-1}.
\end{gather*}%
Thus $c_{A}^{-1}$ fulfills the same equalities that we used for $c_{A}$ in
the computations above. Hence, the same argument entails that there is a
morphism $c_{P}^{\prime }:P\otimes P\rightarrow P\otimes P$ such that
\begin{equation*}
\left( \xi \mathbb{A}\otimes \xi \mathbb{A}\right) c_{P}^{\prime
}=c_{A}^{-1}\left( \xi \mathbb{A}\otimes \xi \mathbb{A}\right) .
\end{equation*}%
Thus $\left( \xi \mathbb{A}\otimes \xi \mathbb{A}\right) c_{P}^{\prime
}c_{P}=c_{A}^{-1}\left( \xi \mathbb{A}\otimes \xi \mathbb{A}\right)
c_{P}=c_{A}^{-1}c_{A}\left( \xi \mathbb{A}\otimes \xi \mathbb{A}\right)
=\left( \xi \mathbb{A}\otimes \xi \mathbb{A}\right) $ and hence $%
c_{P}^{\prime }c_{P}=\mathrm{Id}_{P\otimes P}.$ Similarly $%
c_{P}c_{P}^{\prime }=\mathrm{Id}_{P\otimes P}$ so that $c_{P}$ is
invertible. Moreover using (\ref{form:xibraided}) repeatedly and the fact
that $c_{A}$ is a braiding, one checks that%
\begin{equation*}
\left( \xi \mathbb{A}\otimes \xi \mathbb{A}\otimes \xi \mathbb{A}\right)
\left( c_{P}\otimes P\right) \left( P\otimes c_{P}\right) \left(
c_{P}\otimes P\right) =\left( \xi \mathbb{A}\otimes \xi \mathbb{A}\otimes
\xi \mathbb{A}\right) \left( P\otimes c_{P}\right) \left( c_{P}\otimes
P\right) \left( P\otimes c_{P}\right)
\end{equation*}%
so that $\left( c_{P}\otimes P\right) \left( P\otimes c_{P}\right) \left(
c_{P}\otimes P\right) =\left( P\otimes c_{P}\right) \left( c_{P}\otimes
P\right) \left( P\otimes c_{P}\right) $ which means that $c_{P}$ is a
braiding. Note that (\ref{form:xibraided}) means that $\xi \mathbb{A}%
:P\rightarrow A$ is a morphism of braided objects that will be denoted by $%
\xi \mathbb{A}:\left( P,c_{P}\right) \rightarrow \left( A,c_{A}\right) $.
Let $f:\mathbb{A\rightarrow A}^{\prime }$ be a morphism in $\mathrm{BrBialg}%
_{\mathcal{M}}$ and consider the following diagram
\begin{equation*}
 \xymatrixcolsep{1.7cm}
\xymatrix{
   P\left( \mathbb{A}\right) \ar[r]^{\xi \mathbb{A}} & A \ar[d]^f \ar@<.5ex>[rr]^{\Delta _{A}} \ar@<-.5ex>[rr]_{\left( A\otimes u_{A}\right) r_{A}^{-1}+\left( u_{A}\otimes A\right)
l_{A}^{-1}}&&A\otimes A\ar[d]^{f\otimes f}  \\
P\left( \mathbb{A}^{\prime }\right) \ar[r]^{\xi \mathbb{A}^{\prime }} & A^{\prime } \ar@<.5ex>[rr]^{\Delta _{A^{\prime }}} \ar@<-.5ex>[rr]_{\left( A^{\prime }\otimes u_{A^{\prime }}\right) r_{A^{\prime }}^{-1}+\left( u_{A^{\prime }}\otimes A^{\prime }\right)
l_{A^{\prime }}^{-1}}&&A^{\prime }\otimes A^{\prime }  }
\end{equation*}
Using $\left( f\otimes f\right) \left( A\otimes u_{A}\right)
r_{A}^{-1}=\left( A^{\prime }\otimes u_{A^{\prime }}\right) r_{A^{\prime
}}^{-1}f$ and $\left( f\otimes f\right) \left( u_{A}\otimes A\right)
l_{A}^{-1}=\left( u_{A^{\prime }}\otimes A^{\prime }\right) l_{A^{\prime
}}^{-1}f$ we get that the diagram above serially commutes. Therefore there
is a unique morphism $P\left( f\right) :P\left( \mathbb{A}\right)
\rightarrow P\left( \mathbb{A}^{\prime }\right) $ such that (\ref{form:xiP})
holds. Using (\ref{form:xiP}), (\ref{form:xibraided}), the fact that $f$ is
braided, (\ref{form:xiP}) and (\ref{form:xibraided}) we arrive at $\left(
\xi \mathbb{A}^{\prime }\otimes \xi \mathbb{A}^{\prime }\right) \left(
P\left( f\right) \otimes P\left( f\right) \right) c_{P\left( \mathbb{A}%
\right) }=\left( \xi \mathbb{A}^{\prime }\otimes \xi \mathbb{A}^{\prime
}\right) c_{P\left( \mathbb{A}^{\prime }\right) }\left( P\left( f\right)
\otimes P\left( f\right) \right) $ so that $\left( P\left( f\right) \otimes
P\left( f\right) \right) c_{P}=c_{P^{\prime }}\left( P\left( f\right)
\otimes P\left( f\right) \right) $ which means that $P\left( f\right)
:P\left( \mathbb{A}\right) \rightarrow P\left( \mathbb{A}^{\prime }\right) $
is a morphism of braided objects. This way we get a functor%
\begin{equation*}
P_{\mathrm{Br}}:\mathrm{BrBialg}_{\mathcal{M}}\rightarrow \mathrm{Br}_{%
\mathcal{M}}:\mathbb{A}\mapsto \left( P\left( \mathbb{A}\right) ,c_{P\left(
\mathbb{A}\right) }\right) ,f\mapsto P\left( f\right) .
\end{equation*}%
By the foregoing we have (\ref{form:NatXi}) holds.
\end{proof}

We now investigate some properties of $T_{\mathrm{Br}}$.

\begin{lemma}
\label{lem:TbarStrict}Let $\mathcal{M}$ a preadditive monoidal category with
equalizers and denumerable coproducts. Assume that the tensor functors are
additive and preserve equalizers and denumerable coproducts. By Proposition %
\ref{pro:TbrStrict}, the forgetful functor $\Omega _{\mathrm{Br}}:\mathrm{%
BrAlg}_{\mathcal{M}}\rightarrow \mathrm{Br}_{\mathcal{M}}$ has a left
adjoint $T_{\mathrm{Br}}:\mathrm{Br}_{\mathcal{M}}\rightarrow \mathrm{BrAlg}%
_{\mathcal{M}}.$ For all $\mathbb{B}\in \mathrm{BrBialg}_{\mathcal{M}},$
since $T_{\mathrm{Br}}\left( V,c\right) \in \mathrm{BrAlg}_{\mathcal{M}}$ we
can write it in the form $T_{\mathrm{Br}}\left( V,c\right)
=(A,m_{A},u_{A},c_{A}).$ Regard $A\otimes A$ as an algebra in $\mathcal{M}$
via $m_{A\otimes A}:=\left( m_{A}\otimes m_{A}\right) \left( A\otimes
c_{A}\otimes A\right) $ and $u_{A\otimes A}:=\left( u_{A}\otimes
u_{A}\right) \Delta _{\mathbf{1}}.$ For every $n\in \mathbb{N},$ denote by $%
\alpha _{n}V:V^{\otimes n}\rightarrow \Omega TV$ the canonical injection.
Then there are unique algebra morphisms $\Delta _{A}:A\rightarrow A\otimes A$
and $\varepsilon _{A}:A\rightarrow \mathbf{1}$ such that
\begin{eqnarray}
\Delta _{A}\circ \alpha _{1}V &=&\delta _{V}^{l}+\delta _{V}^{r},
\label{form:TBrdelta} \\
\varepsilon _{A}\circ \alpha _{1}V &=&0,  \label{form:TBreps}
\end{eqnarray}%
where $\delta _{V}^{l}:=\left( u_{A}\otimes \alpha _{1}V\right) \circ
l_{V}^{-1}$ and $\delta _{V}^{r}:=\left( \alpha _{1}V\otimes u_{A}\right)
\circ r_{V}^{-1}.$ Moreover%
\begin{equation}
\varepsilon _{A}\circ \alpha _{n}V=\delta _{n,0}\mathrm{Id}_{\mathbf{1}},%
\text{ for every }n\in \mathbb{N}\text{.}  \label{form:TBrepsGEN}
\end{equation}%
The datum $\left( A,m_{A},u_{A},\Delta _{A},\varepsilon _{A},c_{A}\right) $
is a braided bialgebra. Moreover $T_{\mathrm{Br}}:\mathrm{Br}_{\mathcal{M}%
}\rightarrow \mathrm{BrAlg}_{\mathcal{M}}$ induces the functor%
\begin{eqnarray*}
\overline{T}_{\mathrm{Br}} &:&\mathrm{Br}_{\mathcal{M}}\rightarrow \mathrm{%
BrBialg}_{\mathcal{M}} \\
\left( V,c\right) &\mapsto &\left( T=T_{\mathrm{Br}}\left( V,c\right)
,\Delta _{T},\varepsilon _{T}\right) =(A,m_{A},u_{A},\Delta _{A},\varepsilon
_{A},c_{A}) \\
f &\mapsto &T_{\mathrm{Br}}\left( f\right) .
\end{eqnarray*}%
so that the following diagram commutes.
\begin{equation}
\xymatrixrowsep{12pt}
\xymatrixcolsep{0.7cm}
\xymatrix{\mathrm{BrBialg}_{\mathcal{M}} \ar[rr]^{\mho _{\mathrm{Br}}}&& \mathrm{BrAlg}_{\mathcal{M}} \\
&\mathrm{Br}_{\mathcal{M}}\ar[ul]^{\overline{T}_{\mathrm{Br}}}\ar[ru]_{ T_{\mathrm{Br}} }}
\label{form:OmegRibTbarBr}
\end{equation}
\end{lemma}

\begin{proof}
By Proposition \ref{pro:TbrStrict}, the unit $\eta _{\mathrm{Br}}$ and the
counit $\epsilon _{\mathrm{Br}}$ of the adjunction $\left( T_{\mathrm{Br}%
},\Omega _{\mathrm{Br}}\right) $ are uniquely determined by (\ref%
{form:TbrStrict}). Moreover the diagrams (\ref{form:HtildeTbrOmegaBr})
commute. Given $\left( V,c\right) \in \mathrm{Br}_{\mathcal{M}}$, then $T_{%
\mathrm{Br}}\left( V,c\right) $ becomes an object in $\mathrm{BrBialg}_{%
\mathcal{M}}$ as follows. For all $\mathbb{B}\in \mathrm{BrBialg}_{\mathcal{M%
}},$ consider the canonical isomorphism%
\begin{equation*}
\Phi \left( \left( V,c\right) ,\mathbb{B}\right) :\mathrm{BrAlg}_{\mathcal{M}%
}\left( T_{\mathrm{Br}}\left( V,c\right) ,\mathbb{B}\right) \rightarrow
\mathrm{Br}_{\mathcal{M}}\left( \left( V,c\right) ,\Omega _{\mathrm{Br}%
}\left( \mathbb{B}\right) \right) :f\mapsto \Omega _{\mathrm{Br}}\left(
f\right) \circ \eta _{\mathrm{Br}}\left( V,c\right) .
\end{equation*}%
Since $T_{\mathrm{Br}}\left( V,c\right) \in \mathrm{BrAlg}_{\mathcal{M}}$ we
can write it in the form $T_{\mathrm{Br}}\left( V,c\right)
=(A,m_{A},u_{A},c_{A}).$ By Lemma \ref{Lem:AotAotA}, $(B,m_{B},u_{B},c_{B})%
\in \mathrm{BrAlg}_{\mathcal{M}}$ where $B:=A\otimes A,m_{B}:=\left(
m_{A}\otimes m_{A}\right) \left( A\otimes c_{A}\otimes A\right)
,u_{B}:=\left( u_{A}\otimes u_{A}\right) \Delta _{\mathbf{1}}$ and $%
c_{B}:=\left( A\otimes c_{A}\otimes A\right) \left( c_{A}\otimes
c_{A}\right) \left( A\otimes c_{A}\otimes A\right) .$

In particular we have the morphisms%
\begin{eqnarray*}
\Phi \left( \left( V,c\right) ,(B,m_{B},u_{B},c_{B})\right) &:&\mathrm{BrAlg}%
_{\mathcal{M}}\left( T_{\mathrm{Br}}\left( V,c\right)
,(B,m_{B},u_{B},c_{B})\right) \rightarrow \mathrm{Br}_{\mathcal{M}}\left(
\left( V,c\right) ,(B,c_{B})\right) , \\
\Phi \left( \left( V,c\right) ,\left( \mathbf{1,}m_{\mathbf{1}},u_{\mathbf{1}%
},c_{\mathbf{1}}\right) \right) &:&\mathrm{BrAlg}_{\mathcal{M}}\left( T_{%
\mathrm{Br}}\left( V,c\right) ,\mathbb{B}\right) \rightarrow \mathrm{Br}_{%
\mathcal{M}}\left( \left( V,c\right) ,\left( \mathbf{1,}c_{\mathbf{1}%
}\right) \right) .
\end{eqnarray*}%
where $c_{\mathbf{1}}=\mathrm{Id}_{\mathbf{1\otimes 1}}.$ Note that
\begin{equation}
H\eta _{\mathrm{Br}}\left( V,c\right) \overset{(\ref{form:TbrStrict})}{=}%
\eta H\left( V,c\right) =\eta V=\alpha _{1}V.  \label{form:etaBralpha}
\end{equation}

Let us check that $\delta _{V}^{l}$ is a morphism of braided objects. Using
in the given order the definitions of $c_{B}$ and $\delta _{V}^{l},$ (\ref%
{Br4}) twice, the equality $l_{A}^{-1}\otimes A=l_{A\otimes A}^{-1},$ the
naturality of $l^{-1},$ the equality $l_{A\otimes A}^{-1}=l_{A}^{-1}\otimes
A,$ the equality $r_{A}\otimes A=A\otimes l_{A},$ the naturality of $l^{-1},$
(\ref{Br4}), the equality $\mathbf{1}\otimes l_{A}=r_{\mathbf{1}}\otimes A$,
the equality $r_{\mathbf{1}}=l_{\mathbf{1}}$, the naturality of $r$, the
equality $r_{V}\otimes V=V\otimes l_{V}$, the naturality of $l^{-1},$ the
equality (\ref{form:etaBralpha}) twice, the equality $l_{A\otimes \mathbf{1}%
}^{-1}=l_{A}^{-1}\otimes \mathbf{1,}$ the fact that $\eta _{\mathrm{Br}%
}\left( V,c\right) $ is braided, the equality (\ref{form:etaBralpha}) twice,
the equality $r_{A}^{-1}\otimes A=A\otimes l_{A}^{-1},$ the naturality of $%
l^{-1}$ and the definition of $\delta _{V}^{l},$ we arrive at $c_{B}\left(
\delta _{V}^{l}\otimes \delta _{V}^{l}\right) =\left( \delta _{V}^{l}\otimes
\delta _{V}^{l}\right) c.$ Analogously one gets the equalities $c_{B}\left(
\delta _{V}^{r}\otimes \delta _{V}^{r}\right) =\left( \delta _{V}^{r}\otimes
\delta _{V}^{r}\right) c,$ $c_{B}\left( \delta _{V}^{l}\otimes \delta
_{V}^{r}\right) =\left( \delta _{V}^{r}\otimes \delta _{V}^{l}\right) c$ and
$c_{B}\left( \delta _{V}^{r}\otimes \delta _{V}^{l}\right) =\left( \delta
_{V}^{r}\otimes \delta ^{l}\right) c$. By means of these equalities one
easily gets $c_{B}\left[ \left( \delta _{V}^{l}+\delta _{V}^{r}\right)
\otimes \left( \delta _{V}^{l}+\delta _{V}^{r}\right) \right] =\left[ \left(
\delta _{V}^{l}+\delta _{V}^{r}\right) \otimes \left( \delta _{V}^{l}+\delta
_{V}^{r}\right) \right] c.$ Thus $\delta _{V}^{l}+\delta
_{V}^{r}:V\rightarrow B$ defines a morphism $\widehat{\delta _{V}^{l}+\delta
_{V}^{r}}:\left( V,c\right) \rightarrow \left( B,c_{B}\right) $ in $\mathrm{%
Br}_{\mathcal{M}}$ such that $H\left( \widehat{\delta _{V}^{l}+\delta
_{V}^{r}}\right) =\delta _{V}^{l}+\delta _{V}^{r}$. Hence we can set
\begin{eqnarray*}
\Delta _{T} &:&=\Phi \left( \left( V,c\right) ,(B,m_{B},u_{B},c_{B})\right)
^{-1}\left( \widehat{\delta _{V}^{l}+\delta _{V}^{r}}\right) \in \mathrm{%
BrAlg}_{\mathcal{M}}\left( T_{\mathrm{Br}}\left( V,c\right)
,(B,m_{B},u_{B},c_{B})\right) , \\
\varepsilon _{T} &:&=\Phi \left( \left( V,c\right) ,\left( \mathbf{1,}m_{%
\mathbf{1}},u_{\mathbf{1}},c_{\mathbf{1}}\right) \right) ^{-1}\left(
0\right) \in \mathrm{BrAlg}_{\mathcal{M}}\left( T_{\mathrm{Br}}\left(
V,c\right) ,\left( \mathbf{1,}m_{\mathbf{1}},u_{\mathbf{1}},c_{\mathbf{1}%
}\right) \right) .
\end{eqnarray*}%
Moreover we set%
\begin{equation*}
\Delta _{A}:=H\Omega _{\mathrm{Br}}\Delta _{T}\qquad \text{and }\qquad
\varepsilon _{A}:=H\Omega _{\mathrm{Br}}\varepsilon _{T}.
\end{equation*}%
We have$\ref{form:TbrStrict}$
\begin{gather*}
\Delta _{A}\circ \alpha _{1}V\overset{(\ref{form:etaBralpha})}{=}H\Omega _{%
\mathrm{Br}}\Delta _{T}\circ H\eta _{\mathrm{Br}}\left( V,c\right) =H\left(
\Omega _{\mathrm{Br}}\Delta _{T}\circ \eta _{\mathrm{Br}}\left( V,c\right)
\right) \\
=H\left( \Phi \left( \left( V,c\right) ,(B,m_{B},u_{B},c_{B})\right) \left[
\Delta _{T}\right] \right) =H\left( \widehat{\delta _{V}^{l}+\delta _{V}^{r}}%
\right) =\delta _{V}^{l}+\delta _{V}^{r}, \\
\varepsilon _{A}\circ \alpha _{1}V\overset{(\ref{form:etaBralpha})}{=}%
H\Omega _{\mathrm{Br}}\varepsilon _{T}\circ H\eta _{\mathrm{Br}}\left(
V,c\right) =H\left( \Omega _{\mathrm{Br}}\varepsilon _{T}\circ \eta _{%
\mathrm{Br}}\left( V,c\right) \right) \\
=H\left( \Phi \left( \left( V,c\right) ,\left( \mathbf{1,}m_{\mathbf{1}},u_{%
\mathbf{1}},c_{\mathbf{1}}\right) \right) \left[ \varepsilon _{T}\right]
\right) =H\left( 0\right) =0
\end{gather*}%
so that we get (\ref{form:TBrdelta}) and (\ref{form:TBreps}). Note that,
since the tensor algebra functor is a left adjoint of the forgetful functor
and $\alpha _{1}V=\eta V,$ the unit of the adjunction, we have that the
algebra morphisms $\Delta _{A}$ and $\varepsilon _{A}$ are uniquely
determined by (\ref{form:TBrdelta}) and (\ref{form:TBreps}).

For every $n>0,$ we have%
\begin{equation*}
\varepsilon _{A}\circ \alpha _{n}V\overset{(\ref{form:TVm})}{=}\varepsilon
_{A}\circ m_{\Omega TV}\circ \left( \alpha _{n-1}V\otimes \alpha
_{1}V\right) =\left( \varepsilon _{A}\otimes \varepsilon _{A}\right) \circ
\left( \alpha _{n-1}V\otimes \alpha _{1}V\right) =0
\end{equation*}%
where we used that $\varepsilon _{A}$ is an algebra morphism. Moreover $%
\varepsilon _{A}\circ \alpha _{0}V=\varepsilon _{A}\circ u_{A}=\mathrm{Id}_{%
\mathbf{1}}.$ Thus we get that (\ref{form:TBrepsGEN}) holds. Using (\ref%
{form:defcT}), (\ref{form:TBrepsGEN}), the naturality of $r$, the equality $%
r_{V^{\otimes n}}c_{A}^{0,n}=l_{V^{\otimes n}}$ (which holds since
construction $c_{A}^{0,n}$ fulfills (\ref{form:ct4})), the naturality of $l$
and (\ref{form:TBrepsGEN}), we get
\begin{equation*}
r_{A}\left( A\otimes \varepsilon _{A}\right) c_{A}\left( \alpha _{m}V\otimes
\alpha _{n}V\right) =l_{A}\left( \varepsilon _{A}\otimes A\right) \left(
\alpha _{m}V\otimes \alpha _{n}V\right)
\end{equation*}%
Since this holds true for every $m,n,$ we obtain that $r_{A}\left( A\otimes
\varepsilon _{A}\right) c_{A}=l_{A}\left( \varepsilon _{A}\otimes A\right) $%
. Analogously one gets $l_{A}\left( \varepsilon _{A}\otimes A\right)
c_{A}=r_{A}\left( A\otimes \varepsilon _{A}\right) $ so that (\ref{Br7}) is
proved. Note that $\mathrm{Id}_{A}:A\rightarrow A$ and $\varepsilon
_{A}:A\rightarrow \mathbf{1}$ are morphisms in $\mathrm{BrAlg}_{\mathcal{M}%
}. $ Moreover $(A_{1},A_{2},c_{2,1})$ and $(A_{2},A_{1},c_{1,2})$
(respectively $(A_{1}^{\prime },A_{2}^{\prime },c_{2,1}^{\prime })$ and $%
(A_{2}^{\prime },A_{1}^{\prime },c_{1,2}^{\prime })$) fulfil the
requirements of Proposition \ref{pro:AotB}-2) for $A_{2}=A_{1}=A$ and $%
c_{i,j}=c_{A},i,j\in \left\{ 1,2\right\} $ (resp. for $A_{1}^{\prime }=A,$ $%
A_{2}^{\prime }=\mathbf{1}$, $c_{2,1}^{\prime
}=r_{A}^{-1}l_{A},c_{1,2}^{\prime }=l_{A}^{-1}r_{A},$ $c_{2,2}^{\prime }=%
\mathrm{Id}_{\mathbf{1}\otimes \mathbf{1}},c_{1,1}^{\prime }=c_{A}$).
Moreover, by the foregoing, we have
\begin{equation*}
\left( A\otimes \varepsilon _{A}\right) c_{2,1}=c_{2,1}^{\prime }\left(
\varepsilon _{A}\otimes A\right) ,\qquad \left( \varepsilon _{A}\otimes
A\right) c_{1,2}=c_{1,2}^{\prime }\left( A\otimes \varepsilon _{A}\right) .
\end{equation*}%
Thus, by Proposition \ref{pro:AotB}-3) applied to $f_{1}:=\mathrm{Id}_{A}$
and $f_{2}:=\varepsilon _{A}$, we can conclude that $T\otimes \varepsilon
_{T}$ is a morphism in $\mathrm{BrAlg}_{\mathcal{M}}.$ Thus we have that $%
\left( T\otimes \varepsilon _{T}\right) \Delta _{T}$ is a morphism in $%
\mathrm{BrAlg}_{\mathcal{M}}$. One also checks that $r_{A}:A\otimes \mathbf{1%
}\rightarrow A$ is a morphism in $\mathrm{BrAlg}_{\mathcal{M}}$. Thus we can
denote by $r_{T}$ the morphism $r_{A}$ regarded as a morphism in $\mathrm{%
BrAlg}_{\mathcal{M}}.$ In other words $H\Omega _{\mathrm{Br}}r_{T}=r_{A}.$
Thus we have that $r_{T}\left( T\otimes \varepsilon _{T}\right) \Delta
_{T}:T\rightarrow T$ is a morphism in $\mathrm{BrAlg}_{\mathcal{M}}$. We
have to check that $r_{T}\left( T\otimes \varepsilon _{T}\right) \Delta _{T}=%
\mathrm{Id}_{T}.$ Since the two sides are in $\mathrm{BrAlg}_{\mathcal{M}%
}\left( T_{\mathrm{Br}}\left( V,c\right) ,T_{\mathrm{Br}}\left( V,c\right)
\right) $ we have to prove that%
\begin{equation*}
\Phi \left( \left( V,c\right) ,T_{\mathrm{Br}}\left( V,c\right) \right) %
\left[ r_{T}\left( T\otimes \varepsilon _{T}\right) \Delta _{T}\right] =\Phi
\left( \left( V,c\right) ,T_{\mathrm{Br}}\left( V,c\right) \right) \left[
\mathrm{Id}_{T}\right]
\end{equation*}%
or equivalently%
\begin{equation}
H\Phi \left( \left( V,c\right) ,T_{\mathrm{Br}}\left( V,c\right) \right) %
\left[ r_{T}\left( T\otimes \varepsilon _{T}\right) \Delta _{T}\right]
=H\Phi \left( \left( V,c\right) ,T_{\mathrm{Br}}\left( V,c\right) \right) %
\left[ \mathrm{Id}_{T}\right] .  \label{form:epsdelta}
\end{equation}%
Note that for any braided algebra morphism $\xi :T\rightarrow U$, we have
\begin{gather*}
H\Phi \left( \left( V,c\right) ,U\right) \left( \xi \Delta _{T}\right)
=H\left\{ \Omega _{\mathrm{Br}}\left( \xi \Delta _{T}\right) \eta _{\mathrm{%
Br}}\left( V,c\right) \right\} =H\Omega _{\mathrm{Br}}\left( \xi \right)
H\left\{ \Omega _{\mathrm{Br}}\left( \Delta _{T}\right) \eta _{\mathrm{Br}%
}\left( V,c\right) \right\} \\
=H\Omega _{\mathrm{Br}}\left( \xi \right) H\left[ \widehat{\delta
_{V}^{l}+\delta _{V}^{r}}\right] =H\Omega _{\mathrm{Br}}\left( \xi \right)
\left( \delta _{V}^{l}+\delta _{V}^{r}\right) .
\end{gather*}%
so that%
\begin{equation}
H\Phi \left( \left( V,c\right) ,U\right) \left( \xi \Delta _{T}\right)
=H\Omega _{\mathrm{Br}}\left( \xi \right) \left( \delta _{V}^{l}+\delta
_{V}^{r}\right) \text{, for any }\xi :T\rightarrow U\text{ in }\mathrm{BrAlg}%
_{\mathcal{M}}\text{.}  \label{form:xidelta}
\end{equation}%
The left-hand side of (\ref{form:epsdelta}) is%
\begin{eqnarray*}
&&H\Phi \left( \left( V,c\right) ,T_{\mathrm{Br}}\left( V,c\right) \right) %
\left[ r_{T}\left( T\otimes \varepsilon _{T}\right) \Delta _{T}\right]
\overset{(\ref{form:xidelta})}{=}H\Omega _{\mathrm{Br}}\left[ r_{T}\left(
T\otimes \varepsilon _{T}\right) \right] \left( \delta _{V}^{l}+\delta
_{V}^{r}\right) \\
&=&r_{A}\left( A\otimes \varepsilon _{A}\right) \left( \delta
_{V}^{l}+\delta _{V}^{r}\right) =r_{A}\left( A\otimes \varepsilon
_{A}\right) \left( u_{A}\otimes \alpha _{1}V\right) l_{V}^{-1}+r_{A}\left(
A\otimes \varepsilon _{A}\right) \left( \alpha _{1}V\otimes u_{A}\right)
r_{V}^{-1} \\
&\overset{(\ref{form:TBreps})}{=}&r_{A}\left( \alpha _{1}V\otimes \mathbf{1}%
\right) r_{V}^{-1}=\left( \alpha _{1}V\right) r_{V}r_{V}^{-1}=\alpha _{1}V%
\overset{(\ref{form:etaBralpha})}{=}H\eta _{\mathrm{Br}}\left( V,c\right) \\
&=&H\Phi \left( \left( V,c\right) ,T_{\mathrm{Br}}\left( V,c\right) \right) %
\left[ \mathrm{Id}_{T}\right]
\end{eqnarray*}%
so that $r_{T}\left( T\otimes \varepsilon _{T}\right) \Delta _{T}=\mathrm{Id}%
_{T}.$ Similarly one proves that $l_{T}\left( \varepsilon _{T}\otimes
T\right) \Delta _{T}=\mathrm{Id}_{T}.$ By construction $\Delta
_{A}$ is a morphism of braided algebras so that $\left( \Delta _{A}\otimes
\Delta _{A}\right) c_{A}=c_{A\otimes A}\left( \Delta _{A}\otimes \Delta
_{A}\right) $ i.e.%
\begin{equation*}
\left( \Delta _{A}\otimes \Delta _{A}\right) c_{A}=\left( A\otimes
c_{A}\otimes A\right) \left( c_{A}\otimes c_{A}\right) \left( A\otimes
c_{A}\otimes A\right) \left( \Delta _{A}\otimes \Delta _{A}\right) .
\end{equation*}%
If we apply $\left( A\otimes A\otimes r_{A}\left( A\otimes \varepsilon
_{A}\right) \right) $ we get%
\begin{equation*}
\left( \Delta _{A}\otimes r_{A}\left( A\otimes \varepsilon _{A}\right)
\Delta _{A}\right) c_{A}=\left( A\otimes A\otimes r_{A}\left( A\otimes
\varepsilon _{A}\right) \right) \left( A\otimes c_{A}\otimes A\right) \left(
c_{A}\otimes c_{A}\right) \left( A\otimes c_{A}\otimes A\right) \left(
\Delta _{A}\otimes \Delta _{A}\right) .
\end{equation*}%
The left hand side is $\left( \Delta _{A}\otimes A\right) c_{A}$. Using
equality $A\otimes r_{A}=r_{A\otimes A},$ the naturality of $r$, the
equality $r_{A\otimes A}=A\otimes r_{A},$ equality (\ref{Br7})$,$ equality $%
A\otimes l_{A}=r_{A}\otimes A,$ equality (\ref{Br7})$,$ equality $A\otimes
l_{A}=r_{A}\otimes A$ and equality $r_{A}\left( A\otimes \varepsilon
_{A}\right) \Delta _{A}=\mathrm{Id}_{A},$ we see that the right hand side is
$\left( A\otimes c_{A}\right) \left( c_{A}\otimes A\right) \left( A\otimes
\Delta _{A}\right) .$ Hence we get (\ref{Br5}). Analogously one gets (\ref%
{Br6}).

Note that $\mathrm{Id}_{A}:A\rightarrow A$ and $\Delta _{A}:A\rightarrow
A\otimes A$ are morphisms in $\mathrm{BrAlg}_{\mathcal{M}}.$ Moreover $%
(A_{1},A_{2},c_{2,1})$ and $(A_{2},A_{1},c_{1,2})$ (respectively $%
(A_{1}^{\prime },A_{2}^{\prime },c_{2,1}^{\prime })$ and $(A_{2}^{\prime
},A_{1}^{\prime },c_{1,2}^{\prime })$) fulfil the requirements of
Proposition \ref{pro:AotB}-2) for $A_{2}=A_{1}=A$ and $c_{i,j}=c_{A},i,j\in
\left\{ 1,2\right\} $ (resp. for $A_{1}^{\prime }=A,$ $A_{2}^{\prime
}=A\otimes A$ and $c_{2,1}^{\prime }=\left( c_{A}\otimes A\right) \left(
A\otimes c_{A}\right) ,c_{1,2}^{\prime }=\left( A\otimes c_{A}\right) \left(
c_{A}\otimes A\right) ,c_{1,1}^{\prime }=c_{A},c_{2,2}^{\prime }=c_{A\otimes
A}$ see Lemma \ref{Lem:AotAotA}). Moreover, by the foregoing, we have
\begin{equation*}
\left( A\otimes \Delta _{A}\right) c_{2,1}=c_{2,1}^{\prime }\left( \Delta
_{A}\otimes A\right) \qquad \text{and}\qquad \left( \Delta _{A}\otimes
A\right) c_{1,2}=c_{1,1}^{\prime }\left( A\otimes \Delta _{A}\right) .
\end{equation*}%
Thus, by Proposition \ref{pro:AotB}-3) applied to $f_{1}:=\mathrm{Id}_{A}$
and $f_{2}:=\Delta _{A}$, we can conclude that $T\otimes \Delta _{T}$ is a
morphism in $\mathrm{BrAlg}_{\mathcal{M}}.$ Similarly $\Delta _{T}\otimes T$
is a morphism in $\mathrm{BrAlg}_{\mathcal{M}}.$ We have to check that $%
\left( T\otimes \Delta _{T}\right) \Delta _{T}=\left( \Delta _{T}\otimes
T\right) \Delta _{T}.$ Equivalently we will prove that%
\begin{equation}
H\Phi \left( \left( V,c\right) ,T\otimes T\otimes T\right) \left[ \left(
T\otimes \Delta _{T}\right) \Delta _{T}\right] =H\Phi \left( \left(
V,c\right) ,T\otimes T\otimes T\right) \left[ \left( \Delta _{T}\otimes
T\right) \Delta _{T}\right] .  \label{form:coass}
\end{equation}

If we apply (\ref{form:xidelta}) for $\xi =T\otimes \Delta _{T}$, the
left-hand side of (\ref{form:coass}) becomes%
\begin{eqnarray*}
&&H\Phi \left( \left( V,c\right) ,T\otimes T\otimes T\right) \left[ \left(
T\otimes \Delta _{T}\right) \Delta _{T}\right] \overset{(\ref{form:xidelta})}%
{=}H\Omega _{\mathrm{Br}}\left[ \left( T\otimes \Delta _{T}\right) \right]
\left( \delta _{V}^{l}+\delta _{V}^{r}\right) \\
&=&\left( A\otimes \Delta _{A}\right) \left( \delta _{V}^{l}+\delta
_{V}^{r}\right) =\left( A\otimes \Delta _{A}\right) \left( u_{A}\otimes
\alpha _{1}V\right) l_{V}^{-1}+\left( A\otimes \Delta _{A}\right) \left(
\alpha _{1}V\otimes u_{A}\right) r_{V}^{-1} \\
&\overset{(\ref{form:TBrdelta})}{=}&\left( u_{A}\otimes \left( \delta
_{V}^{l}+\delta _{V}^{r}\right) \right) l_{V}^{-1}+\left( \alpha
_{1}V\otimes \left( u_{A}\otimes u_{A}\right) \Delta _{\mathbf{1}}\right)
r_{V}^{-1} \\
&=&\left[
\begin{array}{c}
\left( u_{A}\otimes u_{A}\otimes \alpha _{1}V\right) \left( \mathbf{1}%
\otimes l_{V}^{-1}\right) l_{V}^{-1}+\left( u_{A}\otimes \alpha _{1}V\otimes
u_{A}\right) \left( \mathbf{1}\otimes r_{V}^{-1}\right) l_{V}^{-1}+ \\
+\left( \alpha _{1}V\otimes u_{A}\otimes u_{A}\right) \left( V\otimes \Delta
_{\mathbf{1}}\right) r_{V}^{-1}%
\end{array}%
\right] .
\end{eqnarray*}%
If we apply (\ref{form:xidelta}) for $\xi =\Delta _{T}\otimes T$, the
right-hand side of (\ref{form:coass}) becomes%
\begin{eqnarray*}
&&H\Phi \left( \left( V,c\right) ,T\otimes T\otimes T\right) \left[ \left(
\Delta _{T}\otimes T\right) \Delta _{T}\right] \overset{(\ref{form:xidelta})}%
{=}H\Omega _{\mathrm{Br}}\left[ \left( \Delta _{T}\otimes T\right) \right]
\left( \delta _{V}^{l}+\delta _{V}^{r}\right) \\
&=&\left( \Delta _{A}\otimes A\right) \left( \delta _{V}^{l}+\delta
_{V}^{r}\right) =\left( \Delta _{A}\otimes A\right) \left( u_{A}\otimes
\alpha _{1}V\right) l_{V}^{-1}+\left( \Delta _{A}\otimes A\right) \left(
\alpha _{1}V\otimes u_{A}\right) r_{V}^{-1} \\
&\overset{(\ref{form:TBrdelta})}{=}&\left( \left( u_{A}\otimes u_{A}\right)
\Delta _{\mathbf{1}}\otimes \alpha _{1}V\right) l_{V}^{-1}+\left( \left(
\delta _{V}^{l}+\delta _{V}^{r}\right) \otimes u_{A}\right) r_{V}^{-1} \\
&=&\left[
\begin{array}{c}
\left( u_{A}\otimes u_{A}\otimes \alpha _{1}V\right) \left( \mathbf{1}%
\otimes l_{V}^{-1}\right) l_{V}^{-1}+\left( u_{A}\otimes \alpha _{1}V\otimes
u_{A}\right) \left( \mathbf{1}\otimes r_{V}^{-1}\right) l_{V}^{-1}+ \\
+\left( \alpha _{1}V\otimes u_{A}\otimes u_{A}\right) \left( V\otimes \Delta
_{\mathbf{1}}\right) r_{V}^{-1}%
\end{array}%
\right] .
\end{eqnarray*}%
where the last equality depends on the definitions $\delta _{V}^{l}$ and $%
\delta _{V}^{r},$ and on the relations $\Delta _{\mathbf{1}}\otimes V=r_{%
\mathbf{1}}^{-1}\otimes V=\mathbf{1}\otimes l_{V}^{-1}$, $\left(
l_{V}^{-1}\otimes \mathbf{1}\right) r_{V}^{-1}=r_{\mathbf{1}\otimes
V}^{-1}l_{V}^{-1}=\left( \mathbf{1}\otimes r_{V}^{-1}\right) l_{V}^{-1}$ and
$r_{V}^{-1}\otimes \mathbf{1}=V\otimes \Delta _{\mathbf{1}}.$

We have so proved that $\left( T\otimes \Delta _{T}\right) \Delta
_{T}=\left( \Delta _{T}\otimes T\right) \Delta _{T}.$ Thus $\left(
A,m_{A},u_{A},\Delta _{A},\varepsilon _{A},c_{A}\right) $ is a braided
bialgebra.

Let $f:\left( V,c\right) \rightarrow \left( V^{\prime },c^{\prime }\right) $
be a morphism in $\mathrm{Br}_{\mathcal{M}}.$ Let us prove that $T_{\mathrm{%
Br}}\left( f\right) $ is a morphism of braided bialgebras. We know that $T_{%
\mathrm{Br}}\left( f\right) $ is a morphism in $\mathrm{BrAlg}_{\mathcal{M}%
}. $ We have to check that $T\left( f\right) =H\Omega _{\mathrm{Br}}T_{%
\mathrm{Br}}\left( f\right) $ is a morphism of coalgebras i.e. that
\begin{eqnarray*}
\left( T\left( f\right) \otimes T\left( f\right) \right) \circ \Delta
_{T\left( V\right) } &=&\Delta _{T\left( V^{\prime }\right) }\circ T\left(
f\right) , \\
\varepsilon _{T\left( V^{\prime }\right) }\circ T\left( f\right)
&=&\varepsilon _{T\left( V\right) }.
\end{eqnarray*}

Take $A:=T\left( V\right) $ and $A^{\prime }:=T\left( V^{\prime }\right) .$
Note that $T\left( Hf\right) :T\left( V\right) \rightarrow T\left( V^{\prime
}\right) $ and $T\left( Hf\right) :T\left( V\right) \rightarrow T\left(
V^{\prime }\right) $ are morphisms in $\mathrm{BrAlg}_{\mathcal{M}}.$
Moreover $(A_{1},A_{2},c_{2,1})$ and $(A_{2},A_{1},c_{1,2})$ (respectively $%
(A_{1}^{\prime },A_{2}^{\prime },c_{2,1}^{\prime })$ and $(A_{2}^{\prime
},A_{1}^{\prime },c_{1,2}^{\prime })$) fulfil the requirements of
Proposition \ref{pro:AotB}-2) for $A_{2}=A_{1}=A$ and $c_{i,j}=c_{A},i,j\in
\left\{ 1,2\right\} $ (resp. for $A_{2}^{\prime }=A_{1}^{\prime }=A^{\prime
} $ and $c_{i,j}^{\prime }=c_{A^{\prime }},i,j\in \left\{ 1,2\right\} ,$ see
Proposition \ref{pro:AotB}). Moreover, since $T\left( Hf\right) =H\Omega _{%
\mathrm{Br}}T_{\mathrm{Br}}\left( f\right) $ and $\Omega _{\mathrm{Br}}T_{%
\mathrm{Br}}\left( f\right) $ is a morphism of braided objects, we have
\begin{eqnarray*}
\left( T\left( Hf\right) \otimes T\left( Hf\right) \right) c_{2,1}
&=&c_{2,1}^{\prime }\left( T\left( Hf\right) \otimes T\left( Hf\right)
\right) , \\
\left( T\left( Hf\right) \otimes T\left( Hf\right) \right) c_{1,2}
&=&c_{1,2}\left( T\left( f\right) \otimes T\left( Hf\right) \right) .
\end{eqnarray*}%
Thus we can conclude that $T_{\mathrm{Br}}\left( f\right) \otimes T_{\mathrm{%
Br}}\left( f\right) $ is a morphism in $\mathrm{BrAlg}_{\mathcal{M}}.$ We
have to check that%
\begin{equation*}
\left( T_{\mathrm{Br}}\left( f\right) \otimes T_{\mathrm{Br}}\left( f\right)
\right) \circ \Delta _{T}=\Delta _{T^{\prime }}\circ T_{\mathrm{Br}}\left(
f\right)
\end{equation*}%
as morphisms in $\mathrm{BrAlg}_{\mathcal{M}}.$ Equivalently we will check
that%
\begin{equation*}
H\Phi \left( \left( V,c\right) ,T\otimes T\right) \left[ \left( T_{\mathrm{Br%
}}\left( f\right) \otimes T_{\mathrm{Br}}\left( f\right) \right) \circ
\Delta _{T}\right] =H\Phi \left( \left( V,c\right) ,T\otimes T\right) \left[
\Delta _{T^{\prime }}\circ T_{\mathrm{Br}}\left( f\right) \right] .
\end{equation*}%
The left hand-side is%
\begin{eqnarray*}
&&H\Phi \left( \left( V,c\right) ,T\otimes T\right) \left[ \left( T_{\mathrm{%
Br}}\left( f\right) \otimes T_{\mathrm{Br}}\left( f\right) \right) \circ
\Delta _{T}\right] \overset{(\ref{form:xidelta})}{=}H\Omega _{\mathrm{Br}}%
\left[ \left( T_{\mathrm{Br}}\left( f\right) \otimes T_{\mathrm{Br}}\left(
f\right) \right) \right] \left( \delta _{V}^{l}+\delta _{V}^{r}\right) \\
&=&\left( \Omega T\left( Hf\right) \otimes \Omega T\left( Hf\right) \right)
\left( u_{A}\otimes \alpha _{1}V\right) l_{V}^{-1}+\left( \Omega T\left(
Hf\right) \otimes \Omega T\left( Hf\right) \right) \left( \alpha
_{1}V\otimes u_{A}\right) r_{V}^{-1} \\
&\overset{(\ref{form:Tf})}{=}&\left( u_{A^{\prime }}\otimes \alpha
_{1}V^{\prime }\right) \left( \mathbf{1}\otimes Hf\right) l_{V}^{-1}+\left(
\alpha _{1}V^{\prime }\otimes u_{A^{\prime }}\right) \left( Hf\otimes
\mathbf{1}\right) r_{V}^{-1} \\
&=&\left( u_{A^{\prime }}\otimes \alpha _{1}V^{\prime }\right) l_{V^{\prime
}}^{-1}Hf+\left( \alpha _{1}V^{\prime }\otimes u_{A^{\prime }}\right)
r_{V^{\prime }}^{-1}Hf=\left( \delta _{V^{\prime }}^{l}+\delta _{V^{\prime
}}^{r}\right) Hf.
\end{eqnarray*}%
The right hand-side is%
\begin{eqnarray*}
&&H\Phi \left( \left( V,c\right) ,T\otimes T\right) \left[ \Delta
_{T^{\prime }}\circ T_{\mathrm{Br}}\left( f\right) \right] =H\left\{ \Omega
_{\mathrm{Br}}\left[ \Delta _{T^{\prime }}\circ T_{\mathrm{Br}}\left(
f\right) \right] \eta _{\mathrm{Br}}\left( V,c\right) \right\} \\
&=&H\left\{ \Omega _{\mathrm{Br}}\left[ \Delta _{T^{\prime }}\right]
\right\} H\left\{ \Omega _{\mathrm{Br}}T_{\mathrm{Br}}\left( f\right) \eta _{%
\mathrm{Br}}\left( V,c\right) \right\} =H\left\{ \Omega _{\mathrm{Br}}\left[
\Delta _{T^{\prime }}\right] \right\} H\left\{ \eta _{\mathrm{Br}}\left(
V^{\prime },c^{\prime }\right) \right\} Hf \\
&=&\left( \delta _{V^{\prime }}^{l}+\delta _{V^{\prime }}^{r}\right) Hf.
\end{eqnarray*}%
Hence the two sides coincide. We have proved that $T_{\mathrm{Br}}\left(
f\right) $ is comultiplicative. Let us check it is counitary i.e. that $%
\varepsilon _{T^{\prime }}\circ T_{\mathrm{Br}}\left( f\right) =\varepsilon
_{T}$ holds in $\mathrm{BrAlg}_{\mathcal{M}}$. Equivalently we have to prove
that $\Phi \left( \left( V,c\right) ,\mathbf{1}\right) \left[ \varepsilon
_{T^{\prime }}\circ T_{\mathrm{Br}}\left( f\right) \right] =\Phi \left(
\left( V,c\right) ,\mathbf{1}\right) \left[ \varepsilon _{T}\right] .$ We
have%
\begin{eqnarray*}
\Phi \left( \left( V,c\right) ,\mathbf{1}\right) \left[ \varepsilon
_{T^{\prime }}\circ T_{\mathrm{Br}}\left( f\right) \right] &=&\Omega _{%
\mathrm{Br}}\left[ \varepsilon _{T^{\prime }}\circ T_{\mathrm{Br}}\left(
f\right) \right] \eta _{\mathrm{Br}}\left( V,c\right) \\
&=&\Omega _{\mathrm{Br}}\left[ \varepsilon _{T^{\prime }}\right] \left\{
\Omega _{\mathrm{Br}}T_{\mathrm{Br}}\left( f\right) \eta _{\mathrm{Br}%
}\left( V,c\right) \right\} \\
&=&\Omega _{\mathrm{Br}}\left[ \varepsilon _{T^{\prime }}\right] \eta _{%
\mathrm{Br}}\left( V^{\prime },c^{\prime }\right) f=0= \\
&=&\Omega _{\mathrm{Br}}\left[ \varepsilon _{T}\right] \eta _{\mathrm{Br}%
}\left( V,c\right) =\Phi \left( \left( V,c\right) ,\mathbf{1}\right) \left[
\varepsilon _{T}\right] .
\end{eqnarray*}

By construction we have that diagram (\ref{form:OmegRibTbarBr}) commutes.
\end{proof}

Next aim is to check that the functor $P_{\mathrm{Br}}:\mathrm{BrBialg}_{%
\mathcal{M}}\rightarrow \mathrm{Br}_{\mathcal{M}}$ of Lemma \ref%
{lem:primitive} is a left adjoint of $\overline{T}_{\mathrm{Br}}$.

\begin{theorem}
\label{teo:TbarStrict}Take the hypotheses and notations of Lemma \ref%
{lem:primitive} i.e let $\mathcal{M}$ be a preadditive monoidal category
with equalizers and assume that the tensor functors are additive and
preserve equalizers. Assume also that the monoidal category $\mathcal{M}$
has denumerable coproducts and that the tensor functors preserve such
coproducts. Then%
\begin{equation*}
\left( \overline{T}_{\mathrm{Br}}:\mathrm{Br}_{\mathcal{M}}\rightarrow
\mathrm{BrBialg}_{\mathcal{M}},P_{\mathrm{Br}}:\mathrm{BrBialg}_{\mathcal{M}%
}\rightarrow \mathrm{Br}_{\mathcal{M}}\right)
\end{equation*}%
is an adjunction. The unit $\overline{\eta }_{\mathrm{Br}}$ and the counit $%
\overline{\epsilon }_{\mathrm{Br}}$ are uniquely determined by the following
equalities%
\begin{eqnarray}
\xi \overline{T}_{\mathrm{Br}}\circ \overline{\eta }_{\mathrm{Br}} &=&\eta _{%
\mathrm{Br}},  \label{form:BarEta} \\
\epsilon _{\mathrm{Br}}\mho _{\mathrm{Br}}\circ T_{\mathrm{Br}}\xi &=&\mho _{%
\mathrm{Br}}\overline{\epsilon }_{\mathrm{Br}},  \label{form:BarEps}
\end{eqnarray}%
where $\left( V,c\right) \in \mathrm{Br}_{\mathcal{M}},\mathbb{B}\in \mathrm{%
BrBialg}_{\mathcal{M}}$ while $\eta _{\mathrm{Br}}$ and $\epsilon _{\mathrm{%
Br}}$ denote the unit and counit of the adjunction $\left( T_{\mathrm{Br}%
},\Omega _{\mathrm{Br}}\right) $ respectively. Moreover $\mho _{\mathrm{Br}}:%
\mathrm{BrBialg}_{\mathcal{M}}\rightarrow \mathrm{BrAlg}_{\mathcal{M}}$
denotes the forgetful functor.
\end{theorem}

\begin{proof}
Let $\left( V,c\right) \in \mathrm{Br}_{\mathcal{M}}$. Let $\mathbb{A}:=%
\overline{T}_{\mathrm{Br}}\left( V,c\right) .$ Write $\mathbb{A}%
:=(A,m_{A},u_{A},\Delta _{A},\varepsilon _{A},c_{A}).$ Consider the equalizer%
\begin{equation*}
\xymatrixcolsep{1.5cm}
\xymatrix{
   P\left( \mathbb{A}\right) \ar[r]^{\xi \mathbb{A}} & A \ar@<.5ex>[rr]^{\Delta _{A}} \ar@<-.5ex>[rr]_{\left( A\otimes u_{A}\right) r_{A}^{-1}+\left( u_{A}\otimes A\right)
l_{A}^{-1}}&&A\otimes A }
\end{equation*}%
Note that the codomain of $\eta _{\mathrm{Br}}\left( V,c\right) $ is $\Omega
_{\mathrm{Br}}T_{\mathrm{Br}}\left( V,c\right) =\left( A,c_{A}\right) $ so
that it makes sense to check if $H\eta _{\mathrm{Br}}\left( V,c\right)
:V\rightarrow A$ is equalized by the pair $\left( \Delta _{A},\left(
A\otimes u_{A}\right) r_{A}^{-1}+\left( u_{A}\otimes A\right)
l_{A}^{-1}\right) .$ We have%
\begin{eqnarray*}
&&\left[ \left( A\otimes u_{A}\right) \circ r_{A}^{-1}+\left( u_{A}\otimes
A\right) \circ l_{A}^{-1}\right] \circ H\eta _{\mathrm{Br}}\left( V,c\right)
\\
&\overset{(\ref{form:etaBralpha})}{=}&\left[ \left( A\otimes u_{A}\right)
\circ r_{A}^{-1}+\left( u_{A}\otimes A\right) \circ l_{A}^{-1}\right] \circ
\alpha _{1}V \\
&=&\left( \alpha _{1}V\otimes u_{A}\right) \circ r_{V}^{-1}+\left(
u_{A}\otimes \alpha _{1}V\right) \circ l_{V}^{-1} \\
&=&\delta _{V}^{r}+\delta _{V}^{l}\overset{(\ref{form:TBrdelta})}{=}\Delta
_{A}\circ \alpha _{1}V\overset{(\ref{form:etaBralpha})}{=}\Delta _{A}\circ
H\eta _{\mathrm{Br}}\left( V,c\right) .
\end{eqnarray*}

By the universal property of equalizers, there is a unique morphism $%
\overline{\eta }_{\mathrm{Br}}\left( V,c\right) :V\rightarrow P\left(
\mathbb{A}\right) $ such that%
\begin{equation}
\xi \mathbb{A}\circ \overline{\eta }_{\mathrm{Br}}\left( V,c\right) =H\eta _{%
\mathrm{Br}}\left( V,c\right) .  \label{form:BarEta0}
\end{equation}%
We have%
\begin{eqnarray*}
&&\left( \xi \mathbb{A\otimes }\xi \mathbb{A}\right) \circ c_{P\left(
\mathbb{A}\right) }\circ \left( \overline{\eta }_{\mathrm{Br}}\left(
V,c\right) \otimes \overline{\eta }_{\mathrm{Br}}\left( V,c\right) \right)
\overset{(\ref{form:xibraided})}{=}c_{A}\circ \left( \xi \mathbb{A\otimes }%
\xi \mathbb{A}\right) \circ \left( \overline{\eta }_{\mathrm{Br}}\left(
V,c\right) \otimes \overline{\eta }_{\mathrm{Br}}\left( V,c\right) \right) \\
&=&c_{A}\circ \left( H\eta _{\mathrm{Br}}\left( V,c\right) \mathbb{\otimes }%
H\eta _{\mathrm{Br}}\left( V,c\right) \right) =\left( H\eta _{\mathrm{Br}%
}\left( V,c\right) \mathbb{\otimes }H\eta _{\mathrm{Br}}\left( V,c\right)
\right) \circ c \\
&=&\left( \xi \mathbb{A\otimes }\xi \mathbb{A}\right) \circ \left( \overline{%
\eta }_{\mathrm{Br}}\left( V,c\right) \otimes \overline{\eta }_{\mathrm{Br}%
}\left( V,c\right) \right) \circ c
\end{eqnarray*}%
and hence%
\begin{equation*}
c_{P\left( \mathbb{A}\right) }\circ \left( \overline{\eta }_{\mathrm{Br}%
}A\otimes \overline{\eta }_{\mathrm{Br}}A\right) =\left( \overline{\eta }_{%
\mathrm{Br}}A\otimes \overline{\eta }_{\mathrm{Br}}A\right) \circ c.
\end{equation*}%
Hence $\overline{\eta }_{\mathrm{Br}}\left( V,c\right) $ induces a morphism
of braided objects that we denote with the same symbol, namely $\overline{%
\eta }_{\mathrm{Br}}\left( V,c\right) :\left( V,c\right) \rightarrow \left(
P\left( \mathbb{A}\right) ,c_{P\left( \mathbb{A}\right) }\right) $.

Let us check that $\overline{\eta }_{\mathrm{Br}}\left( V,c\right) $ is
natural in $\left( V,c\right) .$ Let $f:\left( V,c\right) \rightarrow \left(
V^{\prime },c^{\prime }\right) $ be a morphism in $\mathrm{Br}_{\mathcal{M}}$%
. Then%
\begin{eqnarray*}
&&\xi \overline{T}_{\mathrm{Br}}\left( V^{\prime },c^{\prime }\right) \circ
P_{\mathrm{Br}}\overline{T}_{\mathrm{Br}}\left( f\right) \circ \overline{%
\eta }_{\mathrm{Br}}\left( V,c\right) \\
&\overset{(\ref{form:NatXi})}{=}&\Omega _{\mathrm{Br}}\mho _{\mathrm{Br}}%
\overline{T}_{\mathrm{Br}}\left( f\right) \circ \xi \overline{T}_{\mathrm{Br}%
}\left( V,c\right) \circ \overline{\eta }_{\mathrm{Br}}\left( V,c\right) \\
&\overset{(\ref{form:BarEta0})}{=}&\Omega _{\mathrm{Br}}T_{\mathrm{Br}%
}\left( f\right) \circ \eta _{\mathrm{Br}}\left( V,c\right) \\
&=&\eta _{\mathrm{Br}}\left( V^{\prime },c^{\prime }\right) \circ f\overset{(%
\ref{form:BarEta0})}{=}\xi \overline{T}_{\mathrm{Br}}\left( V^{\prime
},c^{\prime }\right) \circ \overline{\eta }_{\mathrm{Br}}\left( V^{\prime
},c^{\prime }\right) \circ f
\end{eqnarray*}%
and hence $P_{\mathrm{Br}}\overline{T}_{\mathrm{Br}}\left( f\right) \circ
\overline{\eta }_{\mathrm{Br}}\left( V,c\right) =\overline{\eta }_{\mathrm{Br%
}}\left( V^{\prime },c^{\prime }\right) \circ f$ which means that $\overline{%
\eta }_{\mathrm{Br}}\left( V,c\right) $ is natural in $\left( V,c\right) $.
A similar argument holds for $\xi \mathbb{A}$ so that we have proved (\ref%
{form:BarEta}).

The morphism $\overline{\eta }_{\mathrm{Br}}\left( V,c\right) $ will play
the role of the unit of the adjunction $\left( \overline{T}_{\mathrm{Br}},P_{%
\mathrm{Br}}\right) $. Let $\mathbb{B}:=(B,m_{B},u_{B},\Delta
_{B},\varepsilon _{B},c_{B})\in \mathrm{BrBialg}_{\mathcal{M}}$ and consider
the canonical isomorphism%
\begin{eqnarray*}
\Phi \left( P_{\mathrm{Br}}\left( \mathbb{B}\right) ,\mho _{\mathrm{Br}%
}\left( \mathbb{B}\right) \right) &:&\mathrm{BrAlg}_{\mathcal{M}}\left( T_{%
\mathrm{Br}}P_{\mathrm{Br}}\left( \mathbb{B}\right) ,\mho _{\mathrm{Br}%
}\left( \mathbb{B}\right) \right) \rightarrow \mathrm{Br}_{\mathcal{M}%
}\left( P_{\mathrm{Br}}\left( \mathbb{B}\right) ,\Omega _{\mathrm{Br}}\mho _{%
\mathrm{Br}}\left( \mathbb{B}\right) \right) \\
f &\mapsto &\Omega _{\mathrm{Br}}\left( f\right) \circ \eta _{\mathrm{Br}}P_{%
\mathrm{Br}}\left( \mathbb{B}\right) .
\end{eqnarray*}%
Define the morphism $\zeta \mathbb{B}:=\Phi \left( P_{\mathrm{Br}}\left(
\mathbb{B}\right) ,\mho _{\mathrm{Br}}\left( \mathbb{B}\right) \right)
^{-1}\left( \xi \mathbb{B}\right) $. This means that%
\begin{equation}
\Omega _{\mathrm{Br}}\left( \zeta \mathbb{B}\right) \circ \eta _{\mathrm{Br}%
}P_{\mathrm{Br}}\left( \mathbb{B}\right) =\xi \mathbb{B}.  \label{form:zetaB}
\end{equation}%
Set $\zeta B:=H\Omega _{\mathrm{Br}}\left( \zeta \mathbb{B}\right) =\Omega
H_{\mathrm{Alg}}\left( \zeta \mathbb{B}\right) :TP\left( \mathbb{B}\right)
\rightarrow B.$ Note that%
\begin{eqnarray*}
&&\zeta B\circ \eta P\left( \mathbb{B}\right) \overset{(\ref{form:etaBralpha}%
)}{=}H\Omega _{\mathrm{Br}}\left( \zeta \mathbb{B}\right) \circ H\eta _{%
\mathrm{Br}}P_{\mathrm{Br}}\left( \mathbb{B}\right) \\
&=&H\left[ \Omega _{\mathrm{Br}}\left( \zeta \mathbb{B}\right) \circ \eta _{%
\mathrm{Br}}P_{\mathrm{Br}}\left( \mathbb{B}\right) \right] \overset{(\ref%
{form:zetaB})}{=}H\xi \mathbb{B}
\end{eqnarray*}%
so that%
\begin{equation}
\zeta B\circ \eta P\left( \mathbb{B}\right) =H\xi \mathbb{B}\text{.}
\label{form:zetaB2}
\end{equation}

We will check that
\begin{equation}
\Delta _{B}\circ \zeta B=\left( \zeta B\otimes \zeta B\right) \circ \Delta
_{TP\left( \mathbb{B}\right) }.  \label{form:zetaBcomult}
\end{equation}%
The morphisms above are in particular algebra maps. Since $\left( T,\Omega
\right) $ is an adjunction, the equality above holds if%
\begin{equation*}
\Delta _{B}\circ \zeta B\circ \eta P\left( \mathbb{B}\right) =\left( \zeta
B\otimes \zeta B\right) \circ \Delta _{TP\left( \mathbb{B}\right) }\circ
\eta P\left( \mathbb{B}\right) .
\end{equation*}%
The first term is%
\begin{eqnarray*}
&&\Delta _{B}\circ \zeta B\circ \eta P\left( \mathbb{B}\right) \overset{(\ref%
{form:zetaB2})}{=}\Delta _{B}\circ H\xi \mathbb{B} \\
&=&\left( \left( B\otimes u_{B}\right) r_{B}^{-1}+\left( u_{B}\otimes
B\right) l_{B}^{-1}\right) \circ H\xi \mathbb{B}.
\end{eqnarray*}%
On the other hand, using, in the given order, (\ref{form:etaeps}), (\ref%
{form:TBrdelta}), the definitions of $\delta _{P\left( \mathbb{B}\right)
}^{l}$ and $\delta _{P\left( \mathbb{B}\right) }^{r}$, (\ref{form:etaeps}), (%
\ref{form:zetaB2}) and the naturality of the unit constraints, we obtain
that the second term is $\left( \left( B\otimes u_{B}\right)
r_{B}^{-1}+\left( u_{B}\otimes B\right) l_{B}^{-1}\right) \circ H\xi \mathbb{%
B}.$ Thus (\ref{form:zetaBcomult}) holds true. Now we will prove that $%
\varepsilon _{B}\circ \zeta B=\varepsilon _{\Omega TP\left( \mathbb{B}%
\right) }.$ Since $\left( T,\Omega \right) $ is an adjunction, the equality
above holds if%
\begin{equation*}
\varepsilon _{B}\circ \zeta B\circ \eta P\left( \mathbb{B}\right)
=\varepsilon _{\Omega TP\left( \mathbb{B}\right) }\circ \eta P\left( \mathbb{%
B}\right) .
\end{equation*}

We have%
\begin{equation*}
\varepsilon _{B}\circ \zeta B\circ \eta P\left( \mathbb{B}\right) \overset{(%
\ref{form:zetaB2})}{=}\varepsilon _{B}\circ H\xi \mathbb{B}\overset{(\ast )}{%
=}0\overset{(\ref{form:TBreps}),(\ref{form:etaeps})}{=}\varepsilon _{\Omega
TP\left( \mathbb{B}\right) }\circ \eta P\left( \mathbb{B}\right) .
\end{equation*}%
In order to prove (*)\ we proceed as follows. Consider the equalizer
\begin{equation*}
\xymatrixcolsep{1.5cm}
\xymatrix{
   P\left( \mathbb{B}\right) \ar[r]^{\xi \mathbb{B}} & B \ar@<.5ex>[rr]^{\Delta _{B}} \ar@<-.5ex>[rr]_{\left( B\otimes u_{B}\right) r_{B}^{-1}+\left( u_{B}\otimes B\right)
l_{B}^{-1}}&&B\otimes B }
\end{equation*}%
By applying $m_{\mathbf{1}}\circ \left( \varepsilon _{B}\otimes \varepsilon
_{B}\right) $ we get%
\begin{equation*}
m_{\mathbf{1}}\circ \left( \varepsilon _{B}\otimes \varepsilon _{B}\right)
\circ \Delta _{B}\circ \xi \mathbb{B}=m_{\mathbf{1}}\circ \left( \varepsilon
_{B}\otimes \varepsilon _{B}\right) \circ \left[ \left( B\otimes
u_{B}\right) r_{B}^{-1}+\left( u_{B}\otimes B\right) l_{B}^{-1}\right] \circ
\xi \mathbb{B}.
\end{equation*}%
The left hand-side is%
\begin{eqnarray*}
&&m_{\mathbf{1}}\circ \left( \varepsilon _{B}\otimes \varepsilon _{B}\right)
\circ \Delta _{B}\circ \xi \mathbb{B}=m_{\mathbf{1}}\circ \left( \varepsilon
_{B}\otimes \mathbf{1}\right) \circ \left( B\otimes \varepsilon _{B}\right)
\circ \Delta _{B}\circ \xi \mathbb{B} \\
&=&m_{\mathbf{1}}\circ \left( \varepsilon _{B}\otimes \mathbf{1}\right)
\circ r_{B}^{-1}\circ \xi \mathbb{B}=m_{\mathbf{1}}\circ r_{\mathbf{1}%
}^{-1}\circ \varepsilon _{B}\circ \xi \mathbb{B}=\varepsilon _{B}\circ \xi
\mathbb{B}
\end{eqnarray*}%
The right hand-side is%
\begin{eqnarray*}
&&m_{\mathbf{1}}\circ \left( \varepsilon _{B}\otimes \varepsilon _{B}\right)
\circ \left[ \left( B\otimes u_{B}\right) r_{B}^{-1}+\left( u_{B}\otimes
B\right) l_{B}^{-1}\right] \circ \xi \mathbb{B} \\
&=&m_{\mathbf{1}}\circ \left[ \left( \varepsilon _{B}\otimes \mathbf{1}%
\right) r_{B}^{-1}+\left( \mathbf{1}\otimes \varepsilon _{B}\right)
l_{B}^{-1}\right] \circ \xi \mathbb{B} \\
&=&m_{\mathbf{1}}\circ \left[ r_{\mathbf{1}}^{-1}\varepsilon _{B}+l_{\mathbf{%
1}}^{-1}\varepsilon _{B}\right] \circ \xi \mathbb{B}=2\varepsilon _{B}\circ
\xi \mathbb{B}
\end{eqnarray*}%
Hence we get $\varepsilon _{B}\circ \xi \mathbb{B}=2\varepsilon _{B}\circ
\xi \mathbb{B}$ and hence $\varepsilon _{B}\circ \xi \mathbb{B}=0$ as
required. Thus (*) is proved. Summing up, we have proved that $\zeta
B:TP\left( \mathbb{B}\right) \rightarrow B$ is a coalgebra morphism. Since $%
\zeta B:=H\Omega _{\mathrm{Br}}\left( \zeta \mathbb{B}\right) $, we also
know it is a morphism of algebras and braided objects so that there is a
unique morphism $\overline{\epsilon }_{\mathrm{Br}}\mathbb{B}:\overline{T}%
P\left( \mathbb{B}\right) \rightarrow \mathbb{B}$ in $\mathrm{BrBialg}_{%
\mathcal{M}}$ such that%
\begin{equation*}
\mho _{\mathrm{Br}}\left( \overline{\epsilon }_{\mathrm{Br}}\mathbb{B}%
\right) =\zeta \mathbb{B}.
\end{equation*}%
By definition of $\zeta \mathbb{B},$ we have%
\begin{equation*}
\Omega _{\mathrm{Br}}\mho _{\mathrm{Br}}\left( \overline{\epsilon }_{\mathrm{%
Br}}\mathbb{B}\right) \circ \eta _{\mathrm{Br}}P_{\mathrm{Br}}\left( \mathbb{%
B}\right) =\Omega _{\mathrm{Br}}\zeta \mathbb{B}\circ \eta _{\mathrm{Br}}P_{%
\mathrm{Br}}\left( \mathbb{B}\right) \overset{(\ref{form:zetaB})}{=}\xi
\mathbb{B}
\end{equation*}

Observe that $\overline{\epsilon }_{\mathrm{Br}}\mathbb{B}$ is uniquely
determined by the last equality. Note also that%
\begin{equation}
\epsilon _{\mathrm{Br}}\mho _{\mathrm{Br}}\mathbb{B}\circ T_{\mathrm{Br}}\xi
\mathbb{B}=\Phi \left( P_{\mathrm{Br}}\left( \mathbb{B}\right) ,\mho _{%
\mathrm{Br}}\mathbb{B}\right) ^{-1}\left( \xi \mathbb{B}\right) =\zeta
\mathbb{B}=\mho _{\mathrm{Br}}\left( \overline{\epsilon }_{\mathrm{Br}}%
\mathbb{B}\right) .  \label{form:BarEps0}
\end{equation}
Let us check that $\overline{\epsilon }_{\mathrm{Br}}\mathbb{B}$ is natural
in $\mathbb{B}$. Let $f:\mathbb{B}\rightarrow \mathbb{B}^{\prime }$ be a
morphism in $\mathrm{BrBialg}_{\mathcal{M}}$. Then%
\begin{eqnarray*}
&&\mho _{\mathrm{Br}}\left[ \overline{\epsilon }_{\mathrm{Br}}\mathbb{B}%
^{\prime }\circ \overline{T}_{\mathrm{Br}}P_{\mathrm{Br}}\left( f\right) %
\right] \overset{(\ref{form:BarEps0})}{=}\epsilon _{\mathrm{Br}}\mho _{%
\mathrm{Br}}\mathbb{B}^{\prime }\circ T_{\mathrm{Br}}\xi \mathbb{B}^{\prime
}\circ T_{\mathrm{Br}}P_{\mathrm{Br}}\left( f\right) \\
&\overset{(\ref{form:NatXi})}{=}&\epsilon _{\mathrm{Br}}\mho _{\mathrm{Br}}%
\mathbb{B}^{\prime }\circ T_{\mathrm{Br}}\Omega _{\mathrm{Br}}\mho _{\mathrm{%
Br}}\left( f\right) \circ T_{\mathrm{Br}}\xi \mathbb{B} \\
&=&\mho _{\mathrm{Br}}\left( f\right) \circ \epsilon _{\mathrm{Br}}\mho _{%
\mathrm{Br}}\mathbb{B}\circ T_{\mathrm{Br}}\xi \mathbb{B}\overset{(\ref%
{form:BarEps0})}{=}\mho _{\mathrm{Br}}\left[ f\circ \overline{\epsilon }_{%
\mathrm{Br}}\mathbb{B}\right]
\end{eqnarray*}%
Since $\mho _{\mathrm{Br}}$ is faithful, we obtain $\overline{\epsilon }_{%
\mathrm{Br}}\mathbb{B}^{\prime }\circ \overline{T}_{\mathrm{Br}}P_{\mathrm{Br%
}}\left( f\right) =f\circ \overline{\epsilon }_{\mathrm{Br}}\mathbb{B}$ so
that $\overline{\epsilon }_{\mathrm{Br}}\mathbb{B}$ is natural in $\mathbb{B}
$. Thus (\ref{form:BarEps0}) implies that (\ref{form:BarEps}) holds.

Let us check that $\left( \overline{T}_{\mathrm{Br}},P_{\mathrm{Br}}\right) $
is an adjunction with unit $\overline{\eta }_{\mathrm{Br}}$ and counit $%
\overline{\epsilon }_{\mathrm{Br}}$. We compute%
\begin{eqnarray*}
&&\xi \mathbb{B}\circ P_{\mathrm{Br}}\overline{\epsilon }_{\mathrm{Br}}%
\mathbb{B}\circ \overline{\eta }_{\mathrm{Br}}P_{\mathrm{Br}}\mathbb{B}%
\overset{(\ref{form:NatXi})}{=}\Omega _{\mathrm{Br}}\mho _{\mathrm{Br}}%
\overline{\epsilon }_{\mathrm{Br}}\mathbb{B}\circ \xi \overline{T}_{\mathrm{%
Br}}P_{\mathrm{Br}}\mathbb{B}\circ \overline{\eta }_{\mathrm{Br}}P_{\mathrm{%
Br}}\mathbb{B} \\
&\overset{(\ref{form:BarEta})}{=}&\Omega _{\mathrm{Br}}\mho _{\mathrm{Br}}%
\overline{\epsilon }_{\mathrm{Br}}\mathbb{B}\circ \eta _{\mathrm{Br}}P_{%
\mathrm{Br}}\mathbb{B}\overset{(\ref{form:BarEps})}{=}\Omega _{\mathrm{Br}%
}\epsilon _{\mathrm{Br}}\mho _{\mathrm{Br}}\mathbb{B}\circ \Omega _{\mathrm{%
Br}}T_{\mathrm{Br}}\xi \mathbb{B}\circ \eta _{\mathrm{Br}}P_{\mathrm{Br}}%
\mathbb{B} \\
&=&\Omega _{\mathrm{Br}}\epsilon _{\mathrm{Br}}\mho _{\mathrm{Br}}\mathbb{B}%
\circ \eta _{\mathrm{Br}}\Omega _{\mathrm{Br}}\mho _{\mathrm{Br}}\mathbb{B}%
\circ \xi \mathbb{B}=\xi \mathbb{B}
\end{eqnarray*}%
Since $\xi \mathbb{B}$ is a monomorphism, we get $P_{\mathrm{Br}}\overline{%
\epsilon }_{\mathrm{Br}}\mathbb{B}\circ \overline{\eta }_{\mathrm{Br}}P_{%
\mathrm{Br}}\mathbb{B}=\mathrm{Id}_{P_{\mathrm{Br}}\mathbb{B}}.$ We have%
\begin{eqnarray*}
&&\mho _{\mathrm{Br}}\left[ \overline{\epsilon }_{\mathrm{Br}}\overline{T}_{%
\mathrm{Br}}\circ \overline{T}_{\mathrm{Br}}\overline{\eta }_{\mathrm{Br}}%
\right] \overset{(\ref{form:BarEps})}{=}\epsilon _{\mathrm{Br}}\mho _{%
\mathrm{Br}}\overline{T}_{\mathrm{Br}}\circ T_{\mathrm{Br}}\xi \overline{T}_{%
\mathrm{Br}}\circ T_{\mathrm{Br}}\overline{\eta }_{\mathrm{Br}} \\
&&\overset{(\ref{form:BarEta})}{=}\epsilon _{\mathrm{Br}}T_{\mathrm{Br}%
}\circ T_{\mathrm{Br}}\eta _{\mathrm{Br}}=\mathrm{Id}_{T_{\mathrm{Br}}}=\mho
_{\mathrm{Br}}\left[ \mathrm{Id}_{\overline{T}_{\mathrm{Br}}}\right] .
\end{eqnarray*}%
Since $\mho _{\mathrm{Br}}$ is faithful, we get $\overline{\epsilon }_{%
\mathrm{Br}}\overline{T}_{\mathrm{Br}}\circ \overline{T}_{\mathrm{Br}}%
\overline{\eta }_{\mathrm{Br}}=\mathrm{Id}_{\overline{T}_{\mathrm{Br}}}.$
\end{proof}

\begin{proposition}
\label{pro:PrimFunct}Let $\mathcal{M}$ and $\mathcal{M}^{\prime }$ be
preadditive monoidal categories with equalizers. Assume that the tensor
functors are additive and preserve equalizers in both categories. Let $%
\left( F,\phi _{0},\phi _{2}\right) :\mathcal{M}\rightarrow \mathcal{M}%
^{\prime }$ be a monoidal functor which preserves equalizers. Then the
following diagram commutes, where $\mathrm{BrBialg}F$ and $\mathrm{Br}F$ are
the functors of Proposition \ref{pro:BrBialg}.
\begin{equation}
\xymatrixrowsep{15pt} \xymatrixcolsep{35pt}\xymatrix{\mathrm{BrBialg}_{\mathcal{M}}
\ar[r]^{\mathrm{BrBialg}F}\ar[d]_{P_{\mathrm{Br}}}&\mathrm{BrBialg}_{\mathcal{M}^{\prime }}\ar[d]^{{P^\prime_{\mathrm{Br}}}}\\
\mathrm{Br}_{\mathcal{M}} \ar[r]^{\mathrm{Br}F}&\mathrm{Br}_{\mathcal{M}^{\prime }}}  \label{diag:BrF-PBr}
\end{equation}%
Moreover we have
\begin{equation}
\xi ^{\prime }\left( \mathrm{BrBialg}F\right) =\left( \mathrm{Br}F\right)
\xi .  \label{form:comdat3}
\end{equation}
\end{proposition}

\begin{proof}
By Lemma \ref{lem:primitive}, for any $\mathbb{A}:=(A,m_{A},u_{A},\Delta
_{A},\varepsilon _{A},c_{A})\in \mathrm{BrBialg}_{\mathcal{M}}$ we have that
$P_{\mathrm{Br}}\mathbb{A}=\left( P\left( \mathbb{A}\right) ,c_{P\left(
\mathbb{A}\right) }\right) $ where $P\left( \mathbb{A}\right) $ is the
equalizer%
\begin{equation*}
\xymatrixcolsep{1.5cm}
\xymatrix{
   P\left( \mathbb{A}\right) \ar[r]^{\xi \mathbb{A}} & A \ar@<.5ex>[rr]^{\Delta _{A}} \ar@<-.5ex>[rr]_{\left( A\otimes u_{A}\right) r_{A}^{-1}+\left( u_{A}\otimes A\right)
l_{A}^{-1}}&&A\otimes A }
\end{equation*}
and $c_{P\left( \mathbb{A}\right) }$ is defined by (\ref{form:xibraided}).
We have%
\begin{eqnarray*}
&&\left( P_{\mathrm{Br}}^{\prime }\circ \mathrm{BrBialg}F\right) \left(
\mathbb{A}\right) =P_{\mathrm{Br}}^{\prime }\left( \left( \mathrm{BrBialg}%
F\right) \left( \mathbb{A}\right) \right) \\
&=&\left( P^{\prime }\left( \left( \mathrm{BrBialg}F\right) \left( \mathbb{A}%
\right) \right) ,c_{P^{\prime }\left( \left( \mathrm{BrBialg}F\right) \left(
\mathbb{A}\right) \right) }\right)
\end{eqnarray*}%
where%
\begin{eqnarray*}
&&\left( P^{\prime }\left( \left( \mathrm{BrBialg}F\right) \left( \mathbb{A}%
\right) \right) ,\xi ^{\prime }\left( \mathrm{BrBialg}F\right) \left(
\mathbb{A}\right) \right) \\
&=&\left( P^{\prime }\left( \left( FA,m_{FA},u_{FA},\Delta _{FA},\varepsilon
_{FA},c_{FA}\right) \right) ,\xi ^{\prime }\left( FA,m_{FA},u_{FA},\Delta
_{FA},\varepsilon _{FA},c_{FA}\right) \right) \\
&=&\mathrm{Equ}_{\mathcal{M}^{\prime }}\left( \Delta _{FA},\left( FA\otimes
u_{FA}\right) r_{FA}^{-1}+\left( u_{FA}\otimes FA\right) l_{FA}^{-1}\right)
\\
&=&\mathrm{Equ}_{\mathcal{M}^{\prime }}\left( \phi _{2}\left( A,A\right)
\Delta _{FA},\phi _{2}\left( A,A\right) \left( FA\otimes u_{FA}\right)
r_{FA}^{-1}+\phi _{2}\left( A,A\right) \left( u_{FA}\otimes FA\right)
l_{FA}^{-1}\right) \\
&=&\mathrm{Equ}_{\mathcal{M}^{\prime }}\left( F\Delta _{A},\phi _{2}\left(
A,A\right) \left( FA\otimes Fu_{A}\right) \left( FA\otimes \phi _{0}\right)
r_{FA}^{-1}+\phi _{2}\left( A,A\right) \left( Fu_{A}\otimes FA\right) \left(
\phi _{0}\otimes FA\right) l_{FA}^{-1}\right) \\
&=&\mathrm{Equ}_{\mathcal{M}^{\prime }}\left( F\Delta _{A},F\left( A\otimes
u_{A}\right) \phi _{2}\left( A,\mathbf{1}\right) \left( FA\otimes \phi
_{0}\right) r_{FA}^{-1}+F\left( u_{A}\otimes A\right) \phi _{2}\left(
\mathbf{1},A\right) \left( \phi _{0}\otimes FA\right) l_{FA}^{-1}\right) \\
&=&\mathrm{Equ}_{\mathcal{M}^{\prime }}\left( F\Delta _{A},F\left( A\otimes
u_{A}\right) F\left( r_{A}^{-1}\right) +F\left( u_{A}\otimes A\right)
F\left( l_{A}^{-1}\right) \right) \\
&=&F\left( \mathrm{Equ}_{\mathcal{M}}\left( \Delta _{A},\left( A\otimes
u_{A}\right) r_{A}^{-1}+\left( u_{A}\otimes A\right) l_{A}^{-1}\right)
\right) \\
&=&\left( \left( F\circ P\right) \left( A,m_{A},u_{A},\Delta
_{A},\varepsilon _{A},c_{A}\right) ,F\xi \left( A,m_{A},u_{A},\Delta
_{A},\varepsilon _{A},c_{A}\right) \right) \\
&=&\left( FP\left( \mathbb{A}\right) ,F\xi \mathbb{A}\right)
\end{eqnarray*}%
and $c_{P\left( \left( \mathrm{BrBialg}F\right) \left( \mathbb{A}\right)
\right) }$ fulfills%
\begin{eqnarray*}
&&\left( \xi ^{\prime }\left( \mathrm{BrBialg}F\right) \left( \mathbb{A}%
\right) \otimes \xi ^{\prime }\left( \mathrm{BrBialg}F\right) \left( \mathbb{%
A}\right) \right) \circ c_{P^{\prime }\left( \left( \mathrm{BrBialg}F\right)
\left( \mathbb{A}\right) \right) } \\
&\overset{(\ref{form:xibraided})}{=}&c_{FA}\circ \left( \xi ^{\prime }\left(
\mathrm{BrBialg}F\right) \left( \mathbb{A}\right) \otimes \xi ^{\prime
}\left( \mathrm{BrBialg}F\right) \left( \mathbb{A}\right) \right) \\
&=&c_{FA}\circ \left( F\xi \mathbb{A}\otimes F\xi \mathbb{A}\right) =\phi
_{2}^{-1}\left( A,A\right) \circ Fc_{A}\circ \phi _{2}\left( A,A\right)
\circ \left( F\xi \mathbb{A}\otimes F\xi \mathbb{A}\right) \\
&=&\phi _{2}^{-1}\left( A,A\right) \circ Fc_{A}\circ F\left( \xi \mathbb{A}%
\otimes \xi \mathbb{A}\right) \circ \phi _{2}\left( P\left( \mathbb{A}%
\right) ,P\left( \mathbb{A}\right) \right) \\
&\overset{(\ref{form:xibraided})}{=}&\phi _{2}^{-1}\left( A,A\right) \circ
F\left( \xi \mathbb{A}\otimes \xi \mathbb{A}\right) \circ Fc_{P\left(
\mathbb{A}\right) }\circ \phi _{2}\left( P\left( \mathbb{A}\right) ,P\left(
\mathbb{A}\right) \right) \\
&=&\left( F\xi \mathbb{A}\otimes F\xi \mathbb{A}\right) \circ \phi
_{2}^{-1}\left( P\left( \mathbb{A}\right) ,P\left( \mathbb{A}\right) \right)
\circ Fc_{P\left( \mathbb{A}\right) }\circ \phi _{2}\left( P\left( \mathbb{A}%
\right) ,P\left( \mathbb{A}\right) \right) \\
&=&\left( \xi ^{\prime }\left( \mathrm{BrBialg}F\right) \left( \mathbb{A}%
\right) \otimes \xi ^{\prime }\left( \mathrm{BrBialg}F\right) \left( \mathbb{%
A}\right) \right) \circ \phi _{2}^{-1}\left( P\left( \mathbb{A}\right)
,P\left( \mathbb{A}\right) \right) \circ Fc_{P\left( \mathbb{A}\right)
}\circ \phi _{2}\left( P\left( \mathbb{A}\right) ,P\left( \mathbb{A}\right)
\right)
\end{eqnarray*}%
so that%
\begin{equation*}
c_{P^{\prime }\left( \left( \mathrm{BrBialg}F\right) \left( \mathbb{A}%
\right) \right) }=\phi _{2}^{-1}\left( P\left( \mathbb{A}\right) ,P\left(
\mathbb{A}\right) \right) \circ Fc_{P\left( \mathbb{A}\right) }\circ \phi
_{2}\left( P\left( \mathbb{A}\right) ,P\left( \mathbb{A}\right) \right) .
\end{equation*}%
Summing up we get%
\begin{eqnarray*}
&&\left( P_{\mathrm{Br}}^{\prime }\circ \mathrm{BrBialg}F\right) \left(
\mathbb{A}\right) =\left( P^{\prime }\left( \left( \mathrm{BrBialg}F\right)
\left( \mathbb{A}\right) \right) ,c_{P^{\prime }\left( \left( \mathrm{BrBialg%
}F\right) \left( \mathbb{A}\right) \right) }\right) \\
&=&\left( FP\left( \mathbb{A}\right) ,\phi _{2}^{-1}\left( P\left( \mathbb{A}%
\right) ,P\left( \mathbb{A}\right) \right) \circ Fc_{P\left( \mathbb{A}%
\right) }\circ \phi _{2}\left( P\left( \mathbb{A}\right) ,P\left( \mathbb{A}%
\right) \right) \right) \\
&=&\left( \mathrm{Br}F\right) \left( P\left( \mathbb{A}\right) ,c_{P\left(
\mathbb{A}\right) }\right) =\left( \mathrm{Br}F\circ P_{\mathrm{Br}}\right)
\left( \mathbb{A}\right) .
\end{eqnarray*}

We have so proved that $\left( P_{\mathrm{Br}}^{\prime }\circ \mathrm{BrBialg%
}F\right) \left( \mathbb{A}\right) =\left( \mathrm{Br}F\circ P_{\mathrm{Br}%
}\right) \left( \mathbb{A}\right) .$ Let us check that the two functors
coincide also on morphisms. Let $f:\mathbb{A}\rightarrow \mathbb{A}^{\prime
} $ be a morphism in $\mathrm{BrBialg}_{\mathcal{M}}.$ In view of Lemma \ref%
{lem:primitive}, $P\left( f\right) $ makes the diagram
\begin{equation*}
 \xymatrixcolsep{1.5cm}
\xymatrix{
   P\left( \mathbb{A}\right)\ar[d]^{P(f)} \ar[r]^{\xi \mathbb{A}} & A \ar[d]^f \ar@<.5ex>[rr]^{\Delta _{A}} \ar@<-.5ex>[rr]_{\left( A\otimes u_{A}\right) r_{A}^{-1}+\left( u_{A}\otimes A\right)
l_{A}^{-1}}&&A\otimes A\ar[d]^{f\otimes f}  \\
P\left( \mathbb{A}^{\prime }\right) \ar[r]^{\xi \mathbb{A}^{\prime }} & A^{\prime } \ar@<.5ex>[rr]^{\Delta _{A^{\prime }}} \ar@<-.5ex>[rr]_{\left( A^{\prime }\otimes u_{A^{\prime }}\right) r_{A^{\prime }}^{-1}+\left( u_{A^{\prime }}\otimes A^{\prime }\right)
l_{A^{\prime }}^{-1}}&&A^{\prime }\otimes A^{\prime }  }
\end{equation*}
commutative. If we apply $F$ we get the commutative diagram
\begin{equation*}
 \xymatrixcolsep{1.5cm}
\xymatrix{
   FP\left( \mathbb{A}\right)\ar[d]^{FP(f)} \ar[r]^{F\xi \mathbb{A}} & FA \ar[d]^{Ff} \ar@<.5ex>[rr]^{F\Delta _{A}} \ar@<-.5ex>[rr]_{F\left( A\otimes u_{A}\right) Fr_{A}^{-1}+F\left( u_{A}\otimes A\right)
Fl_{A}^{-1}}&&F(A\otimes A)\ar[d]^{F(f\otimes f)}  \\
FP\left( \mathbb{A}^{\prime }\right) \ar[r]^{F\xi \mathbb{A}^{\prime }} & FA^{\prime } \ar@<.5ex>[rr]^{F\Delta _{A^{\prime }}} \ar@<-.5ex>[rr]_{F\left( A^{\prime }\otimes u_{A^{\prime }}\right) Fr_{A^{\prime }}^{-1}+F\left( u_{A^{\prime }}\otimes A^{\prime }\right)
Fl_{A^{\prime }}^{-1}}&&F(A^{\prime }\otimes A^{\prime })  }
\end{equation*}
Composing as above with $\phi _{2}\left( A^{\prime },A^{\prime }\right) $ we
get the commutative diagram whose rows are equalizers
\begin{equation*}
 \xymatrixcolsep{1.5cm}
\xymatrix{
   FP\left( \mathbb{A}\right)\ar[d]^{FP(f)} \ar[r]^{F\xi \mathbb{A}} & FA \ar[d]^{Ff} \ar@<.5ex>[rr]^{\Delta _{F A}} \ar@<-.5ex>[rr]_{\left( F A\otimes u_{FA}\right) r_{FA}^{-1}+\left( u_{F A}\otimes FA\right)
l_{FA}^{-1}}&&F A\otimes FA\ar[d]^{F f\otimes Ff}  \\
FP\left( \mathbb{A}^{\prime }\right) \ar[r]^{F\xi \mathbb{A}^{\prime }} & FA^{\prime } \ar@<.5ex>[rr]^{\Delta _{F A^{\prime }}} \ar@<-.5ex>[rr]_{\left( F A^{\prime }\otimes u_{FA^{\prime }}\right)r_{FA^{\prime }}^{-1}+\left( u_{F A^{\prime }}\otimes FA^{\prime }\right)
l_{FA^{\prime }}^{-1}}&&F A^{\prime }\otimes FA^{\prime }  }
\end{equation*}
Hence
\begin{equation*}
\left( P_{\mathrm{Br}}^{\prime }\circ \mathrm{BrBialg}F\right) \left(
f\right) =P^{\prime }\left( F\left( f\right) \right) =F\left( P\left(
f\right) \right) =\left( F\circ P\right) \left( f\right) .
\end{equation*}%
In conclusion, the diagram in the statement commutes. Moreover (\ref%
{form:comdat3}) holds.
\end{proof}

\section{Braided Categories}

\begin{claim}
\label{def braiding} A \emph{braided monoidal category} $(\mathcal{M}%
,\otimes ,\mathbf{1},a,l,r,c)$ is a monoidal category $(\mathcal{M},\otimes ,%
\mathbf{1})$ equipped with a \emph{braiding} $c$, that is an isomorphism $%
c_{U,V}:U\otimes V\rightarrow V\otimes U$, natural in $U,V\in \mathcal{M}$,
satisfying, for all $U,V,W\in \mathcal{M}$,
\begin{eqnarray*}
a_{V,W,U}\circ c_{U,V\otimes W}\circ a_{U,V,W} &=&(V\otimes c_{U,W})\circ
a_{V,U,W}\circ (c_{U,V}\otimes W), \\
a_{W,U,V}^{-1}\circ c_{U\otimes V,W}\circ a_{U,V,W}^{-1} &=&(c_{U,W}\otimes
V)\circ a_{U,W,V}^{-1}\circ (U\otimes c_{V,W}).
\end{eqnarray*}%
From now on we will omit the associativity and unity constraints unless
needed to clarify the context.

A braided monoidal category is called \emph{symmetric} if we further have $%
c_{V,U}\circ c_{U,V}=\mathrm{Id}_{U\otimes V}$ for every $U,V\in \mathcal{M}$%
.

A \emph{(symmetric) braided monoidal functor} is a monoidal functor $F:%
\mathcal{M}\rightarrow \mathcal{M}^{\prime }$ such that
\begin{equation*}
F\left( c_{U,V}\right) \circ \phi _{2}(U,V)=\phi _{2}(V,U)\circ c_{F\left(
U\right) ,F\left( V\right) }^{\prime }.
\end{equation*}%
More details on these topics can be found in \cite[Chapter XIII]{Kassel}%
.\medskip
\end{claim}

\begin{remark}
Given a braided monoidal category $(\mathcal{M},{\otimes },\mathbf{1},c)$
the category $\mathrm{Alg}_{\mathcal{M}}$ becomes monoidal where, for every $%
A,B\in \mathrm{Alg}_{\mathcal{M}}$ the multiplication and unit of $A\otimes
B $ are given by%
\begin{eqnarray*}
m_{A\otimes B} &:&=\left( m_{A}\otimes m_{B}\right) \circ \left( A\otimes
c_{B,A}\otimes B\right) :\left( A\otimes B\right) \otimes \left( A\otimes
B\right) \rightarrow A\otimes B, \\
u_{A\otimes B} &:&=\left( u_{A}\otimes u_{B}\right) \circ l_{\mathbf{1}%
}^{-1}:\mathbf{1}\rightarrow A\otimes B.
\end{eqnarray*}%
Moreover the forgetful functor $\mathrm{Alg}_{\mathcal{M}}\rightarrow
\mathcal{M}$ is a strict monoidal functor, cf. \cite[page 60]{Joyal-Street}.
\end{remark}

\begin{definition}
\label{cl: brdBialg} A \emph{bialgebra} in a braided monoidal category $(%
\mathcal{M},{\otimes },\mathbf{1},c)$ is a coalgebra $(B,\Delta ,\varepsilon
)$ in the monoidal category $\mathrm{Alg}_{\mathcal{M}}$. Equivalently a
bialgebra is a quintuple $\left( A,m,u,\Delta ,\varepsilon \right) $ where $%
\left( A,m,u\right) $ is an algebra in $\mathcal{M}$, $\left( A,\Delta
,\varepsilon \right) $ is a coalgebra in $\mathcal{M}$ such that $\Delta $
and $\varepsilon $ are morphisms of algebras where $A\otimes A$ is an
algebra as in the previous remark. Denote by $\mathrm{Bialg}_{\mathcal{M}}$
the category of bialgebras in $\mathcal{M}$ and their morphisms, defined in
the expected way.
\end{definition}

\begin{proposition}
\label{coro:BrBialg}Let $\mathcal{M}$ be a braided monoidal category.
Consider the obvious functors
\begin{equation*}
J:\mathcal{M}\rightarrow \mathrm{Br}_{\mathcal{M}},\quad J_{\mathrm{Alg}}:%
\mathrm{Alg}_{\mathcal{M}}\rightarrow \mathrm{BrAlg}_{\mathcal{M}}\quad
\text{and}\quad J_{\mathrm{Bialg}}:\mathrm{Bialg}_{\mathcal{M}}\rightarrow
\mathrm{BrBialg}_{\mathcal{M}}
\end{equation*}
which act as the identity on morphisms and defined on objects by
\begin{eqnarray*}
JV &=&\left( V,c_{V,V}\right) , \\
J_{\mathrm{Alg}}\left( A,m_{A},u_{A}\right) &=&\left(
A,m_{A},u_{A},c_{A,A}\right) , \\
J_{\mathrm{Bialg}}\left( B,m_{B},u_{B},\Delta _{B},\varepsilon _{B}\right)
&=&\left( B,m_{B},u_{B},\Delta _{B},\varepsilon _{B},c_{B,B}\right) .
\end{eqnarray*}

Then $J,$ $J_{\mathrm{Alg}}$ and $J_{\mathrm{Bialg}}$ are full, faithful and
conservative. Moreover the following diagram commutes.%
\begin{equation}
\xymatrixrowsep{15pt} \xymatrixcolsep{35pt}
\xymatrix{\mathrm{Bialg}_{\mathcal{M}}\ar[r]^{J_\mathrm{Bialg}}\ar[d]_{\mho}
&\mathrm{BrBialg}_{\mathcal{M}}\ar[d]^{\mho_{\mathrm{Br}}}\\
\mathrm{Alg}_{\mathcal{M}} \ar[r]^{J_\mathrm{Alg}}&\mathrm{BrAlg}_{\mathcal{M}}}
\qquad
\xymatrixrowsep{15pt} \xymatrixcolsep{35pt}
\xymatrix{\mathrm{Alg}_{\mathcal{M}}\ar[r]^{J_\mathrm{Alg}}\ar[d]_{\Omega}
&\mathrm{BrAlg}_{\mathcal{M}}\ar[d]^{\Omega_{\mathrm{Br}}}\\
\mathcal{M} \ar[r]^{J}&\mathrm{Br}_{\mathcal{M}}}
\label{diag:JAlg-Bialg}
\end{equation}
\end{proposition}

\begin{proof}
It is clear that $\left( V,c_{V,V}\right) $ is an object in $\mathrm{Br}_{%
\mathcal{M}}.$ Moreover any morphism in $\mathcal{M}$ becomes a morphism in $%
\mathrm{Br}_{\mathcal{M}}$ with respect to the braiding of $\mathcal{M}$.
Thus the functor $J$ is well-defined. Let us check that $J$ is full and
faithful. For $V,V^{\prime }\in \mathcal{M},$%
\begin{equation*}
\mathrm{Br}_{\mathcal{M}}\left( JV,JV^{\prime }\right) =\mathrm{Br}_{%
\mathcal{M}}\left( \left( V,c_{V,V}\right) ,\left( V^{\prime },c_{V^{\prime
},V^{\prime }}\right) \right) =\mathcal{M}\left( V,V^{\prime }\right) .
\end{equation*}

Using the naturality of the braiding in $\mathcal{M}$, one proves that $%
\left( A,m_{A},u_{A},c_{A,A}\right) $ is a braided algebra in $\mathcal{M}$
for every algebra $\left( A,m_{A},u_{A}\right) $ in $\mathcal{M}.$ Moreover
any morphism of algebras becomes a morphism of braided algebras with respect
to the braiding of $\mathcal{M}$. Thus the functor $J_{\mathrm{Alg}}$ is
well-defined. Let us check that $J_{\mathrm{Alg}}$ is full and faithful. For
$\left( A,m_{A},u_{A}\right) $ and $\left( A^{\prime },m_{A^{\prime
}},u_{A^{\prime }}\right) $ objects in $\mathrm{Alg}_{\mathcal{M}}$,%
\begin{eqnarray*}
&&\mathrm{BrAlg}_{\mathcal{M}}\left( J_{\mathrm{Alg}}\left(
A,m_{A},u_{A}\right) ,J_{\mathrm{Alg}}\left( A^{\prime },m_{A^{\prime
}},u_{A^{\prime }}\right) \right) \\
&=&\mathrm{BrAlg}_{\mathcal{M}}\left( \left( A,m_{A},u_{A},c_{A,A}\right)
,\left( A^{\prime },m_{A^{\prime }},u_{A^{\prime }},c_{A^{\prime },A^{\prime
}}\right) \right) \\
&=&\mathrm{Alg}_{\mathcal{M}}\left( \left( A,m_{A},u_{A}\right) ,\left(
A^{\prime },m_{A^{\prime }},u_{A^{\prime }}\right) \right) .
\end{eqnarray*}

By Definition \ref{cl: brdBialg}, a bialgebra in $\mathcal{M}$ is a
quintuple $\left( B,m_{B},u_{B},\Delta _{B},\varepsilon _{B}\right) $ where $%
\left( B,m_{B},u_{B}\right) $ is an algebra and $\left( B,\Delta
_{B},\varepsilon _{B}\right) $ a coalgebra in $\mathcal{M}$ such that $%
\Delta _{B}$ and $\varepsilon _{B}$ are morphisms of algebras where $%
B\otimes B$ is an algebra via the braiding of $\mathcal{M}$. Using the
naturality of the braiding in $\mathcal{M}$, one proves that $\left(
B,m_{B},u_{B},\Delta _{B},\varepsilon _{B},c_{B,B}\right) $ is indeed a
braided bialgebra in $\mathcal{M}$. Furthermore any morphism of bialgebras $%
f $ is indeed a morphism of braided bialgebras. Thus the functor $J_{\mathrm{%
Bialg}}$ is well-defined. Let us check that $J_{\mathrm{Bialg}}$ is full and
faithful. For $\left( B,m_{B},u_{B},\Delta _{B},\varepsilon _{B}\right) $
and $\left( B^{\prime },m_{B^{\prime }},u_{B^{\prime }},\Delta _{B^{\prime
}},\varepsilon _{B^{\prime }}\right) $ objects in $\mathrm{Bialg}_{\mathcal{M%
}}$,%
\begin{eqnarray*}
&&\mathrm{BrBialg}_{\mathcal{M}}\left( J_{\mathrm{Bialg}}\left(
B,m_{B},u_{B},\Delta _{B},\varepsilon _{B}\right) ,J_{\mathrm{Bialg}}\left(
B^{\prime },m_{B^{\prime }},u_{B^{\prime }},\Delta _{B^{\prime
}},\varepsilon _{B^{\prime }}\right) \right) \\
&=&\mathrm{BrBialg}_{\mathcal{M}}\left( \left( B,m_{B},u_{B},\Delta
_{B},\varepsilon _{B},c_{B,B}\right) ,\left( B^{\prime },m_{B^{\prime
}},u_{B^{\prime }},\Delta _{B^{\prime }},\varepsilon _{B^{\prime
}},c_{B^{\prime },B^{\prime }}\right) \right) \\
&=&\mathrm{Bialg}_{\mathcal{M}}\left( \left( B,m_{B},u_{B},\Delta
_{B},\varepsilon _{B}\right) ,\left( B^{\prime },m_{B^{\prime
}},u_{B^{\prime }},\Delta _{B^{\prime }},\varepsilon _{B^{\prime }}\right)
\right) .
\end{eqnarray*}%
The commutativity of the diagrams and the fact that $J,$ $J_{\mathrm{Alg}}$
and $J_{\mathrm{Bialg}}$ are conservative are clear.
\end{proof}

\begin{proposition}
\label{pro:JAlg-J}Take the hypotheses and notations of Proposition \ref%
{pro:TbrStrict} and assume that $\mathcal{M}$ is braided. Then we have a
commutative diagram%
\begin{equation}
\xymatrixrowsep{15pt} \xymatrixcolsep{35pt}
\xymatrix{\mathrm{Alg}_{\mathcal{M}}\ar[r]^{J_\mathrm{Alg}}\ar[d]_{T}
&\mathrm{BrAlg}_{\mathcal{M}}\ar[d]^{T_{\mathrm{Br}}}\\
\mathcal{M} \ar[r]^{J}&\mathrm{Br}_{\mathcal{M}}}
\label{form:JT}
\end{equation}%
where $J$ and $J_{\mathrm{Alg}}$ are as in Proposition \ref{coro:BrBialg}.
\end{proposition}

\begin{proof}
Set $\left( T,m_{T},u_{T}\right) :=TV.$ We have%
\begin{eqnarray*}
\left( J_{\mathrm{Alg}}\circ T\right) \left( V\right) &=&J_{\mathrm{Alg}%
}\left( T,m_{T},u_{T}\right) =\left( T,m_{T},u_{T},c_{T,T}\right) , \\
\left( T_{\mathrm{Br}}\circ J\right) \left( V\right) &=&T_{\mathrm{Br}%
}\left( V,c_{V,V}\right) =\left( T,m_{T},u_{T},c_{T}\right) ,
\end{eqnarray*}

where $c_{T}$ is uniquely determined by (\ref{form:defcT}). Let us check
that $c_{T}=c_{T,T}.$ Since $\mathcal{M}$ is braided and $\alpha
_{n}V:V^{\otimes n}\rightarrow T$ is a morphism in $\mathcal{M}$, we have
that $c_{T,T}\circ \left( \alpha _{m}V\otimes \alpha _{n}V\right) =\left(
\alpha _{n}V\otimes \alpha _{m}V\right) \circ c_{V^{\otimes m},V^{\otimes
n}} $. Since the tensor functors preserves denumerable coproducts, $c_{T,T}$
is uniquely determined by this equality. Thus it will coincide with $c_{T}$
once proved that $c_{V^{\otimes m},V^{\otimes n}}=c_{T}^{m,n}$ for every $%
m,n\in
\mathbb{N}
$. In view of Proposition \ref{pro:CT}, it suffices to check that $%
c_{V^{\otimes m},V^{\otimes n}}$ fulfills the analogues of the equalities in
that statement (which are fulfilled, by construction, by $c_{T}^{m,n}$). For
objects $X,Y,Z\in \mathcal{M},$
\begin{eqnarray*}
\left( c_{X,Z}\otimes Y\right) \circ \left( X\otimes c_{Y,Z}\right)
&=&c_{X\otimes Y,Z}\text{,\qquad }\left( Y\otimes c_{X,Z}\right) \circ
\left( c_{X,Y}\otimes Z\right) =c_{X,Y\otimes Z}, \\
c_{\mathbf{1},Z}\circ l_{Z}^{-1} &=&r_{Z}^{-1},\text{\qquad }c_{Z,\mathbf{1}%
}\circ r_{Z}^{-1}=l_{Z}^{-1}
\end{eqnarray*}%
If we take $X=V^{\otimes l},Y=V^{\otimes m}$ and $Z=V^{\otimes n}$ we get%
\begin{eqnarray*}
\left( c_{V^{\otimes l},V^{\otimes n}}\otimes V^{\otimes m}\right) \circ
\left( V^{\otimes l}\otimes c_{V^{\otimes m},V^{\otimes n}}\right)
&=&c_{V^{\otimes l}\otimes V^{\otimes m},V^{\otimes n}}=c_{V^{\otimes \left(
l+m\right) },V^{\otimes n}}\text{,} \\
\left( V^{\otimes m}\otimes c_{V^{\otimes l},V^{\otimes n}}\right) \circ
\left( c_{V^{\otimes l},V^{\otimes m}}\otimes V^{\otimes n}\right)
&=&c_{V^{\otimes l},V^{\otimes m}\otimes V^{\otimes n}}=c_{V^{\otimes
l},V^{\otimes \left( m+n\right) }}, \\
c_{\mathbf{1},V^{\otimes n}}\circ l_{V^{\otimes n}}^{-1} &=&r_{V^{\otimes
n}}^{-1},\text{\qquad }c_{V^{\otimes n},\mathbf{1}}\circ r_{V^{\otimes
n}}^{-1}=l_{V^{\otimes n}}^{-1}
\end{eqnarray*}

Hence, $c_{V^{\otimes m},V^{\otimes n}}$ fulfills equalities of the same
form of the ones defining $c_{T}^{m,n}$. Hence, in order to check that $%
c_{V^{\otimes m},V^{\otimes n}}=c_{T}^{m,n}$ we have only to prove that it
holds for $m,n\in \left\{ 0,1\right\} .$ But this is true. Summing up we
have proved that $c_{T,T}=c_{T}$ and hence $\left( J_{\mathrm{Alg}}\circ
T\right) \left( V\right) =\left( T_{\mathrm{Br}}\circ J\right) \left(
V\right) .$ Moreover, for every morphism $f$ in $\mathcal{M}$ we have $%
\left( J_{\mathrm{Alg}}\circ T\right) \left( f\right) =T\left( f\right)
=\left( T_{\mathrm{Br}}\circ J\right) \left( f\right) .$ Hence $J_{\mathrm{%
Alg}}\circ T=T_{\mathrm{Br}}\circ J$.
\end{proof}

\begin{claim}
\label{cl: Bialg} Let $\mathcal{M}$ be a preadditive braided monoidal
category with equalizers. Assume that the tensor functors are additive and
preserve equalizers. Let $\mathrm{Bialg}_{\mathcal{M}}$ be the category of
bialgebras in $\mathcal{M}$ and $\Theta :\mathrm{Bialg}_{\mathcal{M}%
}\rightarrow \mathcal{M}$ be the forgetful functor. Define the functor%
\begin{equation*}
P:=H\circ P_{\mathrm{Br}}\circ J_{\mathrm{Bialg}}:\mathrm{Bialg}_{\mathcal{M}%
}\rightarrow \mathcal{M}
\end{equation*}%
For any $\mathbb{B}:=\left( B,m_{B},u_{B},\Delta _{B},\varepsilon
_{B}\right) \in \mathrm{Bialg}_{\mathcal{M}}$ we have%
\begin{eqnarray*}
P\left( \mathbb{B}\right) &=&\left( H\circ P_{\mathrm{Br}}\circ J_{\mathrm{%
Bialg}}\right) \left( \mathbb{B}\right) =\left( H\circ P_{\mathrm{Br}%
}\right) \left( J_{\mathrm{Bialg}}\left( \mathbb{B}\right) \right) \\
&=&H\left( P\left( J_{\mathrm{Bialg}}\left( \mathbb{B}\right) \right)
,c_{P\left( J_{\mathrm{Bialg}}\left( \mathbb{B}\right) \right) }\right)
=P\left( J_{\mathrm{Bialg}}\left( \mathbb{B}\right) \right) \\
&=&P\left( B,m_{B},u_{B},\Delta _{B},\varepsilon _{B},c_{B,B}\right)
=P\left( B,\Delta _{B},\varepsilon _{B},u_{B}\right)
\end{eqnarray*}%
The canonical inclusion $\xi P\left( B,\Delta _{B},\varepsilon
_{B},u_{B}\right) :P\left( B,\Delta _{B},\varepsilon _{B},u_{B}\right)
\rightarrow B$ will be denoted by $\xi \mathbb{B}$. Thus we have the
equalizer%
\begin{equation*}
\xymatrixcolsep{1.5cm}
\xymatrix{P\left( \mathbb{B}\right) \ar[r]^{\xi \mathbb{B}} & B \ar@<.5ex>[rr]^{\Delta _{B}} \ar@<-.5ex>[rr]_{\left( B\otimes u_{B}\right) r_{B}^{-1}+\left( u_{B}\otimes B\right)
l_{B}^{-1}}&&B\otimes B }
\end{equation*}
\end{claim}

\begin{proposition}
\label{pro:J}Let $\mathcal{M}$ be a preadditive braided monoidal category
with equalizers. Assume that the tensor functors are additive and preserve
equalizers. Then we have a commutative diagram
\begin{equation}
\xymatrixrowsep{15pt} \xymatrixcolsep{35pt}
\xymatrix{\mathrm{Bialg}_{\mathcal{M}}\ar[r]^{J_\mathrm{Bialg}}\ar[d]_{P}
&\mathrm{BrBialg}_{\mathcal{M}}\ar[d]^{P_{\mathrm{Br}}}\\
{\mathcal{M}} \ar[r]^{J}&\mathrm{Br}_{\mathcal{M}}}
\label{form:JPbar}
\end{equation}%
where the horizontal arrows are the functors of Proposition \ref{coro:BrBialg}%
. Furthermore
\begin{equation}
\xi J_{\mathrm{Bialg}}=J\xi .  \label{form:xij}
\end{equation}
\end{proposition}

\begin{proof}
Let $\mathbb{B}:=\left( B,m_{B},u_{B},\Delta _{B},\varepsilon _{B}\right)
\in \mathrm{Bialg}_{\mathcal{M}}$. Then
\begin{equation*}
J_{\mathrm{Bialg}}\mathbb{B}=\left( B,m_{B},u_{B},\Delta _{B},\varepsilon
_{B},c_{B,B}\right) ,
\end{equation*}%
where $c_{B,B}$ is the braiding of $B$ in $\mathcal{M}$. Looking at (\ref%
{cl: Bialg}) and Lemma (\ref{lem:primitive}) we have $P_{\mathrm{Br}}J_{%
\mathrm{Bialg}}\mathbb{B}=\left( P\mathbb{B},c_{P}\right) $ where $\left(
\xi \mathbb{B}\otimes \xi \mathbb{B}\right) c_{P}=c_{B,B}\left( \xi \mathbb{B%
}\otimes \xi \mathbb{B}\right) $ and $\xi \mathbb{B}:P\left( \mathbb{B}%
\right) \rightarrow B$ is the morphism of definition of the equalizer. Since
$\xi \mathbb{B}$ belongs to $\mathcal{M},$ it is compatible with the
braiding so that $\left( \xi \mathbb{B}\otimes \xi \mathbb{B}\right)
c_{P,P}=c_{B,B}\left( \xi \mathbb{B}\otimes \xi \mathbb{B}\right) .$ Since
the tensor functors preserve equalizers we have that $\xi \mathbb{B}\otimes
\xi \mathbb{B}$ is a monomorphism and hence $c_{P}=c_{P,P}.$ Thus $P_{%
\mathrm{Br}}J_{\mathrm{Bialg}}\mathbb{B}=\left( P\mathbb{B},c_{P}\right)
=\left( P\mathbb{B},c_{P,P}\right) =JP\mathbb{B}$. Let $f:\mathbb{B}%
\rightarrow \mathbb{B}^{\prime }$ be a morphism in $\mathrm{Bialg}_{\mathcal{%
M}}.$ Then $P_{\mathrm{Br}}J_{\mathrm{Bialg}}f\in \mathrm{Br}_{\mathcal{M}%
}\left( P_{\mathrm{Br}}J_{\mathrm{Bialg}}\mathbb{B},P_{\mathrm{Br}}J_{%
\mathrm{Bialg}}\mathbb{B}^{\prime }\right) =\mathrm{Br}_{\mathcal{M}}\left(
JP\mathbb{B},JP\mathbb{B}^{\prime }\right) .$ By Proposition \ref%
{coro:BrBialg}, $J$ is full and faithful so that there is a unique morphism $%
g:P\mathbb{B}\rightarrow P\mathbb{B}^{\prime }$ such that $Jg=P_{\mathrm{Br}%
}J_{\mathrm{Bialg}}f.$ By definition of $P,$ we have $Pf=HP_{\mathrm{Br}}J_{%
\mathrm{Bialg}}f=HJg=g$ so that we get $JPf=Jg=P_{\mathrm{Br}}J_{\mathrm{%
Bialg}}f.$ This implies that $P_{\mathrm{Br}}\circ J_{\mathrm{Bialg}}=J\circ
P$.

Note that $\xi J_{\mathrm{Bialg}}\mathbb{B}$ goes from $P_{\mathrm{Br}}J_{%
\mathrm{Bialg}}\mathbb{B}$ to$\mathbb{\ }\Omega _{\mathrm{Br}}\mho _{\mathrm{%
Br}}J_{\mathrm{Bialg}}\mathbb{B}$. Now the first object is $JP\mathbb{B}.$
The second object is $\Omega _{\mathrm{Br}}\mho _{\mathrm{Br}}J_{\mathrm{%
Bialg}}\mathbb{B}=\Omega _{\mathrm{Br}}J_{\mathrm{Alg}}\mho \mathbb{B}%
=J\Omega \mho \mathbb{B}.$ Thus $\xi J_{\mathrm{Bialg}}\mathbb{B}\in \mathrm{%
Br}_{\mathcal{M}}\left( JP\mathbb{B},J\Omega \mho \mathbb{B}\right) .$ Using
again that $J$ is full and faithful, we get $\xi J_{\mathrm{Bialg}}\mathbb{B%
}=J\alpha $ for a unique morphism $\alpha .$ If we compose both sides of
this equality by $H$ we get $\alpha =H\xi J_{\mathrm{Bialg}}\mathbb{B}.$ By %
\ref{cl: Bialg} and Lemma \ref{lem:primitive}, we have that $H\xi J_{\mathrm{%
Bialg}}\mathbb{B}=H\xi \mathbb{B}=\xi \mathbb{B}.$ Thus $\alpha =\xi \mathbb{%
B}$ and hence we get $\xi J_{\mathrm{Bialg}}=J\xi .$
\end{proof}

\begin{claim}
\label{claim:Tbar}Let $\mathcal{M}$ be a preadditive braided monoidal
category with equalizers. Assume that the tensor functors are additive and
preserve equalizers. Assume further that $\mathcal{M}$ has denumerable
coproducts and that the tensor functors preserve such coproducts. By Remark %
\ref{cl: AlgMon}, the forgetful functor $\Omega :\mathrm{Alg}_{\mathcal{M}%
}\rightarrow \mathcal{M}$ has a left adjoint $T:\mathcal{M}\rightarrow
\mathrm{Alg}_{\mathcal{M}}$. Let us check that $T\left( V\right) $ becomes
an object in $\mathrm{Bialg}_{\mathcal{M}}$. Let $V\in \mathcal{M}$ and
consider
\begin{equation*}
\left( \overline{T}_{\mathrm{Br}}\circ J\right) \left( V\right) =\overline{T}%
_{\mathrm{Br}}\left( V,c_{V,V}\right) .
\end{equation*}%
Denote this braided bialgebra by $\left( A,m_{A},u_{A},\Delta
_{A},\varepsilon _{A},c_{A}\right) .$ Note that%
\begin{gather*}
T\left( V\right) =THJV\overset{(\ref{form:HtildeTbrOmegaBr})}{=}H_{\mathrm{%
Alg}}T_{\mathrm{Br}}J\left( V\right) \\
\overset{(\ref{form:OmegRibTbarBr})}{=}H_{\mathrm{Alg}}\mho _{\mathrm{Br}}%
\overline{T}_{\mathrm{Br}}J\left( V\right) =\left( A,m_{A},u_{A}\right) .
\end{gather*}%
Let us compute the braiding $c_{A}$. We have%
\begin{eqnarray*}
\left( A,m_{A},u_{A},c_{A}\right) &=&\mho _{\mathrm{Br}}\left(
A,m_{A},u_{A},\Delta _{A},\varepsilon _{A},c_{A}\right) \\
&=&\mho _{\mathrm{Br}}\overline{T}_{\mathrm{Br}}J\left( V\right) \overset{(%
\ref{form:OmegRibTbarBr})}{=}T_{\mathrm{Br}}J\left( V\right) \overset{(\ref%
{form:JT})}{=}J_{\mathrm{Alg}}T\left( V\right) \\
&=&\left( A,m_{A},u_{A},c_{A,A}\right)
\end{eqnarray*}%
so that $c_{A}=c_{A,A}.$ Since $\left( A,m_{A},u_{A},\Delta _{A},\varepsilon
_{A},c_{A,A}\right) =\left( A,m_{A},u_{A},\Delta _{A},\varepsilon
_{A},c_{A}\right) \ $which is a braided bialgebra, it is clear that $\left(
A,m_{A},u_{A},\Delta _{A},\varepsilon _{A}\right) $ is a bialgebra in $%
\mathcal{M}$ that will be denoted by $\overline{T}\left( V\right) .$ By
construction we have%
\begin{equation*}
\left( \overline{T}_{\mathrm{Br}}\circ J\right) \left( V\right) =J_{\mathrm{%
Bialg}}\left( \overline{T}\left( V\right) \right) .
\end{equation*}%
Let $f:V\rightarrow W$ be a morphism in $\mathcal{M}$. Then $$\left(
\overline{T}_{\mathrm{Br}}\circ J\right) \left( f\right) \in \mathrm{BrBialg}%
_{\mathcal{M}}\left( J_{\mathrm{Bialg}}\left( \overline{T}\left( V\right)
\right) ,J_{\mathrm{Bialg}}\left( \overline{T}\left( W\right) \right)
\right) .$$ By Proposition \ref{coro:BrBialg}, we have that $J_{\mathrm{Bialg}%
}$ is full and faithful so that there is a unique morphism $\overline{T}%
\left( f\right) \in \mathrm{Bialg}_{\mathcal{M}}\left( \overline{T}\left(
V\right) ,\overline{T}\left( W\right) \right) $ such that $\left( \overline{T%
}_{\mathrm{Br}}\circ J\right) \left( f\right) =J_{\mathrm{Bialg}}\left(
\overline{T}\left( f\right) \right) .$ In this way we have defined a functor%
\begin{equation*}
\overline{T}:\mathcal{M}\rightarrow \mathrm{Bialg}_{\mathcal{M}}
\end{equation*}%
such that $\overline{T}_{\mathrm{Br}}\circ J=J_{\mathrm{Bialg}}\circ
\overline{T}.$ Thus we get the commutative diagram%
\begin{equation}
\xymatrixrowsep{15pt} \xymatrixcolsep{35pt}
\xymatrix{\mathrm{Bialg}_{\mathcal{M}}\ar[r]^{J_\mathrm{Bialg}}
&\mathrm{BrBialg}_{\mathcal{M}}\\
\mathcal{M}\ar[u]^{\overline{T}} \ar[r]^{J}&\mathrm{Br}_{\mathcal{M}}\ar[u]_{\overline{T}_{\mathrm{Br}}}}
\label{form:JTbar}
\end{equation}%
Note that, by construction we have that (\ref{form:TBrdelta}) and (\ref%
{form:TBrepsGEN}) hold. We compute
\begin{gather*}
\mho \overline{T}=H_{\mathrm{Alg}}J_{\mathrm{Alg}}\mho \overline{T}\overset{(%
\ref{diag:JAlg-Bialg})}{=}H_{\mathrm{Alg}}\mho _{\mathrm{Br}}J_{\mathrm{Bialg%
}}\overline{T}\overset{(\ref{form:JTbar})}{=}H_{\mathrm{Alg}}\mho _{\mathrm{%
Br}}\overline{T}_{\mathrm{Br}}J \\
\overset{(\ref{form:OmegRibTbarBr})}{=}H_{\mathrm{Alg}}T_{\mathrm{Br}}J%
\overset{(\ref{form:HtildeTbrOmegaBr})}{=}THJ=T
\end{gather*}%
so that the following diagram commutes.
\begin{equation}
\xymatrixrowsep{12pt}
\xymatrixcolsep{0.7cm}
\xymatrix{\mathrm{Bialg}_{\mathcal{M}} \ar[rr]^{\mho }&& \mathrm{Alg}_{\mathcal{M}} \\
&\mathcal{M}\ar[ul]^{\overline{T}}\ar[ru]_{ T }}
\label{form:OmegRibTbar}
\end{equation}
\end{claim}

Let us check $\overline{T}$ is a left adjoint of the functor $P:\mathrm{Bialg%
}_{\mathcal{M}}\rightarrow \mathcal{M}$.

\begin{theorem}
\label{teo:Tbar}Let $\mathcal{M}$ be a preadditive braided monoidal category
with equalizers. Assume that the tensor functors are additive and preserve
equalizers. Assume further that $\mathcal{M}$ has denumerable coproducts and
that the tensor functors preserve such coproducts. Then we can consider the
morphisms $\overline{\eta }_{\mathrm{Br}}$ and $\overline{\epsilon }_{%
\mathrm{Br}}$ of Theorem \ref{teo:TbarStrict} and the functor $\overline{T}$
of (\ref{claim:Tbar}).

1) There are unique natural transformations $\overline{\eta }:\mathrm{Id}_{%
\mathcal{M}}\rightarrow P\overline{T}$ and $\overline{\epsilon }\mathbb{B}:%
\overline{T}P\rightarrow \mathrm{Id}_{\mathrm{Bialg}_{\mathcal{M}}}$ such
that%
\begin{eqnarray}
\overline{\eta }_{\mathrm{Br}}J &=&J\overline{\eta },
\label{form:etabarbrVSetabar} \\
\overline{\epsilon }_{\mathrm{Br}}J_{\mathrm{Bialg}} &=&J_{\mathrm{Bialg}}%
\overline{\epsilon }.  \label{form:epsbarbrVSepsbar}
\end{eqnarray}

2)\ The pair $\left( \overline{T},P\right) $ is an adjunction with unit $%
\overline{\eta }$ and counit $\overline{\epsilon }.$

3) The unit $\overline{\eta }$ and the counit $\overline{\epsilon }$ are
uniquely determined by the following equalities%
\begin{equation}
\xi \overline{T}\circ \overline{\eta }=\eta ,  \label{form:etabarVSeta}
\end{equation}%
\begin{equation}
\epsilon \mho \circ T\xi =\mho \overline{\epsilon },
\label{form:epsbarVSeps}
\end{equation}%
where $\eta $ and $\epsilon $ denote the unit and counit of the adjunction $%
\left( T,\Omega \right) $ respectively.
\end{theorem}

\begin{proof}
1) Let $V\in \mathcal{M}$ and $\mathbb{B}\in \mathrm{Bialg}_{\mathcal{M}}$.
Since $P_{\mathrm{Br}}\overline{T}_{\mathrm{Br}}JV\overset{(\ref{form:JTbar})%
}{=}P_{\mathrm{Br}}J_{\mathrm{Bialg}}\overline{T}V\overset{(\ref{form:JPbar})%
}{=}JP\overline{T}V$, we have that $\overline{\eta }_{\mathrm{Br}}JV\in
\mathrm{Br}_{\mathcal{M}}\left( JV,P_{\mathrm{Br}}\overline{T}_{\mathrm{Br}%
}JV\right) =\mathrm{Br}_{\mathcal{M}}\left( JV,JP\overline{T}V\right) .$
Since $\overline{T}_{\mathrm{Br}}P_{\mathrm{Br}}J_{\mathrm{Bialg}}\mathbb{B}%
\overset{(\ref{form:JPbar})}{=}\overline{T}_{\mathrm{Br}}JP\mathbb{B}\overset%
{(\ref{form:JTbar})}{=}J_{\mathrm{Bialg}}\overline{T}P\mathbb{B},$ we have
that
\begin{equation*}
\overline{\epsilon }_{\mathrm{Br}}J_{\mathrm{Bialg}}\mathbb{B}\in \mathrm{%
BrBialg}_{\mathcal{M}}\left( \overline{T}_{\mathrm{Br}}P_{\mathrm{Br}}J_{%
\mathrm{Bialg}}\mathbb{B},J_{\mathrm{Bialg}}\mathbb{B}\right) = \mathrm{%
BrBialg}_{\mathcal{M}}\left( J_{\mathrm{Bialg}}\overline{T}P\mathbb{B},J_{%
\mathrm{Bialg}}\mathbb{B}\right) .
\end{equation*}
Now, by Proposition \ref{coro:BrBialg}, both $J$ and $J_{\mathrm{Bialg}}$
are full and faithful. Thus there are unique morphisms $\overline{\eta }%
V:V\rightarrow P\overline{T}V$ and $\overline{\epsilon }\mathbb{B}:\overline{%
T}P\mathbb{B}\rightarrow \mathbb{B}$ such that $\overline{\eta }_{\mathrm{Br}%
}JV=J\overline{\eta }V$ and $\overline{\epsilon }_{\mathrm{Br}}J_{\mathrm{%
Bialg}}\mathbb{B}=J_{\mathrm{Bialg}}\overline{\epsilon }\mathbb{B}.$ Note
that $\overline{\eta }V=HJ\overline{\eta }V=H\overline{\eta }_{\mathrm{Br}}JV
$ so that $\overline{\eta }V$ is natural in $V.$ Let us check that $%
\overline{\epsilon }\mathbb{B} $ is natural in $\mathbb{B}$. Given a
morphism $f:\mathbb{B}\rightarrow \mathbb{B}^{\prime }$ we have
\begin{eqnarray*}
J_{\mathrm{Bialg}}\left( \overline{\epsilon }\mathbb{B}^{\prime }\circ
\overline{T}Pf\right) &=&J_{\mathrm{Bialg}}\overline{\epsilon }\mathbb{B}%
^{\prime }\circ J_{\mathrm{Bialg}}\overline{T}Pf\overset{(\ref{form:JTbar})}{%
=}\overline{\epsilon }_{\mathrm{Br}}J_{\mathrm{Bialg}}\mathbb{B}^{\prime
}\circ \overline{T}_{\mathrm{Br}}JPf \\
&&\overset{(\ref{form:JPbar})}{=}\overline{\epsilon }_{\mathrm{Br}}J_{%
\mathrm{Bialg}}\mathbb{B}^{\prime }\circ \overline{T}_{\mathrm{Br}}P_{%
\mathrm{Br}}J_{\mathrm{Bialg}}f \\
&=&J_{\mathrm{Bialg}}f\circ \overline{\epsilon }_{\mathrm{Br}}J_{\mathrm{%
Bialg}}\mathbb{B}^{\prime }=J_{\mathrm{Bialg}}f\circ J_{\mathrm{Bialg}}%
\overline{\epsilon }\mathbb{B}^{\prime }=J_{\mathrm{Bialg}}\left( f\circ
\overline{\epsilon }\mathbb{B}\right) .
\end{eqnarray*}%
Since $J_{\mathrm{Bialg}}$ is faithful, we get $\overline{\epsilon }\mathbb{B%
}^{\prime }\circ \overline{T}Pf=f\circ \overline{\epsilon }\mathbb{B}$ so
that $\overline{\epsilon }\mathbb{B}$ is natural in $\mathbb{B}$.

2) We compute%
\begin{gather*}
J\left( P\overline{\epsilon }\circ \overline{\eta }P\right) \overset{(\ref%
{form:JPbar})}{=}P_{\mathrm{Br}}J_{\mathrm{Bialg}}\overline{\epsilon }\circ J%
\overline{\eta }P\overset{(\ref{form:etabarbrVSetabar}),(\ref%
{form:epsbarbrVSepsbar})}{=}P_{\mathrm{Br}}\overline{\epsilon }_{\mathrm{Br}%
}J_{\mathrm{Bialg}}\circ \overline{\eta }_{\mathrm{Br}}JP \\
\overset{(\ref{form:JPbar})}{=}P_{\mathrm{Br}}\overline{\epsilon }_{\mathrm{%
Br}}J_{\mathrm{Bialg}}\circ \overline{\eta }_{\mathrm{Br}}P_{\mathrm{Br}}J_{%
\mathrm{Bialg}}=P_{\mathrm{Br}}J_{\mathrm{Bialg}}\overset{(\ref{form:JPbar})}%
{=}JP.
\end{gather*}%
Since $J$ is faithful, we obtain $P\overline{\epsilon }\circ \overline{\eta }%
P=P.$ We also have%
\begin{gather*}
J_{\mathrm{Bialg}}\left( \overline{\epsilon }\overline{T}\circ \overline{T}%
\overline{\eta }\right) \overset{(\ref{form:JTbar})}{=}J_{\mathrm{Bialg}}%
\overline{\epsilon }\overline{T}\circ \overline{T}_{\mathrm{Br}}J\overline{%
\eta }\overset{(\ref{form:etabarbrVSetabar}),(\ref{form:epsbarbrVSepsbar})}{=%
}\overline{\epsilon }_{\mathrm{Br}}J_{\mathrm{Bialg}}\overline{T}\circ
\overline{T}_{\mathrm{Br}}\overline{\eta }_{\mathrm{Br}}J \\
\overset{(\ref{form:JTbar})}{=}\overline{\epsilon }_{\mathrm{Br}}\overline{T}%
_{\mathrm{Br}}J\circ \overline{T}_{\mathrm{Br}}\overline{\eta }_{\mathrm{Br}%
}J=\overline{T}_{\mathrm{Br}}J\overset{(\ref{form:JTbar})}{=}J_{\mathrm{Bialg%
}}\overline{T}.
\end{gather*}%
Since $J_{\mathrm{Bialg}}$ is faithful, we get $\overline{\epsilon }%
\overline{T}\circ \overline{T}\overline{\eta }=\overline{T}$. We have so
proved that $\left( \overline{T},P\right) $ is an adjunction with unit $%
\overline{\eta }$ and counit $\overline{\epsilon }.$

3) We have%
\begin{gather*}
\xi \overline{T}\circ \overline{\eta }=HJ\left( \xi \overline{T}\circ
\overline{\eta }\right) =H\left( J\xi \overline{T}\circ J\overline{\eta }%
\right) \overset{(\ref{form:xij}),(\ref{form:etabarbrVSetabar})}{=}H\left(
\xi J_{\mathrm{Bialg}}\overline{T}\circ \overline{\eta }_{\mathrm{Br}%
}J\right) \\
\overset{(\ref{form:JTbar})}{=}H\left( \xi \overline{T}_{\mathrm{Br}}J\circ
\overline{\eta }_{\mathrm{Br}}J\right) \overset{(\ref{form:BarEta})}{=}H\eta
_{\mathrm{Br}}J\overset{(\ref{form:TbrStrict})}{=}\eta HJ=\eta
\end{gather*}%
and%
\begin{gather*}
\mho \overline{\epsilon }=H_{\mathrm{Alg}}\mho _{\mathrm{Br}}J_{\mathrm{Bialg%
}}\overline{\epsilon }\overset{(\ref{form:epsbarbrVSepsbar})}{=}H_{\mathrm{%
Alg}}\mho _{\mathrm{Br}}\overline{\epsilon }_{\mathrm{Br}}J_{\mathrm{Bialg}}%
\overset{(\ref{form:BarEps})}{=}H_{\mathrm{Alg}}\left( \epsilon _{\mathrm{Br}%
}\mho _{\mathrm{Br}}\circ T_{\mathrm{Br}}\xi \right) J_{\mathrm{Bialg}} \\
=H_{\mathrm{Alg}}\epsilon _{\mathrm{Br}}\mho _{\mathrm{Br}}J_{\mathrm{Bialg}%
}\circ H_{\mathrm{Alg}}T_{\mathrm{Br}}\xi J_{\mathrm{Bialg}}\overset{(\ref%
{form:TbrStrict}),(\ref{form:HtildeTbrOmegaBr})}{=}\epsilon H_{\mathrm{Alg}%
}\mho _{\mathrm{Br}}J_{\mathrm{Bialg}}\circ TH\xi J_{\mathrm{Bialg}} \\
\overset{(\ref{form:xij})}{=}\epsilon \mho \circ THJ\xi =\epsilon \mho \circ
T\xi .
\end{gather*}%
Since $\xi \overline{T}$ is a monomorphism and $\mho $ is faithful, we get
that $\overline{\eta }$ and $\overline{\epsilon }$ are uniquely determined
by (\ref{form:etabarVSeta}) and (\ref{form:epsbarVSeps}) respectively.
\end{proof}

\begin{proposition}
\label{pro:BialgF}Let $\mathcal{M}$ and $\mathcal{M}^{\prime }$ be braided
monoidal categories. Let $\left( F,\phi _{0},\phi _{2}\right) :\mathcal{M}%
\rightarrow \mathcal{M}^{\prime }$ be a braided monoidal functor. Then $F$
induces a functor $\mathrm{Bialg}F:\mathrm{Bialg}_{\mathcal{M}}\rightarrow
\mathrm{Bialg}_{\mathcal{M}^{\prime }}$ which acts as $F$ on morphisms and
which is defined, on objects, by%
\begin{equation*}
\mathrm{Bialg}F\left( B,m_{B},u_{B},\Delta _{B},\varepsilon _{B}\right)
:=\left( FB,m_{FB},u_{FB},\Delta _{FB},\varepsilon _{FB}\right)
\end{equation*}

where%
\begin{eqnarray*}
m_{FB} &:&=Fm_{B}\circ \phi _{2}\left( B,B\right) :FB\otimes FB\rightarrow
FB, \\
u_{FB} &:&=Fu_{B}\circ \phi _{0}:\mathbf{1}\rightarrow FB, \\
\Delta _{FB} &:&=\phi _{2}^{-1}\left( B,B\right) \circ F\Delta
_{B}:FB\rightarrow FB\otimes FB, \\
\varepsilon _{FB} &:&=\phi _{0}^{-1}\circ F\varepsilon _{B}:FB\rightarrow
\mathbf{1},
\end{eqnarray*}%
and the following diagrams commute.%
\begin{equation*}
\xymatrixrowsep{15pt} \xymatrixcolsep{35pt}
\xymatrix{\mathcal{M}\ar[r]^{F}\ar[d]_{J}
&\mathcal{M}^\prime \ar[d]^{J^\prime}\\
{\mathrm{Br}_\mathcal{M}} \ar[r]^{\mathrm{Br}F}&\mathrm{Br}_{\mathcal{M}^\prime}}
\qquad
\xymatrixrowsep{15pt} \xymatrixcolsep{35pt}
\xymatrix{\mathrm{Bialg}_\mathcal{M}\ar[r]^{\mathrm{Bialg}F}\ar[d]_{J_\mathrm{Bialg}}
&\mathrm{Bialg}_{\mathcal{M}^\prime} \ar[d]^{J^\prime_\mathrm{Bialg}}\\
{\mathrm{BrBialg}_\mathcal{M}} \ar[r]^{\mathrm{BrBialg}F}&\mathrm{BrBialg}_{\mathcal{M}^\prime}}
\qquad
\xymatrixrowsep{15pt} \xymatrixcolsep{35pt}
\xymatrix{\mathrm{Bialg}_\mathcal{M}\ar[r]^{\mathrm{Bialg}F}\ar[d]_{\mho}
&\mathrm{Bialg}_{\mathcal{M}^\prime} \ar[d]^{\mho^\prime}\\
{\mathrm{Alg}_\mathcal{M}} \ar[r]^{\mathrm{Alg}F}&\mathrm{Alg}_{\mathcal{M}^\prime}}
\end{equation*}

\begin{enumerate}
\item[1)] $\mathrm{Bialg}F$ is an equivalence (resp. category isomorphism or
conservative) whenever $F$ is.

\item[2)] If $F$ preserves equalizers, the following diagram commutes
\begin{equation*}
\xymatrixrowsep{15pt} \xymatrixcolsep{35pt}
\xymatrix{\mathrm{Bialg}_\mathcal{M}\ar[r]^{\mathrm{Bialg}F}\ar[d]_{P}
&\mathrm{Bialg}_{\mathcal{M}^\prime} \ar[d]^{P^\prime}\\
\mathcal{M}\ar[r]^{F}&\mathcal{M}^\prime}
\end{equation*}%
and
\begin{equation}
\xi \left( \mathrm{Bialg}F\right) =F\xi  \label{form:xiBialgF}
\end{equation}
\end{enumerate}
\end{proposition}

\begin{proof}
Let us check that the first diagram commutes.%
\begin{eqnarray*}
\left( \mathrm{Br}F\circ J\right) \left( M\right) &=&\mathrm{Br}F\left(
M,c_{M,M}\right) =\left( FM,\phi _{2}^{-1}\left( M,M\right) \circ
Fc_{M,M}\circ \phi _{2}\left( M,M\right) \right) \\
\overset{(\ast )}{=}\left( FM,c_{FM,FM}\right) &=&\left( J^{\prime }\circ
F\right) \left( M\right) ,
\end{eqnarray*}%
where in (*) we used that $F$ is braided. The functors $\mathrm{Br}F\circ J$
and $J^{\prime }\circ F$ trivially coincide on morphisms. We have%
\begin{eqnarray*}
&&\left( \mathrm{BrBialg}F\circ J_{\mathrm{Bialg}}\right) \left(
B,m_{B},u_{B},\Delta _{B},\varepsilon _{B}\right) \\&=&\mathrm{BrBialg}F\left(
B,m_{B},u_{B},\Delta _{B},\varepsilon _{B},c_{B,B}\right) \\
&=&\left( FB,m_{FB},u_{FB},\Delta _{FB},\varepsilon _{FB},\phi
_{2}^{-1}\left( B,B\right) \circ Fc_{B,B}\circ \phi _{2}\left( B,B\right)
\right) \\
&\overset{(\ast )}{=}&\left( FB,m_{FB},u_{FB},\Delta _{FB},\varepsilon
_{FB},c_{FB,FB}\right)
\end{eqnarray*}%
Now, since $$\left( FB,m_{FB},u_{FB},\Delta _{FB},\varepsilon
_{FB},c_{FB,FB}\right) =\left( \mathrm{BrBialg}F\circ J_{\mathrm{Bialg}%
}\right) \left( B,m_{B},u_{B},\Delta _{B},\varepsilon _{B}\right) \in
\mathrm{BrBialg}_{\mathcal{M}^{\prime }}$$ and $c_{FB,FB}$ is the braiding of
$FB$ in $\mathcal{M}^{\prime }$ we conclude that $\left(
FB,m_{FB},u_{FB},\Delta _{FB},\varepsilon _{FB}\right) \in \mathrm{Bialg}_{%
\mathcal{M}^{\prime }}$. Moreover for every morphism $f:\left(
B,m_{B},u_{B},\Delta _{B},\varepsilon _{B}\right) \rightarrow \left(
B^{\prime },m_{B^{\prime }},u_{B^{\prime }},\Delta _{B^{\prime
}},\varepsilon _{B^{\prime }}\right) $ in $\mathrm{Bialg}_{\mathcal{M}},$ we
have $\left( \mathrm{BrBialg}F\circ J_{\mathrm{Bialg}}\right) \left(
f\right) =\left( \mathrm{BrBialg}F\right) \left( f\right) =Ff$ so that $Ff$
is a morphism with domain $\left( FB,m_{FB},u_{FB},\Delta _{FB},\varepsilon
_{FB},c_{FB,FB}\right) $ and codomain $\left( FB^{\prime },m_{FB^{\prime
}},u_{FB^{\prime }},\Delta _{FB^{\prime }},\varepsilon _{FB^{\prime
}},c_{FB^{\prime },FB^{\prime }}\right) .$ Thus $Ff$ is a morphism in $%
\mathrm{Bialg}_{\mathcal{M}^{\prime }}$. Hence $\mathrm{Bialg}F$ is a
well-defined functor. Note also that 
\begin{gather*}
\left( J_{\mathrm{Bialg}}^{\prime }\circ \mathrm{Bialg}F\right) \left(
B,m_{B},u_{B},\Delta _{B},\varepsilon _{B}\right) =J_{\mathrm{Bialg}%
}^{\prime }\left( FB,m_{FB},u_{FB},\Delta _{FB},\varepsilon _{FB}\right)\\        
=\left( FB,m_{FB},u_{FB},\Delta _{FB},\varepsilon _{FB},c_{FB,FB}\right)
\end{gather*}%
so that the functors $J_{\mathrm{Bialg}}^{\prime }\circ \mathrm{Bialg}F$
and $\mathrm{BrBialg}F\circ J_{\mathrm{Bialg}}$ coincide on objects. They
trivially coincide on morphisms too so that the second diagram commutes. The
third diagram commutes by definition of $\mathrm{Bialg}F$ and $\mathrm{Alg}F$%
.

1) Assume that $F$ preserves equalizers. By Proposition \ref{pro:PrimFunct},
we have
\begin{eqnarray*}
P^{\prime }\left( \mathrm{Bialg}F\right) &=&H^{\prime }P_{\mathrm{Br}%
}^{\prime }J_{\mathrm{Bialg}}^{\prime }\left( \mathrm{Bialg}F\right)
=H^{\prime }P_{\mathrm{Br}}^{\prime }\left( \mathrm{BrBialg}F\right) J_{%
\mathrm{Bialg}}\overset{(\ref{diag:BrF-PBr})}{=}H^{\prime }\left( \mathrm{Br}%
F\right) P_{\mathrm{Br}}J_{\mathrm{Bialg}} \\
\overset{(\ref{form:JPbar})}{=}FHJP &=&FP.
\end{eqnarray*}%
and%
\begin{equation*}
\xi \left( \mathrm{Bialg}F\right) =\xi ^{\prime }J_{\mathrm{Bialg}}^{\prime
}\left( \mathrm{Bialg}F\right) =\xi ^{\prime }\left( \mathrm{BrBialg}%
F\right) J_{\mathrm{Bialg}}\overset{(\ref{form:comdat3})}{=}\left( \mathrm{Br%
}F\right) \xi J_{\mathrm{Bialg}}=\left( \mathrm{Br}F\right) \xi =F\xi .
\end{equation*}

2) Assume that $F$ is an equivalence. By Proposition \ref{coro:BrBialg}, $J_{%
\mathrm{Bialg}}$ and $J_{\mathrm{Bialg}}^{\prime }$ are both full and
faithful. By Proposition \ref{pro:BrBialg}, $\mathrm{BrBialg}F$ is a
category equivalence. Given $X$ and $Y$ objects in $\mathrm{Bialg}_{\mathcal{%
M}}\,$we have
\begin{eqnarray*}
&&\mathrm{Bialg}_{\mathcal{M}^{\prime }}\left( \left( \mathrm{Bialg}F\right)
X,\left( \mathrm{Bialg}F\right) Y\right) \cong \mathrm{BrBialg}_{\mathcal{M}%
^{\prime }}\left( \left( J_{\mathrm{Bialg}}^{\prime }\circ \mathrm{Bialg}%
F\right) X,\left( J_{\mathrm{Bialg}}^{\prime }\circ \mathrm{Bialg}F\right)
Y\right) \\
&=&\mathrm{BrBialg}_{\mathcal{M}^{\prime }}\left( \left( \mathrm{BrBialg}%
F\circ J_{\mathrm{Bialg}}\right) X,\left( \mathrm{BrBialg}F\circ J_{\mathrm{%
Bialg}}\right) Y\right) \cong \mathrm{Bialg}_{\mathcal{M}}\left( X,Y\right) .
\end{eqnarray*}%
The composition of these maps is the map assigning $\left( \mathrm{Bialg}%
F\right) \left( f\right) $ to a morphism $f$ so that $\mathrm{Bialg}F$ is
full and faithful. In order to prove it is an equivalence, it remains to
check that it is essentially surjective i.e. that each object $\left(
B^{\prime },m_{B^{\prime }},u_{B^{\prime }},\Delta _{B^{\prime
}},\varepsilon _{B^{\prime }}\right) \in \mathrm{Bialg}_{\mathcal{M}^{\prime
}}$ is isomorphic to $\left( \mathrm{Bialg}F\right) \left(
B,m_{B},u_{B},\Delta _{B},\varepsilon _{B}\right) $ for some object $\left(
B,m_{B},u_{B},\Delta _{B},\varepsilon _{B}\right) $ in $\mathrm{Bialg}_{%
\mathcal{M}}$.

Let $\left( B^{\prime },m_{B^{\prime }},u_{B^{\prime }},\Delta _{B^{\prime
}},\varepsilon _{B^{\prime }}\right) \in \mathrm{Bialg}_{\mathcal{M}^{\prime
}}$. Then
\begin{equation*}
\left( B^{\prime },m_{B^{\prime }},u_{B^{\prime }},\Delta _{B^{\prime
}},\varepsilon _{B^{\prime }},c_{B^{\prime },B^{\prime }}\right) =J_{\mathrm{%
Bialg}}^{\prime }\left( B^{\prime },m_{B^{\prime }},u_{B^{\prime }},\Delta
_{B^{\prime }},\varepsilon _{B^{\prime }}\right) \in \mathrm{BrBialg}_{%
\mathcal{M}^{\prime }}
\end{equation*}%
Since $\mathrm{BrBialg}F$ is essentially surjective, there exists $\left(
B,m_{B},u_{B},\Delta _{B},\varepsilon _{B},c_{B}\right) \in \mathrm{BrBialg}%
_{\mathcal{M}}$ and an isomorphism
\begin{equation*}
f:\left( B^{\prime },m_{B^{\prime }},u_{B^{\prime }},\Delta _{B^{\prime
}},\varepsilon _{B^{\prime }},c_{B^{\prime },B^{\prime }}\right) \rightarrow
\left( \mathrm{BrBialg}F\right) \left( B,m_{B},u_{B},\Delta _{B},\varepsilon
_{B},c_{B}\right)
\end{equation*}%
in $\mathrm{BrBialg}_{\mathcal{M}^{\prime }}$. Since
\begin{equation*}
\left( \mathrm{BrBialg}F\right) \left( B,m_{B},u_{B},\Delta _{B},\varepsilon
_{B},c_{B}\right) =\left( FB,m_{FB},u_{FB},\Delta _{FB},\varepsilon
_{FB},\phi _{2}^{-1}\left( B,B\right) \circ Fc_{B}\circ \phi _{2}\left(
B,B\right) \right) ,
\end{equation*}
we have
\begin{equation*}
\phi _{2}^{-1}\left( B,B\right) \circ Fc_{B}\circ \phi _{2}\left( B,B\right)
\circ \left( f\otimes f\right) =\left( f\otimes f\right) \circ c_{B^{\prime
},B^{\prime }}.
\end{equation*}%
Since $f$ is ,in particular, a morphism in $\mathcal{M}^{\prime }$, by the
naturality of the braiding, we get%
\begin{eqnarray*}
Fc_{B} &=&\phi _{2}\left( B,B\right) \circ \left( f\otimes f\right) \circ
c_{B^{\prime },B^{\prime }}\circ \left( f^{-1}\otimes f^{-1}\right) \circ
\phi _{2}^{-1}\left( B,B\right) \\
&=&\phi _{2}\left( B,B\right) \circ c_{FB,FB}^{\prime }\circ \left( f\otimes
f\right) \circ \left( f^{-1}\otimes f^{-1}\right) \circ \phi _{2}^{-1}\left(
B,B\right) \\
&=&\phi _{2}\left( B,B\right) \circ c_{FB,FB}^{\prime }\circ \phi
_{2}^{-1}\left( B,B\right) =Fc_{B,B}.
\end{eqnarray*}%
Since $F$ is faithful, we obtain $c_{B}=c_{B,B}$. Thus $\left(
B,m_{B},u_{B},\Delta _{B},\varepsilon _{B},c_{B}\right) \in \mathrm{BrBialg}%
_{\mathcal{M}}$ means that $\left( B,m_{B},u_{B},\Delta _{B},\varepsilon
_{B}\right) \in \mathrm{Bialg}_{\mathcal{M}}$ so that%
\begin{gather*}
\left( \mathrm{BrBialg}F\right) \left( B,m_{B},u_{B},\Delta _{B},\varepsilon
_{B},c_{B}\right) =\left( \mathrm{BrBialg}F\right) \left(
B,m_{B},u_{B},\Delta _{B},\varepsilon _{B},c_{B,B}\right) \\
=\left( \mathrm{BrBialg}F\right) J_{\mathrm{Bialg}}\left(
B,m_{B},u_{B},\Delta _{B},\varepsilon _{B}\right) =\left( J_{\mathrm{Bialg}%
}^{\prime }\circ \mathrm{Bialg}F\right) \left( B,m_{B},u_{B},\Delta
_{B},\varepsilon _{B}\right) .
\end{gather*}%
Thus $f\in \mathrm{BrBialg}_{\mathcal{M}^{\prime }}\left( J_{\mathrm{Bialg}%
}^{\prime }\left( B^{\prime },m_{B^{\prime }},u_{B^{\prime }},\Delta
_{B^{\prime }},\varepsilon _{B^{\prime }}\right) ,\left( J_{\mathrm{Bialg}%
}^{\prime }\circ \mathrm{Bialg}F\right) \left( B,m_{B},u_{B},\Delta
_{B},\varepsilon _{B}\right) \right) .$ Since $J_{\mathrm{Bialg}}^{\prime }$
is full, there is a morphism $g:\left( B^{\prime },m_{B^{\prime
}},u_{B^{\prime }},\Delta _{B^{\prime }},\varepsilon _{B^{\prime }}\right)
\rightarrow \left( \mathrm{Bialg}F\right) \left( B,m_{B},u_{B},\Delta
_{B},\varepsilon _{B}\right) $ in such $\mathrm{Bialg}_{\mathcal{M}^{\prime
}}$ that $f=J_{\mathrm{Bialg}}^{\prime }\left( g\right) .$ Since $J_{\mathrm{%
Bialg}}^{\prime }$ is full and faithful, we get that $g$ is an isomorphism
too. Therefore $\mathrm{Bialg}F$ is essentially surjective and hence an
equivalence.

Assume that $F$ is a category isomorphism. By Proposition \ref{pro:BrBialg},
$\mathrm{BrBialg}F$ is a category isomorphism. Indeed the inverse is, by
construction $\mathrm{BrBialg}G$ where $G$ is the inverse of $F$. We have
\begin{equation*}
J_{\mathrm{Bialg}}^{\prime }\circ \mathrm{Bialg}F\circ \mathrm{Bialg}G=%
\mathrm{BrBialg}F\circ J_{\mathrm{Bialg}}\circ \mathrm{Bialg}G=\mathrm{%
BrBialg}F\circ \mathrm{BrBialg}G\circ J_{\mathrm{Bialg}}^{\prime }=J_{%
\mathrm{Bialg}}^{\prime }
\end{equation*}%
and hence $\mathrm{Bialg}F\circ \mathrm{Bialg}G=\mathrm{Id}_{\mathrm{Bialg}_{%
\mathcal{M}^{\prime }}}$ (as $J_{\mathrm{Bialg}}^{\prime }$ is faithful and
trivially injective on objects). Similarly $\mathrm{Bialg}G\circ \mathrm{%
Bialg}F=\mathrm{Id}_{\mathrm{Bialg}_{\mathcal{M}}}$. Hence $\mathrm{Bialg}F$
is a category isomorphism.

If $F$ is conservative, then, by Proposition \ref{pro:BrBialg} and
Proposition \ref{coro:BrBialg}, we have that $\mathrm{BrBialg}F\circ J_{%
\mathrm{Bialg}}$ is conservative. Hence $J_{\mathrm{Bialg}}^{\prime }\circ
\mathrm{Bialg}F$ is conservative. From this we get that $\mathrm{Bialg}F$ is
conservative.
\end{proof}

\end{document}